\documentclass[11pt]{amsart}
\usepackage{amsmath} 
\usepackage{amsthm}
\usepackage{bm}
\usepackage{amssymb} 
\usepackage{amsfonts}
\usepackage{graphicx}
\usepackage{color}
\usepackage[normalem]{ulem}
\usepackage{comment}
\usepackage{mathtools}
\usepackage{xcolor}
\usepackage{hyperref}
\usepackage{cleveref}
\usepackage{enumitem}
\usepackage{caption,subcaption}
\usepackage{pinlabel}
\newtheorem{theorem}{Theorem}[section]
\newtheorem*{thmb4}{Theorem~\ref{thm:b4}}
\newtheorem*{unique}{Theorem~\ref{uniquenesscor}}
\newtheorem*{l0not1}{Lemma~\ref{L:0not1}}
\newtheorem*{thmbdrysmall}{Theorem~\ref{thm:boundarysmall}}
\newtheorem*{arbbig}{Corollary~\ref{cor:arbbig}}
\newtheorem*{murasugi}{Theorem~\ref{thm:murasugi}}
\newtheorem{lemma}[theorem]{Lemma}
\newtheorem{proposition}[theorem]{Proposition}
\newtheorem{question}[theorem]{Question}

\newtheorem{corollary}[theorem]{Corollary}
\newtheorem{claim}[theorem]{Claim}
\theoremstyle{definition}
\newtheorem{definition}[theorem]{Definition}
\newtheorem{remark}[theorem]{Remark}
\newtheorem{example}[theorem]{Example}
\newtheorem{conjecture}{Conjecture}
\newtheorem*{algorithm}{Algorithm}

\definecolor{ForestGreen}{rgb}{0, 0.686, 0}

\newcommand{\mblue}[1]{\marginpar{\textcolor{blue}{\tiny{#1}}}}
\newcommand{\blue}[1]{\textcolor{blue}{#1}}
\newcommand{\red}[1]{\textcolor{red}{#1}}

\newcommand{\Hom}{\text{Hom}}
\newcommand{\Z}{\mathbb{Z}}
\newcommand{\Q}{\mathbb{Q}}
\newcommand{\N}{\mathbb{N}}

\newcommand{\CP}{\mathbb{C}P}
\newcommand{\T}{\mathcal{T}}
\newcommand{\tri}{\mathcal{T}}
\newcommand{\M}{\mathcal{M}}
\let\O\relax
\newcommand{\O}{\mathcal{O}}
\let\int\relax
\newcommand{\int}{\mathring}

\newcommand{\boundary}{\partial}
\newcommand{\into}{\hookrightarrow}
\newcommand{\C}{\mathcal{C}}
\newcommand{\D}{\mathcal{D}}
\newcommand{\ob}{\mathcal{OB}}

\newcommand{\rL}{r\mathcal{L}}
\let\L\relax
\newcommand{\L}{\mathcal{L}}
\newcommand{\id}{\text{id}}

    \author[Castro]{Nickolas A. Castro}
    \address{University of Arkansas\\Fayetteville, AR 72704}
    \email{nacastro@uark.edu}
    
    \author[Islambouli]{Gabriel Islambouli}
    \address{University of Waterloo, Waterloo, ON, Canada  N2L 3G1}
    \email{gislambo@uwaterloo.ca}
    
    \author[Miller]{Maggie Miller}
    \address{Massachusetts Institute of Technology\\Cambridge, MA 02139, USA}
    \email{maggiehm@mit.edu}
    
    \author[Tomova]{Maggy Tomova}
    \address{University of Iowa\\Iowa City, IA 52240, USA}
    \email{maggy-tomova@uiowa.edu}
    
    \title[The relative $\L$-invariant of a compact $4$-manifold]{The relative $\L$-invariant of a compact $4$-manifold}
    
    \thanks{The first author was partially supported by the University of California,  Davis Chancellor's Postdoctoral Fellowship Program. The third author is supported by NSF Grant No. DGE-1656466. The forth author is partially supported by NSF Grant No. 1664583.}

\begin{document}  




\begin{abstract}
In this paper, we introduce the {\emph{relative $\L$-invariant}} $\rL(X)$ of a smooth, orientable, compact 4-manifold $X$ with boundary. This invariant is defined by measuring the lengths of certain paths in the cut complex of a trisection surface for $X$. This is motivated by the definition of the $\L$-invariant for smooth, orientable, closed 4-manifolds by Kirby and Thompson. We show that if $X$ is a rational homology ball, then $\rL(X)=0$ if and only if $X\cong B^4$. 

In order to better understand relative trisections, we also produce an algorithm to glue two relatively trisected 4-manifold by any Murasugi sum or plumbing in the boundary, and also prove that any two relative trisections of a given 4-manifold $X$ are related by interior stabilization, relative stabilization, and the relative double twist, which we introduce in this paper as a trisection version of one of Piergallini and Zuddas's moves on open book decompositions. Previously, it was only known (by Gay and Kirby) that relative trisections inducing equivalent open books on $X$ are related by interior stabilizations.
\end{abstract}



\maketitle

\section{Introduction}

In this paper, we introduce the {\emph{relative $\L$-invariant}}, a trisection-theoretic invariant of a compact $4$-manifold $X$ with boundary which we denote $\rL(X)$. This invariant is modeled after the $\L$-invariant $\L(Y)$ of Kirby and Thompson~\cite{linvariant} defined for a closed $4$-manifold $Y$. We review the details of the $\L$-invariant in Section~\ref{sec:closedL}.

This invariant has the following interesting property.

\begin{thmb4}
If $X$ is a rational homology ball with $\rL(X)=0$, then $X\cong B^4$.
\end{thmb4}

This mirrors the situation in the closed case, as Kirby and Thompson~\cite{linvariant} showed that for $X$ a rational homology sphere, $\L(X)=0$ if and only if $X\cong S^4$.

Roughly, $\rL$ measures the minimal complexity of a relative trisection diagram of $X$. By minimizing this complexity over all relative trisection diagrams of $X$, we obtain a manifold invariant. We also define two similar invariants $\rL^\boundary(X)$ and $\rL^\circ(X)$ which minimize (over relative trisection diagrams) complexities associated to the boundary and interior of $X$, respectively. Intuitively, one should think of $\rL^\partial$ and $\rL^\circ$ as being the relative $\L$-invariant restricted to the boundary or interior of a manifold, respectively. We make this precise in Section~\ref{sec2:relLdef}. We review the construction of relative trisections as introduced by Gay and Kirby~\cite{gaykirby} in Section~\ref{sec:reltridef}.

We show also that when $\rL^\boundary(X)$ is small, then the boundary of $X$ has simple topology.


\begin{l0not1}
If $\T$ is a $(g,k;p,b)$-relative trisection diagram of a 4-manifold $X$ with $\rL^\partial(T)\le 1$, then $\partial X\cong \#_{2p+b-1} S^1\times S^2$.
\end{l0not1}
\begin{thmbdrysmall}
Let $\T$ be a $(g,k;p,b)$-relative trisection of 4-manifold $X$ with $\rL^\boundary(\T)<2(2p+b-1)$. Then $\partial X$ admits an $S^1\times S^2$ summand.
\end{thmbdrysmall}


On the other end of the spectrum, we show that $\rL(X)$ can be large.
\begin{arbbig}
For any $n\in\mathbb{N}$, there exists a 4-manifold $X$ with $\rL(X)\ge n$.
\end{arbbig}

We also introduce a new move on relative trisection diagrams called a \emph{relative double twist}. This move achieves a Harer twist on the open book induced by the described trisection. As we already know how to achieve Hopf stabilization of the open book and interior stabilization of the described trisection, this allows us to relate any two relative trisections of a fixed $4$-manifold $X$ by a combination of relative double twists, interior stabilizations, and relative stabilizations. This strengthens the uniqueness results of Gay--Kirby~\cite{gaykirby} and the first author~\cite{nickthesis} by removing the requirement that $\partial W$ be a rational homology sphere.

\begin{unique}
Any two relative trisections $\T_1$ and $\T_2$ of a $4$-manifold with connected boundary can be made isotopic after a finite number of interior stabilizations, relative stabilizations, relative double twists, and the inverses of these moves applied to each of $\T_1$ and $\T_2$.
\end{unique}

We describe the various stabilization moves in detail in Section~\ref{sec:stab}. In Section~\ref{sec2:relLdef}, we show that while interior stabilization of a relative trisection $\T$ cannot increase $\rL(\T)$, both relative stabilization and relative double twist can increase $\rL(\T)$.

Moreover, we explicitly describe how to perform the Murasugi sum operation via relative trisection diagrams.

\begin{murasugi}
There is an explicit algorithm to glue two trisections together by Murasugi sum. That is, given relatively trisected 4-manifolds $X,X'$, we may produce a relative trisection of $X\natural X'$ where the induced open book on $\boundary (X\natural X')$ may be any Murasugi sum of the open books on $\boundary X,\boundary X'$.
\end{murasugi}
This has not previously appeared in the literature on trisections. Recall that Murasugi sum is a generalized version of plumbing; we review this definition briefly at the beginning of Section~\ref{sec:murasugi}.




\subsection*{Organization}
We break the paper into the following sections.
\begin{enumerate}
    \item[Section~\ref{sec1:preliminaries}:] We recall basic definitions regarding trisections and the $\L$-invariant for closed 4-manifolds. In this section, we describe stabilizing operations for relative trisections and prove Theorem~\ref{uniquenesscor}. 
    \item[Section~\ref{sec2:relLdef}:] We define the relative $\L$-invariant, investigate its basic properties, and compare it to the $\L$-invariant. We also prove Theorem~\ref{thm:murasugi} as an item of independent interest, using the language introduced in this section.
    \item[Section~\ref{sec3:smallL}:] We study the topology of 4-manifolds with small $\L$-invariant. In particular, we prove Theorems~\ref{thm:boundarysmall} and~\ref{thm:b4}. 
    As a corollary of Theorem~\ref{thm:boundarysmall}, we conclude Corollary~\ref{cor:arbbig}. 
    \item[Section~\ref{sec4:arccomplex}:] Given a relative trisection $\T$ of $X$, we relate $\rL^\boundary(\T)$ to the displacement distance of the monodromy of the open book induced by $\T$ on $\boundary X$. We use this comparison to construct relative trisections with large $\L$ invariant but whose 3-manifold boundary has small homology.
\end{enumerate}

\subsection*{Acknowledgements}
This project began at the Spring Trisector’s Meeting at the University of Georgia in February, 2018. Thanks to Jeff Meier and Juanita Pinz\'{o}n-Caicedo for early interesting discussions, and to Rom\'an Aranda and Jesse Moeller for very helpful discussions after our first draft. Thanks also to an anonymous referee for providing many useful comments.

\section{Preliminaries}\label{sec1:preliminaries}

\subsection{Trisections}
A {\emph{trisection}} is a decomposition of a smooth, closed, orientable 4-manifold into three standard pieces. These decompositions were introduced by Gay and Kirby~\cite{gaykirby} as a 4-dimensional analogue of Heegaard splittings of 3-manifolds. Though this paper focuses on the case of manifolds with boundary, we begin with the definition for a closed 4-manifold as a warm up.

\begin{definition}
For non-negative integer $g$ and a triple of non-negative integers $k=(k_1,k_2,k_3)$ with $g\ge k_i$, a $(g,k)$-\emph{trisection} $\tri$ of a smooth, closed, orientable $4$-manifold $X$ is a decomposition of $X$ into three pieces $X_1$, $X_2$, $X_3$ so that:
\begin{itemize}
	\item[i)] $X=X_1\cup X_2\cup X_3$ and $X_i\cap X_j=\boundary X_i\cap\boundary X_j$ for $i\neq j$,
	\item[ii)] Each $X_i\cong \natural_k S^1\times B^3,$
	\item[iii)] Each $X_i\cap X_j$ ($i\neq j$) is a $3$-dimensional handlebody $\natural_g S^1\times D^2$,
	\item[iv)] The triple intersection $X_1\cap X_2\cap X_3$ is a genus $g$ surface $\Sigma$ that is properly embedded in $X$.	
\end{itemize}

We may write $\tri=(X_1,X_2,X_3).$ We may also write $(X,\tri)$ as a pair to indicate that $X$ is a $4$-manifold with associated trisection $\tri$ of $X$.
\end{definition}

Gay and Kirby~\cite{gaykirby} proved that every smooth, closed, orientable 4-manifold admits a trisection, which is unique up to a stabilization move. One nice feature of a trisection is that all of the information of the 4-manifold can be encoded by curves in the triple intersection surface. We will sketch an argument for this fact, which relies on the following theorem of Laudenbach--Poenaru~\cite{laudenbach}.

\begin{theorem}[Laudenbach--Poenaru~\cite{laudenbach}]\label{laudenbach}
Let $M\cong\natural_k S^1\times B^3$. Every self-diffeomorphism of $\boundary M$ extends to a self-diffeomorphism of $M$.
\end{theorem}

As a consequence of Theorem~\ref{laudenbach}, a trisection $\tri=(X_1,X_2,X_3)$ is determined by its {\emph{spine}} $(X_1\cap X_2, X_2\cap X_3, X_1\cap X_3)$. This holds because given the inclusion of $\boundary X_i=(X_i\cap X_j)\cup(X_i\cap X_k)$ into $X$, there is a unique way to glue in the rest of $X_i$ up to diffeomorphism. The 3-dimensional handlebody $X_i\cap X_j$ is in turn determined by $g$ curves on the surface $X_1\cap X_2\cap X_3$, and so a trisection is completely determined up to diffeomorphism by a triple of curves on a surface. This leads us to the following definition.

\begin{definition}

A {\emph{trisection diagram}} $\D=(\Sigma; \alpha, \beta, \gamma)$ consists of
\begin{itemize}
	\item[i)] $\Sigma$, a closed genus--$g$ surface,
	\item[ii)] $\alpha, \beta,$ and $\gamma$, each of which are non-separating collections of $g$ disjoint simple closed curves on $\Sigma$.
\end{itemize}
Moreover, we require that each triple $(\Sigma; \alpha, \beta), (\Sigma; \beta, \gamma), (\Sigma; \alpha, \gamma)$ be a Heegaard diagram for $\#_{k_i}(S^1\times S^2)$ for some non-negative integer $k_i$ ($i=1,2,3$, respectively).
\end{definition}

Given a trisection diagram $\D=(\Sigma; \alpha, \beta, \gamma)$, we may recover a trisected 4-manifold $(X,\tri)$ by the following process:
\begin{enumerate}[label=(\arabic*)]
    \item Start with $\Sigma \times D^2$,
    \item Attach a copy of $H\times I$ to $\Sigma\times[e^{2\pi i-\epsilon}, e^{2\pi i+\epsilon}]$. Each $\alpha$ curve in $\Sigma\times\{e^{2\pi i}\}$ should bound a disk into $H\times e^{2\pi i}$. 
    \item Similarly, attach a copy of $H\times I$ near $\Sigma\times \{e^{4\pi i}\}$ and another near $\Sigma\times \{1\}$, with the handlebodies determined by the $\beta$ and $\gamma$ curves respectively.
    \item Glue in a (uniquely determined) 4-dimensional $1$--handlebody to each of the three resulting boundary components.

\end{enumerate}
    The trisected 4-manifold $(X,\tri)$ is well-defined up to diffeomorphism. We say that $\D$ is a {\emph{trisection diagram}} of $(X,\tri)$, or just a diagram of $\tri$.

So far, this discussion has been about closed 4-manifolds. However, in this paper we are more interested in 4-manifolds with nonempty boundary.

\subsection{Relative trisections}\label{sec:reltridef}

A relative trisection is a generalization of a trisection to the case of a 4-manifold with boundary. Again we decompose a given 4-manifold into three standard pieces, though in this case the pieces meet in a slightly more intricate manner. We give the precise definition below.

\begin{definition}
For integers $g,p,b$ and a triple $k=(k_1,k_2,k_3)$ with $g\ge p\ge 0$, $b\ge 1$, and $2p+b-1\le k_i\le 2g+b-1$, a $(g,k;p,b)$-\emph{relative trisection} $\tri$ of a compact, orientable $4$-manifold $X$ with connected, nonempty boundary is a decomposition of $X$ into three pieces $X_1$, $X_2$, $X_3$ so that:
\begin{itemize}
	\item[i)] $X=X_1\cup X_2\cup X_3$ and $X_i\cap X_j=\boundary X_i\cap\boundary X_j$ for $i\neq j$,
	\item[ii)] The triple intersection $X_1\cap X_2\cap X_3$ is a genus $g$ surface $\Sigma$ with $b$ boundary components, properly embedded in $X$,	\item[iii)] Each $X_i\cong \natural_{k_i} S^1\times B^3,$
	\item[iv)] Each $X_i\cap X_j$ ($i\neq j$) is a $3$-dimensional compression body from $\Sigma$ to a genus--$p$ surface contained in $\boundary X$,
	\item[v)] There are agreeing product structures \[X_i\cap\boundary X\cong [(X_i\cap X_{i-1})\cap\boundary X]\times I\cong -[(X_i\cap X_{i+1})\cap\boundary X]\times I.\] 
\end{itemize}

We may write $\tri=(X_1,X_2,X_3)$.

The product structures on each $X_i\cap\boundary X$ induce an open book structure on $\boundary X$ with binding $\boundary \Sigma$ in which $X_i\cap X_j\cap\boundary X$ is a single page. We will write $\mathcal{O}_\tri$ to denote the open book induced on $\boundary X$ by $\tri$.

If we are abstractly given a relative trisection $\tri$, we may write $X_\tri$ to denote the trisected $4$-manifold decomposed by $\tri$. We may also write $(X,\tri)$ as a pair to indicate that $X$ is a $4$-manifold with associated relative trisection $\tri$ of $X$.
\end{definition}

The above definition can be extended to a $4$-manifold with more than one boundary component, but here we specify connected boundary for simplicity. Relative trisections were first introduced by Gay and Kirby~\cite{gaykirby} and shown to exist for all bounded, compact $4$-manifolds, even when specifying the boundary data of the induced open book.

\begin{theorem}\cite{gaykirby}
Given any open book decomposition $\O$ of $\partial X$ there is a relative trisection $\tri$ of $X$ which induces the open book $\O$.
\end{theorem}

The first author~\cite{nickthesis} showed that we may glue trisected 4-manifolds along common boundary as long as long as the relative trisections induce equivalent open books.

\begin{theorem}\cite{nickthesis}
Let $(W, \tri)$ and $(W', \tri')$ be relatively trisected $4$-manifolds such that $\partial W \cong \partial W'$. The relative trisections can be glued together along their diffeomorphic boundaries to induce a closed, trisected $4$-manifold $(W\underset{\partial}{\cup} W', \tri\cup \tri')$ if the induced open books $\O_\tri$ and $\O_{\tri'}$ are compatible (see Section~\ref{gluingsection}).
\end{theorem}

The first author, together with Gay and Pinz{\'o}n-Caicedo~\cite{cgp}, showed that relative trisections can be completely described diagramatically.

\begin{definition}

A {\emph{relative trisection diagram}} $\D=(\Sigma; \alpha, \beta, \gamma)$ consists of
\begin{itemize}
	\item[i)] $\Sigma$, a genus $g$ surface with $b$ boundary components,
	\item[ii)] $\alpha, \beta,$ and $\gamma$, each of which are a non-separating collection of $g-p$ disjoint simple closed curves on $\Sigma$.
\end{itemize}
Moreover, each triple $(\Sigma; \alpha, \beta), (\Sigma; \beta, \gamma), (\Sigma; \gamma, \alpha)$ can be made standard as in Figure~\ref{F:standard} after handleslides of the curves and diffeomorphisms of $\Sigma$.
\end{definition}

\begin{figure}
{\centering
    \labellist
		\pinlabel \rotatebox{-90}{\resizebox{8pt}{.8in}{$\}$}} at 90 5
		\pinlabel {$k-2p-b+1$} at 90 -10
		\pinlabel \rotatebox{-90}{\resizebox{8pt}{.8in}{$\}$}} at 257 5
		\pinlabel {$g+p+b-k-1$} at 257 -10
		\pinlabel \rotatebox{-90}{\resizebox{8pt}{.8in}{$\}$}} at 423 5
		\pinlabel {$p$} at 423 -10
		\pinlabel \resizebox{8pt}{1in}{$\}$} at 585 80
		\pinlabel $b$ at 600 80
	\endlabellist
	\includegraphics[scale=.6]{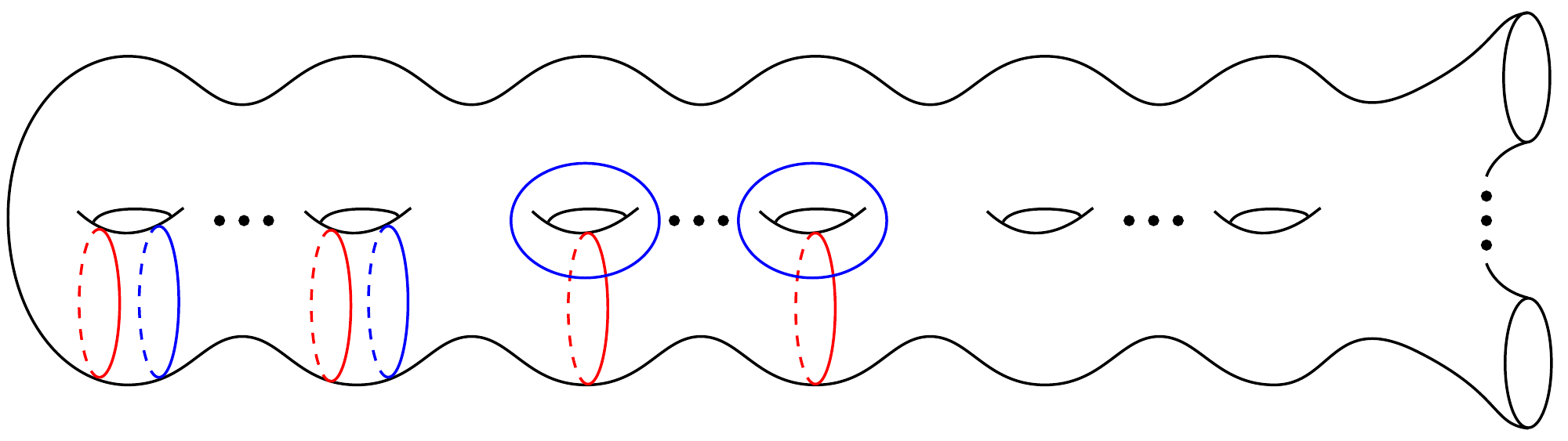}
	\caption{Standard position for a pair of collections of curves in a $(g,k;p,b)$-relative trisection diagram}
	\label{F:standard}
}
\end{figure}
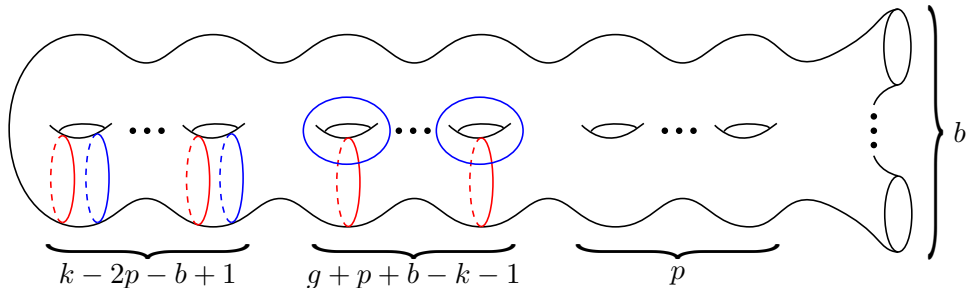

We say that a relative trisection diagram $\D$ describes or determines the relative trisection $\tri=(X_1,X_2,X_3)$ if $\tri$ has the property that under some identification $X_1\cap X_2\cap X_3$ with $\Sigma$ (hence the naming convention), then
\begin{itemize}
    \item [i)] $X_1\cap X_2$ strongly deformation retracts to $\Sigma\cup(3$-dimensional $2$--handles along the $(g-p)$ $\alpha$ curves),
    \item[ii)] $X_2\cap X_3$ strongly deformation retracts to $\Sigma\cup(3$-dimensional $2$--handles along the $(g-p)$ $\beta$ curves),
    \item[iii)] $X_3\cap X_1$ strongly deformation retracts to $\Sigma\cup(3$-dimensional $2$--handles along the $(g-p)$ $\gamma$ curves).
\end{itemize}

\begin{theorem}\cite{cgp}
Every relative trisection $\tri$ of a $4$-manifold $X^4$ can be described by a relative trisection diagram.
\end{theorem}

Note in particular that a relative trisection diagram then determines an open book on $\boundary X$. Given the relative trisection diagram, this open book is determined up to automorphisms of $\boundary X$ that extend over $X$. When $(X,\T)$ is specified, then a diagram $\D$ of $\T$ determines this open book up to isotopy.

Since a relative trisection diagram $\D$ determines a relative trisection $\T$, we may write $\O_\D$ to mean the open book induced by $\T$ on $X_\D:=X_\T$. We similarly say that $\D$ induces the (abstract) open book $\O_\D$. Again, when $(X,\T)$ is specified first, then this open book is determined up to isotopy in $\boundary X$, otherwise $\O_D$ is determined up to an automorphism of $\boundary X$ extending over $X$.

\subsection{Trisection stabilization}\label{sec:stab}
Interior stabilization of trisections was introduced in~\cite{gaykirby}. This is completely analogous to stabilization of Heegaard splittings.

\begin{definition}\label{stabdef}
Let $\T=(X_1,X_2,X_3)$ be a trisection or relative trisection of a 4-manifold $X$. Let $\T_1=(Y_1,Y_2,Y_3)$ be a $(1,k)$-trisection of $S^4$, where $\{k_1,k_2,k_3\}=\{0,0,1\}$ (as in Figure~\ref{F:S4genus1}). We obtain another trisection $\T'=(Z_1,Z_2,Z_3)$ of $X$ by taking the connected-sum $(X,\T'):=(X,\T)\#(S^4,\T_1)$. Here, the ball removed from each of $X, S^4$ when performing the connected-sum is centered at a point in the triple-intersection surface of $\T$, $\T_1$. We arrange $\T$ and $\T_1$ so that $X_i$ and $Y_i$ meet in a ball along the connected-sum 3-sphere. Then we let $Z_i=X_i\natural Y_i$.

We say $\T'$ is obtained by {\emph{stabilizing}} $\T$. Conversely, we say that $\T$ is obtained by {\emph{destabilizing}} $\T'$. When $X$ has boundary, we will usually refer to (de)stabilization as {\emph{interior}} (de)stabilization.
\end{definition}

Note that there are three kinds of stabilization we may perform on a trisection $\T$. One of the these stabilizations increases $k_1$ while fixing $k_2$ and $k_3$; the other stabilizations increase $k_2$ or $k_3$ (see Figure~\ref{F:S4genus1}). When $(X,\T)$ is a relatively trisected 4-manifold with boundary, interior stabilization does not affect the induced open book on $\boundary X$. Gay and Kirby~\cite{gaykirby} show that any two relative trisections $(X,\T_1),(X,\T_2)$ which induce the same open book on the bounding $3$-manifold $\boundary X$ become isotopic after finitely many interior stabilizations of each of $\T_1, \T_2$.

When $(X,\T)$ is a relatively trisected 4-manifold, we consider two different forms of stabilizations that take place near $\boundary X$. First we briefly discuss Lefschetz fibrations over the disk, through which both stabilizations must pass. The reader is referred to~\cite{nickthesis,cgp, CO, OZ} for details. 

\begin{definition}
A \emph{Lefschetz fibration} of a $4$--manifold with boundary $W$ is a smooth map $f:W \rightarrow D^2$ with a finite number of isolated singularities $C_f=\{c_1, \ldots, c_n\}$ such that each of the critical points can be locally modeled by the map $(z,w)\mapsto z^2+w^2.$
\end{definition}
For any $y\in D^2\setminus C_f$, $f^{-1}(y)\cong F$, where $F$ is a surface with boundary called the \emph{regular fiber of $f$}. If $c_i\in C_f$ is a critical value of $f$, then we refer to $f^{-1}(q)$ as a \emph{critical fiber}. 
It is well known that each $c_i\in C_f$ corresponds to a simple, closed curve $\delta_i \subset F$ called a \emph{vanishing cycle}. The critical fiber corresponding to $c_i$ is obtained by contracting the vanishing cycle $\delta_i$ to a point, resulting in a nodal singularity. The topology of a Lefschetz fibration can be recovered by the data of the regular fiber and an ordering of vanishing cycles as follows: We attach $4$--dimensional $2$--handles to $F\times D^2$ along neighborhoods of the vanishing cycles in sequential order. Each $2$--handle is attached to $F\times\{x_i\}$ for distinct values of $x_i\in partial D^2$.  
The framing of a $2$--handle is $-1$ if the local model of the corresponding critical point is orientation preserving, and is $+1$ if the local model is orientation reversing. A Lefschetz fibration with both orientation preserving and reversing local models is called \emph{achiral}.

One key feature shared by Lefschetz fibrations and relative trisections of $W$ is that they induce an open book decomposition on $\partial W.$ While obtaining the explicit (abstract) open book from a relative trisection diagram is quite involved (see Section~\ref{sec:monodromy}), obtaining the open book decomposition from a Lefschetz fibration is quite simple. Given an ordering of the vanishing cycles $\delta_1, \ldots, \delta_n$ on $F$, let $\tau(\delta_i)$ denote the positive Dehn twist of $F$ along the curve $\delta_i$. The open book decomposition of $\partial W$ is $\mathcal{O}_f=(F, \phi),$ where $\phi\in \emph{Map}(F, \partial F)$ is the composition of Dehn twists $\tau^{-\sigma_n}(\delta_n)\cdots\tau^{-\sigma_1}(\delta_1)$, where $$\sigma_i=\begin{cases}-1&\text{the local model of $c_i$ is orientation-preserving}\\\hphantom{-}1&\text{otherwise}.\end{cases}$$ 

By carefully adding a canceling $1$-$2$ pair to $W$, we can obtain a new Lefschetz fibration $f'$ of $W$. We require that the feet of the $1$--handle are attached to neighborhoods of points on the binding of $\mathcal{O}$. The attaching sphere $\delta$ of the $2$--handle is comprised of two arcs $\delta =a\cup a'$, where $a$ is a properly embedded arc in a single page $\mathcal{O}$ and $a'$ is the core of the $1$--handle. Finally, $\delta$ must have framing $\pm1$. This modification is referred to as an {\emph{$S$ move}} in~\cite{pz}. This ensures that the bounding open book is modified by a positive/negative Hopf stabilization, $\partial S$. The effect on the regular fiber is depicted in Figure~\ref{F:LefStab}. Note that the regular fibers of $f'$ differ from those of $f$ by an additional $1$--handle (i.e. a band), and $C_{f'}=C_f\cup c$, where $c$ is a Lefschetz singularity with vanishing cycle $\delta$. In the sequential ordering of vanishing cycles for $f'$, $\delta$ appears last.

\begin{figure}{\centering
\subcaptionbox{The regular fiber of $f$.}{
\includegraphics[scale=.5]{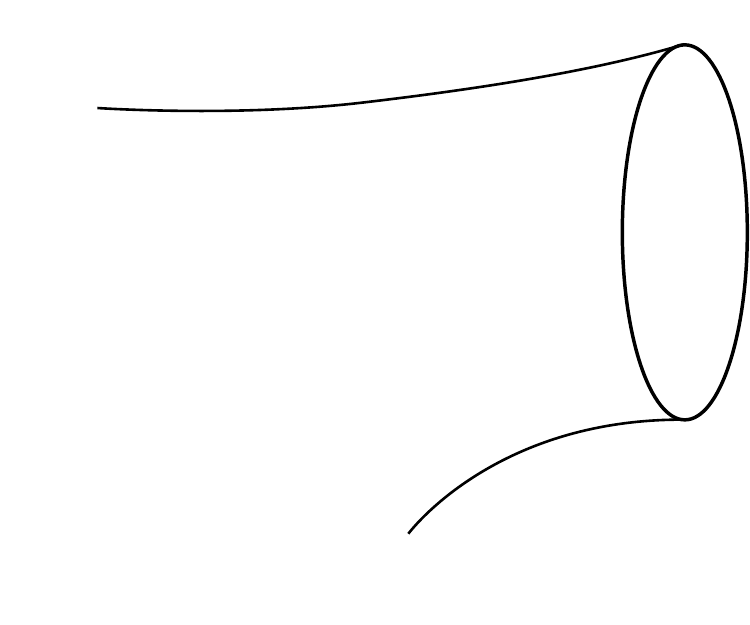}
}
\hspace{1cm}
\subcaptionbox{The regular fiber of $f'$.
}{
\includegraphics[scale=.5]{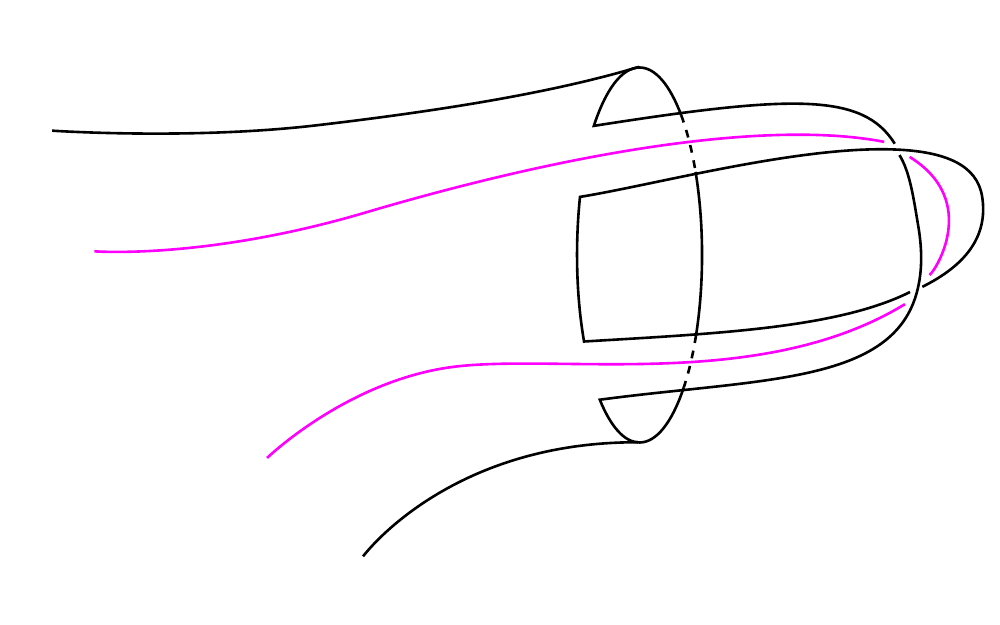}
}
\subcaptionbox{The local result of a relative stabilization.}{\includegraphics[scale=.75]{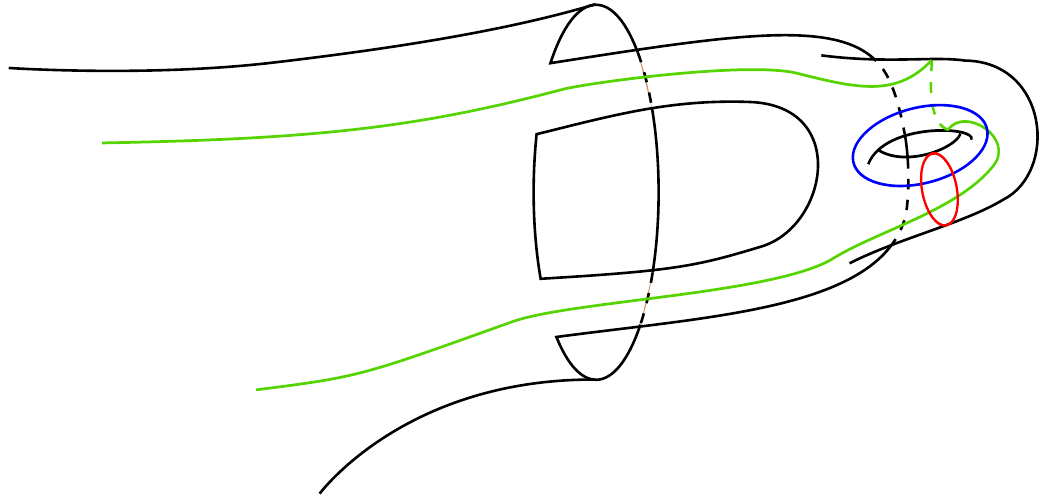}
}
\caption{Modifying a Lefschetz fibration $\mathcal{F}$ by adding a cancelling $1$-$2$--pair. In~\cite{pz}, this is referred to as an $S$ move. This induces a Hopf stabilization to the bounding open book decomposition. In general, one could attach the Hopf band to different boundary components. The new regular fiber is obtained from the old fiber by adding a band. The new (last) vanishing cycle runs over a core of this band.} 
\label{F:LefStab}}
\end{figure}
In a neighborhood of a Lefschetz singularity, there is a local perturbation $(z,w)\mapsto z^2+w^2+t\mathrm{Re}(w)$, known as \emph{wrinkling}, which changes the nodal singularity to a triply cusped singular set. Roughly speaking, in the case of Lefschetz fibrations over the disk, wrinkling all of the singularities will result in a relative trisection (diagram) which induces the same open book decomposition on $\partial W$ as the initial Lefschetz singularity. Adding a cancelling $1$-$2$ pair as above (\emph{to a relative trisection}) and wrinkling the corresponding singularity gives rise to a \emph{relative stabilization}, which we shortly define in Definition~\ref{def:relstab}. This procedure on Lefschetz fibrations motivates the general definition of a relative stabilization. The local model for obtaining a relative trisection diagram via wrinkling is shown in Figure~\ref{F:wrinkle}. The total move on relative trisection diagrams is sketched in Figure~\ref{F:LefStab}. 

\begin{figure}{\centering
\subcaptionbox{}{
\includegraphics[scale=.45]{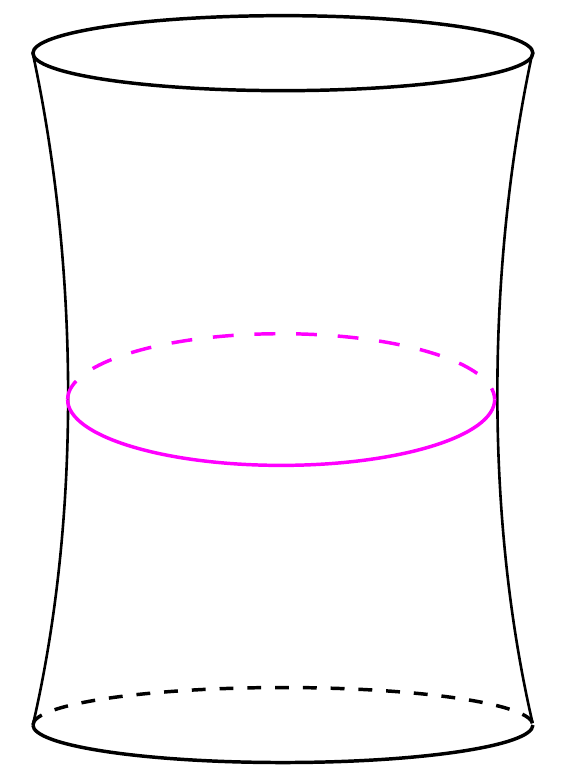}
}
\hspace{1cm}
\subcaptionbox{}{
\includegraphics[scale=.45]{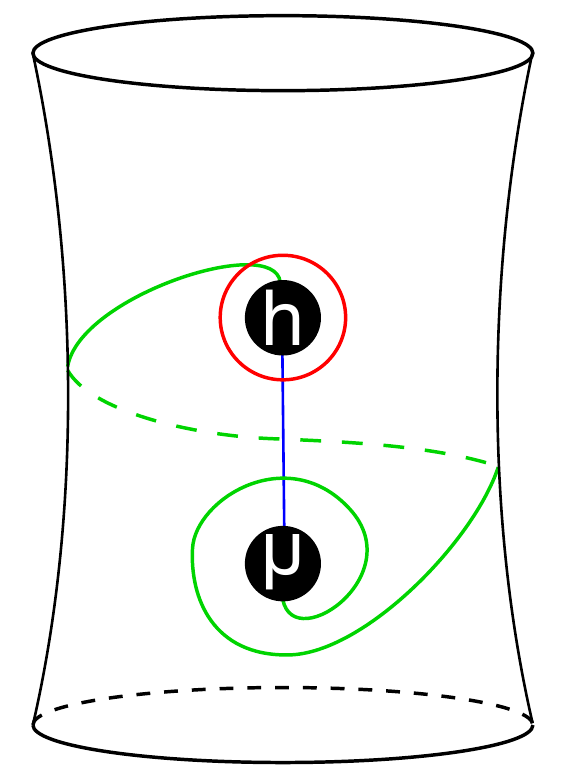}
}
\caption{The effect of wrinkling a vanishing cycle.}
\label{F:wrinkle}
}
\end{figure}


\begin{definition}\label{def:relstab}
Let $\T=(X_1,X_2,X_3)$ be a relative trisection of a 4-manifold $X$. Let $\T_1=(Y_1,Y_2,Y_3)$ be a $(1,1;0,2)$-trisection of $B^4$. We obtain a new relative trisection $\T'=(Z_1,Z_2,Z_3)$ of $X$ by taking the boundary connected-sum $(X,\T'):=(X,\T)\natural(B^4,\T_1)$. To form this boundary connected-sum, we Murasugi-sum the open books $\O_\T$ and $\O_{\T_1}$ along squares, making sure to glue the plumbed square in $X_i\cap X_j\cap\partial X$ to one in $Y_i\cap Y_j\cap\partial B^4$. Then we let $Z_i=X_i\natural Y_i$.

We say $\T'$ is obtained by {\emph{relatively stabilizing}} $\T$. Conversely, we say that $\T$ is obtained by {\emph{relatively destabilizing}} $\T'$. We illustrate a relative stabilization from this perspective in Figure~\ref{F:relstab}. In practice, we typically specify the Murasugi-sum square in the page $\Sigma_{\alpha}=X_1\cap X_2\cap\partial X$ in order to perform the stabilization operation diagramatically, so this is what we have done in Figure~\ref{F:relstab}.
\end{definition}

\begin{figure}{\centering
\subcaptionbox{An arc $\delta$ in $\Sigma_{\alpha}$. The arc $\delta$ may intersect many $\beta$ and $\gamma$ curves, as indicated.}{
\includegraphics[width=2in]{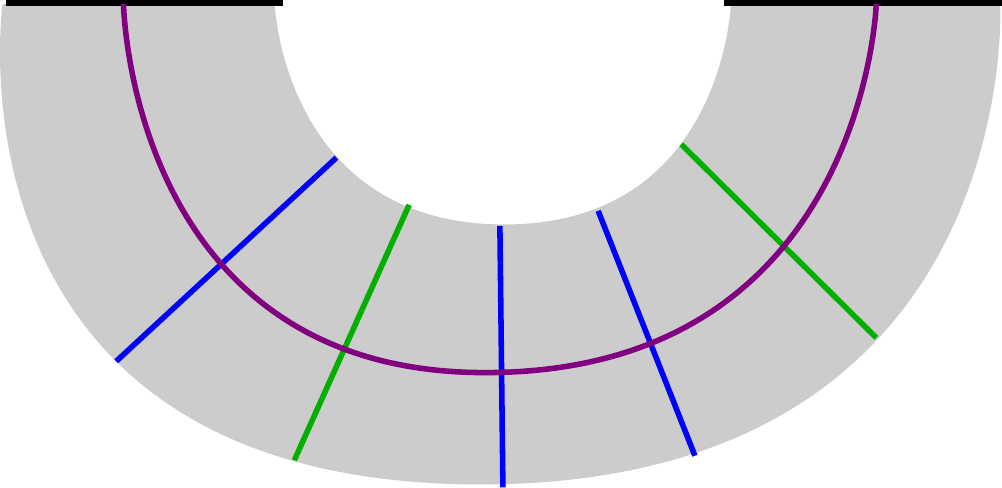}
}
\hspace{1cm}
\subcaptionbox{The effect of positive relative stabilization along $\delta$.}{
\includegraphics[width=2in]{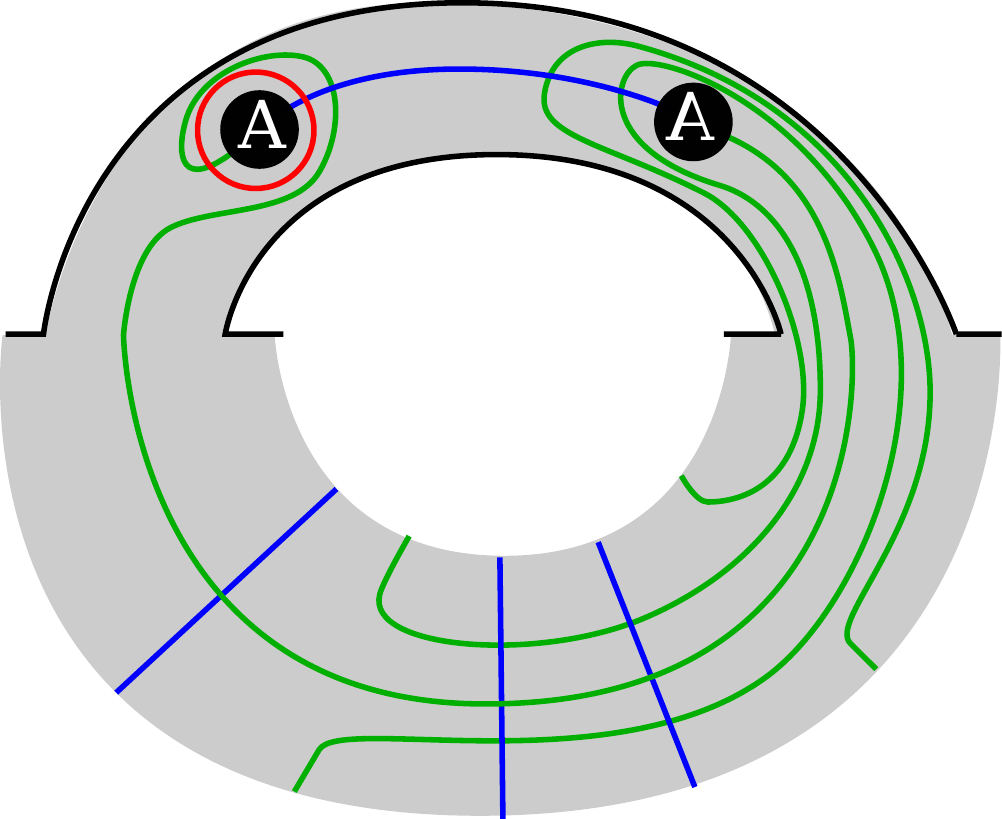}
}
\hspace{1cm}
\subcaptionbox{The effect of negative relative stabilization along $\delta$.}{
\includegraphics[width=2in]{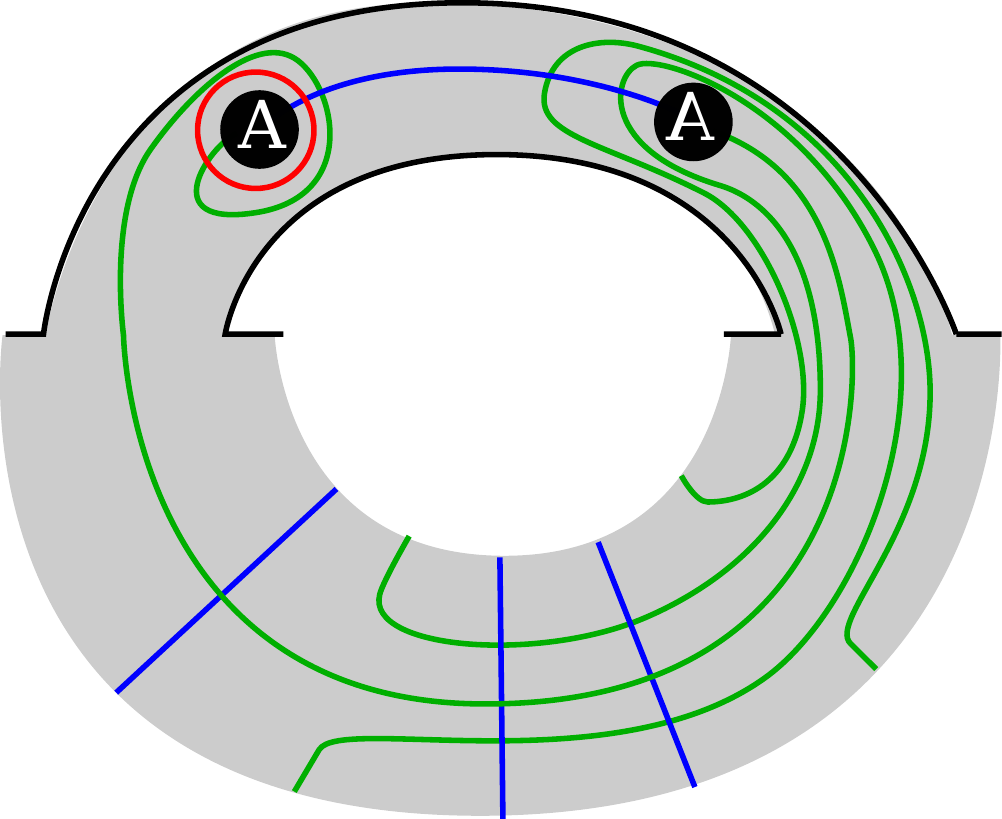}
}
\caption{Local illustration of positive and negative relative stabilization. Compare to Figure~\ref{F:LefStab}.}\label{F:relstab}
}

\end{figure}
We will see a more detailed construction of plumbings of trisections in Theorem~\ref{thm:murasugi}.

When a $3$-manifold $Y$ has a unique $\textrm{spin}^c$ structure, any two open books on $Y$ can be related by Hopf stabilizations/destabilizations~\cite{giroux}. Since $\partial H$ does not change the $\textrm{spin}^c$ structure of the homotopy plane field supported by an open book, not all open book decompositions of a fixed $3$-manifold can be related by Hopf stabilizations and destabilizations.


Piergallini and Zuddas~\cite{pz} show that the {\emph{$U$}} move on Lefschetz fibrations 
can achieve the goal of changing the $spin^c$ structure associated to the induced open book of $Y$. We are particular interested in the effect of the $U$ move on the boundary open book.

\begin{definition}
Let $\O$ be an open book on a 3-manifold defined by a page $P$ and monodromy $\phi:P\to P$. Isotope $\phi$ so that it fixes an open disk $D$ pointwise. Let $P'$ be a copy of $P$ with two smaller disjoint open disks within $D$ deleted. Let $\phi':P\to P$ be an automorphism which agrees with $\phi$ outside $D$, and inside $D$ consists of Dehn twists about each component of $\boundary P'\setminus\boundary P$ of opposite sign. We say that the open book $\O'$ with page $P$ and monodromy $\phi$ is obtained from $\O$ by a {\emph{$\partial U$ move}}. 
\end{definition}

It is an easy exercise to show that the total spaces of $\O$ and $\O'$ are homeomorphic 3-manifolds. Just as there are many ways to perform positive or negative Hopf stabilization to an open book (i.e. many choices of square in the page along which to plumb a Hopf band), there are many ways to choose a disk $D$ in a page $P$ along which to perform a $\partial U$ move. In practice, one should isotopy the monodromy $\phi$ of an open book $\O$ to be a product of Dehn twists along curves $C_1,\ldots, C_n$ in $P$. Then one can consider performing the $\partial U$ move in any component of $P\setminus(C_1\cup\cdots\cup C_n)$. Piergallini and Zuddas~\cite{pz} show that some sequence of these $\partial U$ moves (and inverses) and Hopf stabilizations can turn $\O$ into any open book of the same 3-manifold.

\begin{theorem}[{\cite[Theorem 3.5]{pz}}]\label{pzthm}
Any two open book decompositions of a closed, oriented 3-manifold 
$M^3$ can be made ambiently isotopic after a finite number of $\partial S_{\pm}$ moves (positive and negative Hopf stabilization) and $\partial U$ moves and their inverses.
\end{theorem}


We introduce a \emph{relative double twist} of a relative trisection diagram, which alters the $\textrm{spin}^c$ structure associated to $\O_\D$ by inducing a $\partial U$ move. By Theorem~\ref{pzthm}, we conclude that this move, together with the relative stabilization, will allow us to modify a relative trisection diagram so as to induce any desired open book on the boundary.

\begin{definition}[Relative double twist]\label{def:twist}
Let $\T=(X_1,X_2,X_3)$ be a $(g,k$; $p,b)$-relative trisection of a 4-manifold $X^4$.


Suppose there is a disk $D\subset \Sigma:=X_1\cap X_2\cap X_3$ disjoint from all the compression circles used to build the compression bodies $X_1\cap X_2, X_2\cap X_3, X_3\cap X_1$ on $\Sigma$. Then there is a copy $E_i=D\times I\times I$ in each $X_i$. Suppose that $E_i$ and $E_j$ agree in $X_i\cap X_j$ for each $i,j$. This condition is not automatically satisfied;  when this is true, we will say that $D$ {\emph{satisfies the relative double twist criterion.}} In this case, $E_1\cup E_2\cup E_3$ is a copy of $D\times D^2$, with $D$ identified with $D\times\{0\}$ and $(D\times D^2)\cap\partial X=D\times\partial D^2$.

Now consider the $(2,1;0,2)$-relative trisection $(Y_1,Y_2,Y_3)$ of $S^2\times D^2$ pictured on the right of Figure~\ref{fig:reldoubletwist}. There is a disk $D'$ in $Y_1\cap Y_2\cap Y_3$ that similarly gives a copy of $D'\times D^2$ inside $S^2\times D^2$.

\begin{figure}
	\labellist
		\pinlabel {$D$} at 120 110
		\pinlabel {$\Sigma$} at 120 165
		\pinlabel {$D'$} at 335 105
	\endlabellist
	\includegraphics[scale=.6]{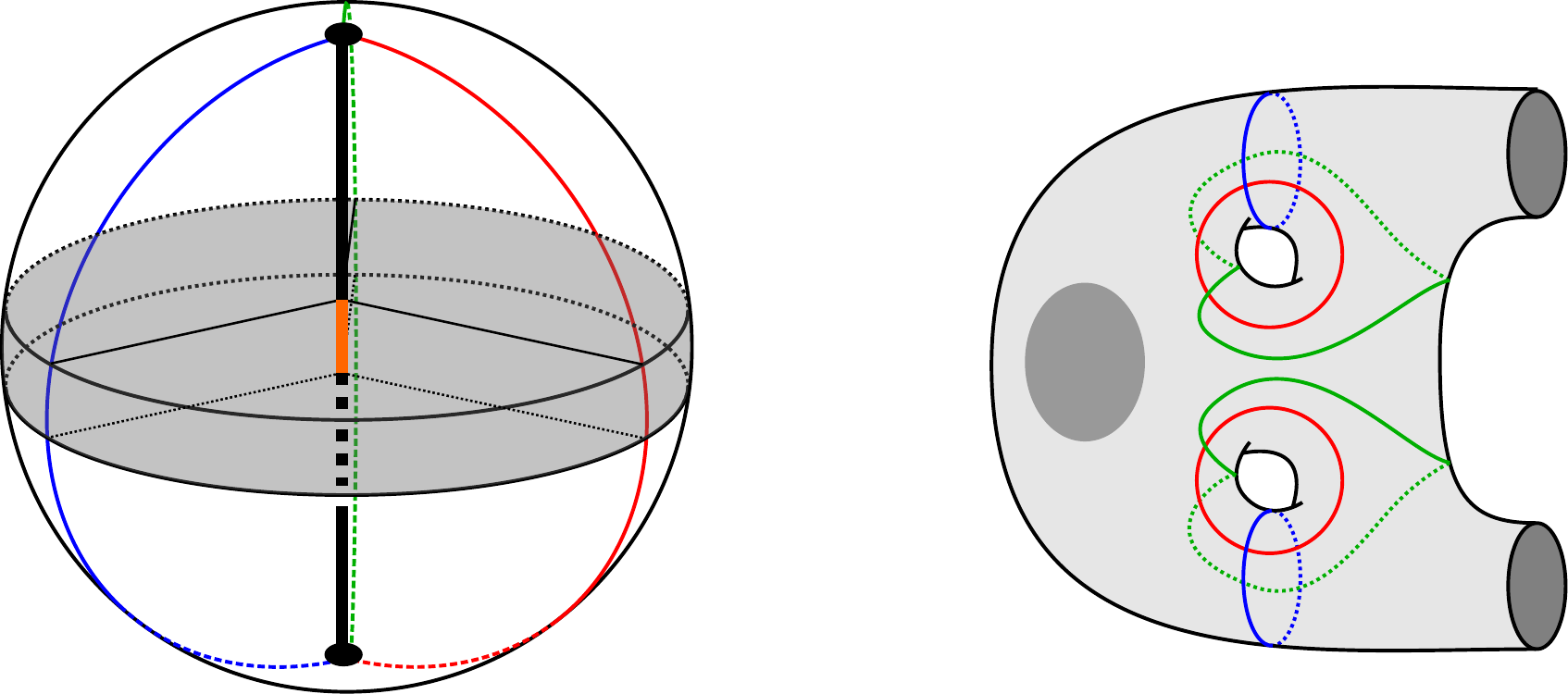}
	\caption{Left: a schematic of a disk $D$ in a trisection surface $\Sigma$ for a trisection $\T$ of a 4-manifold $X^4$ satisfying the relative double twist criterion. We shade the copy of $D\times D^2$ inside $X$. Right: a trisection diagram of he trisection $(Y_1,Y_2,Y_3)$ of $S^2\times D^2$. We shade the disk $D'$ as in Definition~\ref{def:twist}. We similarly find a copy of $D'\times D^2$ inside $S^2\times D^2$. To perform the relative double twist operation to $\T$ along $D$, we delete $D\times D^2$ from $X$ and $D'\times D^2$ from $S^2\times D^2$, then glue the resulting manifolds along $D\times\partial D^2\sim D'\times\partial D^2$ so that $X_i\setminus (D\times I\times I)$ glues to $Y_i\setminus(D'\times I\times I)$.}
	\label{fig:reldoubletwist}
\end{figure}

Let \[Z^4=(X^4\setminus(D\times D^2))\cup_{\partial D\times S^1\sim \partial D'\times S^1}(S^2\times D^2\setminus D'\times D^2),\] with \[Z_i=(X_i\setminus E_i)\cup(Y_i\setminus (D'\times D^2)).\]

Then $Z^4\cong X^4$, and $\T':=(Z_1,Z_2,Z_3)$ is a $(g+2,k+2;p,b+2)$-relative trisection of $X^4$. We say that $\T'$ is obtained from $\T$ by a {\emph{relative double twist along $D$.}} in Figure~\ref{F:Doubletwist}, we show how to obtain a relative trisection diagram for $\T'$ from one for $\T$.
Note that $\O_{\T'}$ is obtained from $\O_{\T}$ by a $\partial U$ move along $D\times\{\text{pt}\}$.

\end{definition}

\begin{figure}{\centering
\subcaptionbox{Preparing to do a relative double twist to a relative trisection diagram $\D=(\Sigma;\alpha,\beta,\gamma)$ along a disk $D$.
\label{F:Ua}}{
\labellist
  \pinlabel {$D$} at 110 80
\endlabellist
\includegraphics[width=2in]{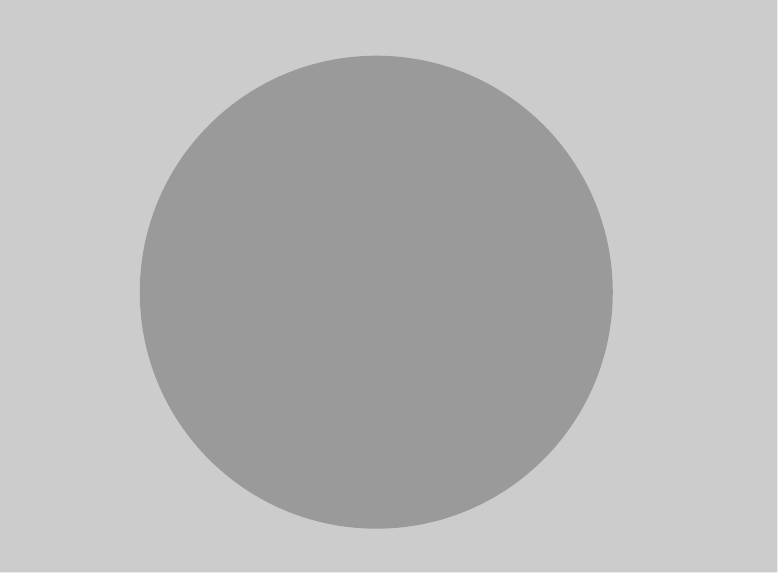}
}
\hspace{1cm}
\subcaptionbox{After performing the relative double twist, we obtain a relative trisection diagram $\D'=(\Sigma',\alpha',\beta',\gamma')$. 
\label{F:Ub}}{
\labellist
\endlabellist
\includegraphics[width=2in]{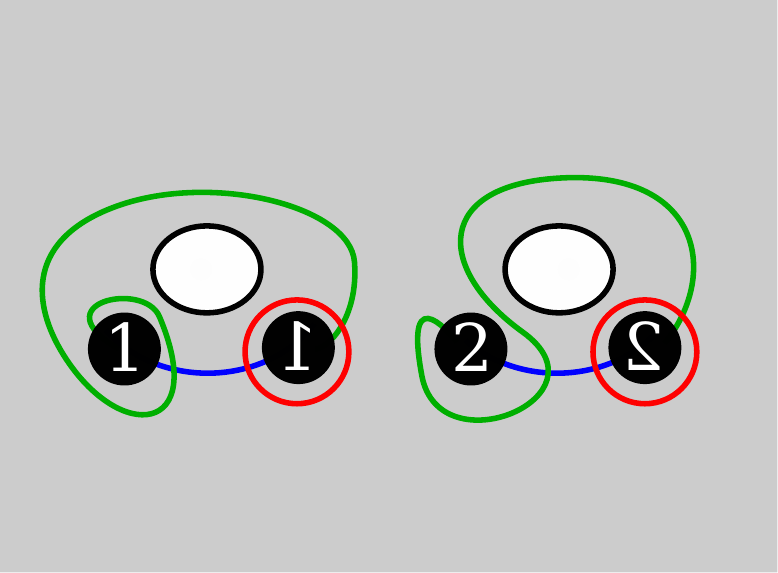}
}


\caption[width=\paperwidth]{Moving from Figure (A) to figure (B) is an instance of the relative double twist on relative trisection diagram, if $D$ satisfies the relative double twist criterion, as in Definition~\ref{def:twist}. See Remark~\ref{rem:dobdryu} if interested in finding such a disk.  The surface $\Sigma'$ is obtained from $\Sigma$ by deleting two open disks and attaching two tubes with $D$. Each of $\alpha,\beta,\gamma$ is included in $\alpha',\beta',\gamma'$ (correspondingly) and we draw the curves in $\alpha'-\alpha,\beta'-\beta,\gamma'-\gamma$.}
\label{F:Doubletwist}
}
\end{figure}

The description of the relative double twist move as being the result of gluing a 4-manifold to $S^2\times D^2$ after removing a 2-handle from each was inspired by the ``poking" move of Aranda and Moeller~\cite{jesseroman}.

\begin{proposition}\label{prop:dobdryu}
Let $\T=(X_1,X_2,X_3)$ be a relative trisection of $X^4$. By performing a relative double twist, we may induce any $\partial U$ move on $\O_\T$.
\end{proposition}
\begin{proof}
Choose a $\partial U$ move on $\O_\T$. That is, factor the monodromy of $\O_\T$ into a product of Dehn twists and choose a disk $D$ in $\Sigma_{\alpha}$ that is disjoint from each Dehn twisted curve in the monodromy $\phi$ of $\O_\T$.  View $D$ as being in $\Sigma$. Since $D$ is fixed by $\phi$, it is possible to isotope $X_2$ and $X_3$ so that the copies of $D\times I\times I$ in each $X_i$ agree in $X_i\cap X_j$ (i.e., so that $D$ satisfies the relative double twist criterion). Then perform a relative double twist on $\T$ along $D$.

\end{proof}

\begin{remark}\label{rem:dobdryu}

In the proof of Proposition~\ref{prop:dobdryu}, we may have to isotope $X_2$ and $X_3$ in order to do the relative double twist. One should view this as analogous to the condition that in order to perform a $\partial U$ move on an open book, we must first isotope the open book to fix a disk pointwise. If we start with a relative trisection diagram $\D=(\Sigma;\alpha,\beta,\gamma)$ for $X$ in mind, then this isotopy might induce slides of the $\beta$ and $\gamma$ curves in $X$ that might seem mysterious.

In practice, to achieve a relative double twist, it is helpful to start with a relative trisection diagram in which the support of the open book monodromy $\phi$ is easily visible. Once this is done, the process for the relative double twist is illustrated in Figure~\ref{F:dobdryu} and explained in what follows. Slide the $\alpha,\beta$ curves to make them standard, so that $\D$ yields a Kirby diagram of $X$ containing $\Sigma_\alpha$, as in~\cite{price,nonorientable}.  The algorithm of~\cite{withboundary} reverses this procedure (up to interior (de)stabilization) by projecting the 2-handle attaching circles to $\Sigma_{\alpha}$ and attaching certain decorated tubes (see~\cite{withboundary} for details). Now $\phi$ is supported in a neighborhood of the projected $2$-handle circles. Let $D$ be a disk in $\Sigma_{\alpha}$ disjoint from any tubes or projections of 2-handle circles. Choose an arc $\delta$ in $(\Sigma_{\alpha}\cap\Sigma)\setminus\mathring{D}$ from $\partial D$ to $\partial \Sigma_{\alpha}$. (The arc $\delta$ may intersect projections of 2-handle circles.) Then we may obtain a trisection diagram $\D'$ of $X$ by adding two dotted circles (one of which is a double of $\delta$) and two 2-handles to the Kirby diagram as in Figure~\ref{F:dobdryu} and then applying the procedure of~\cite{withboundary} to obtain a relative trisection diagram. The trisection described by $\D'$ is obtained from the one described by $\D$ up to interior (de)stabilization.

\begin{figure}{\centering
\subcaptionbox{A relative trisection diagram $\D=(\Sigma;\alpha,\beta,\gamma)$ with $\alpha$ and $\beta$ standard.
\label{F:twista}}{
\labellist
\endlabellist
\includegraphics[width=1.2in]{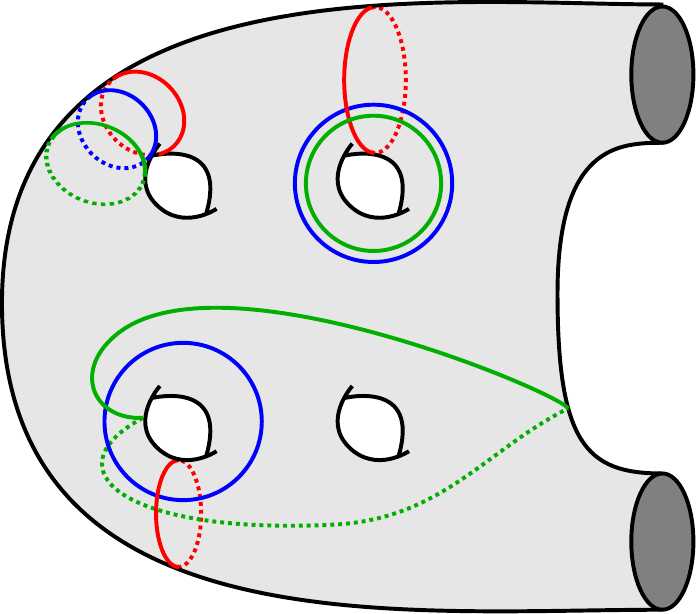}
}
\hspace{1cm}
\subcaptionbox{As in~\cite{price,nonorientable}, we obtain a Kirby diagram from $\D$ containing $\Sigma_{\alpha}$. Note the projections of the 2-handle attaching circles to $\Sigma_{\alpha}$. 
\label{F:twistb}}{
\labellist
    \pinlabel {\textcolor{ForestGreen}{$-1$}} at 15 50
\endlabellist
\includegraphics[width=1.2in]{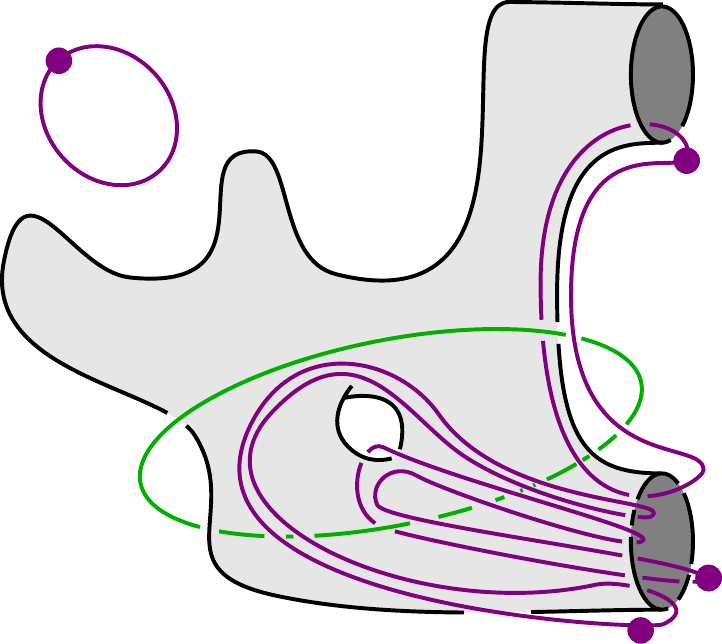}
}
\hspace{1cm}
\subcaptionbox{Choose a disk $D$ (shaded) in $\Sigma_{\alpha}$ disjoint from the projections of the 2-handle curves, and an arc $\delta$ from $\partial D$ to $\partial\Sigma_{\alpha}$.}
{
\labellist
    \pinlabel {\textcolor{ForestGreen}{$-1$}} at 15 50
\endlabellist
\includegraphics[width=1.2in]{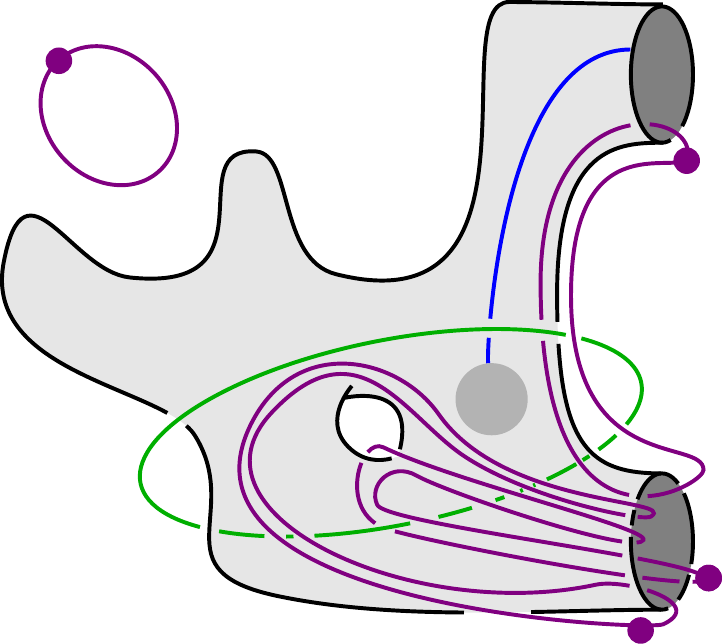}
}

\vspace{1cm}
\subcaptionbox{Delete two open disks in $\Sigma_{\alpha}$ in $D$. Add two 2-handles (with framing $+1$ and $-1$, not pictured for space) to the Kirby diagram around these holes, and two dotted circles (one linking the two 2-handles, the other doubling the arc $\delta$).}
{
\labellist
    \pinlabel {\textcolor{ForestGreen}{$-1$}} at 25 50
\endlabellist
\includegraphics[width=2in]{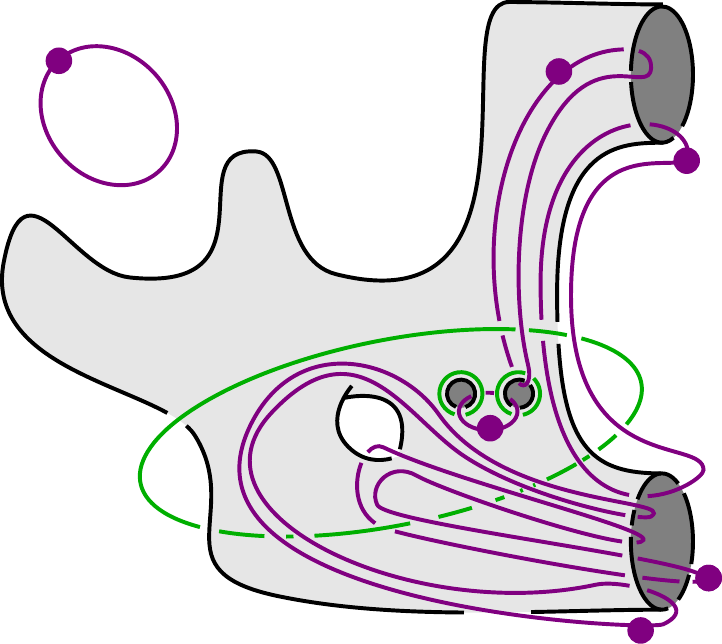}
}
\hspace{1cm}
\subcaptionbox{Perform the algorithm of~\cite{withboundary} to turn the picture back into a relative trisection diagram $\D'$. The open book $\O_{\D'}$ is obtained from $\O_{\D}$ by a $\partial U$ move along $D$; the relative trisection $\T_{\D'}$ is obtained from $\T_\D$ by a relative double twist along $D$ up to interior (de)stabilization. Note that Figure~\ref{F:Ub} is contained in $\D'$.}
{

\includegraphics[width=2in]{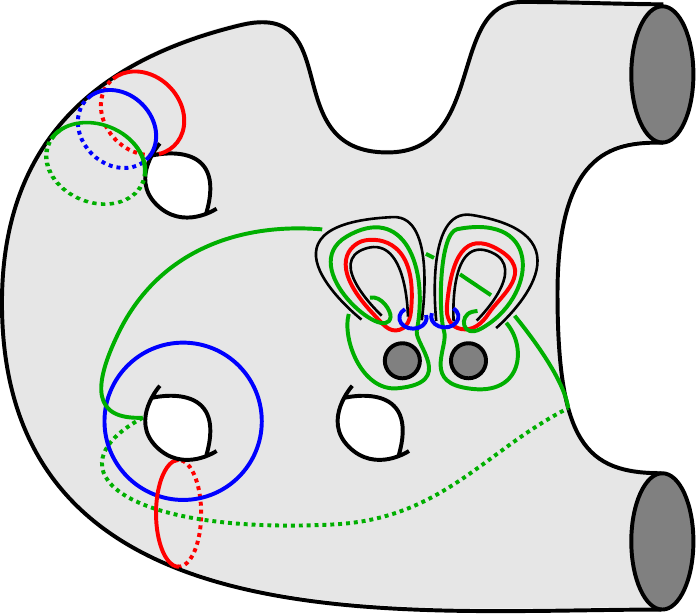}
}

\caption[width=\paperwidth]{Moving from Figure (A) to figure (E), we show how to diagrammatically perform the relative double twist to a relative trisection diagram. Not illustrated: we preemptively standardize the $\alpha$ and $\beta$ curves of the starting relative trisection diagram.}
\label{F:dobdryu}
}
\end{figure}



\end{remark}

Thus, a relative trisection of a given 4-manifold is unique up to a simple set of moves. Previously, this was only known among relative trisections with the same boundary data~\cite{gaykirby} or suitably similar boundary data~\cite{nickthesis}.

\begin{theorem}\label{uniquenesscor}
Let $\T$ and $\T'$ be relative trisections of a 4-manifold $X$, with a fixed identification $X\cong X_\T\cong X_{\T'}$. Then $\T$ and $\T'$ are related by a sequence of ambient isotopies, stabilizations, relative stabilizations, relative double twists, and the inverses of these moves.
\end{theorem}

\begin{proof}
By Theorem~\ref{pzthm} and Proposition~\ref{prop:dobdryu}, we may perform relative stabilizations, relative double twists, and inverse moves to $\T$ until we obtain a relative trisection $\T''$ with $\O_{\T''}$ ambiently isotopic to $\O_{\T'}$. Then by~\cite[Theorem 21]{gaykirby}, we can perform interior stabilizations and destabilizations to $\T'$ to obtain a relative trisection ambiently isotopic to $\T$.
\end{proof}


\subsection{The Monodromy Algorithm}\label{sec:monodromy}
An essential component to encoding a relatively trisected $4$-manifold via a trisection diagram is the monodromy algorithm from~\cite{cgp}, which we review here for completeness. Let $\mathcal{D}=(\Sigma; \alpha, \beta,\gamma)$ be a $(g,k;p,b)$ relative trisection diagram for a $4$-manifold $W$ with connected boundary. The page of the open book induced by $\mathcal{D}$ is the genus $p$ surface with $b$ boundary components $\Sigma_\alpha$ obtained by surgering $\Sigma$ along the $\alpha$ curves. Any essential, properly embedded arc in the page $\Sigma_\alpha$ which misses the disks resulting from surgering $\Sigma$ along the $\alpha$ curves can be identified with an arc in the trisection surface $\Sigma$.

Let $A_\alpha=\{a_1, \ldots, a_l\}$ be collection of $l=2p+b-1$ disjoint, essential, properly embedded arcs in $\Sigma_\alpha$ such that their complement in $\Sigma_\alpha$ is a disk (We will think of $A_\alpha$ as a subset of both $\Sigma$ and $\Sigma_\alpha$ and will specify which surface the arcs are in when necessary.) We fix a point on each boundary component of $\Sigma$, and we consider all arcs in this paper up to isotopies which do not take the endpoints of the arcs over these fixed points. The following algorithm produces a collection of arcs $\overline{A}=\{\overline{a}_1, \ldots, \overline{a}_l\} \subset \Sigma_\alpha$ which defines a diffeomorphism $\phi:\Sigma_\alpha \rightarrow \Sigma_\alpha$ by requiring $\phi(a_i)=\overline{a}_i$ for each $i$. This $\phi$ is the monodromy of the open book $\O_{\D}$.

\begin{algorithm}\cite{cgp}
\begin{itemize}
	\item[1.] Slide $\alpha$ curves and $A_\alpha$ over $\alpha$ curves (without introducing intersections to $\alpha\cup A_\alpha$) and $\beta$ curves over $\beta$ curves (without introducing self-intersections to $\beta$) until they are disjoint from $\beta$. Call the resulting collection of arcs $A_\beta=\{b_1, \ldots, b_l\}$, where $b_i$ is obtained from $a_i$. Let $\beta'$ be the curves resulting from sliding $\beta$, so $\beta'\cap A_\beta=\emptyset$.
	\item[2.] Slide $\beta'$ curves and $A_\beta$ over $\beta'$ curves (without introducing intersections to $\beta'\cup A_\beta$) and slide $\gamma$ curves over $\gamma$ curves (without introducing self-intersections to $\gamma$) until they are disjoint from $\gamma$. 
	Call the resulting collection of arcs $A_\gamma = \{c_1, \ldots, c_l\}$, where $c_i$ is obtained from $b_i$. Let $\gamma'$ be the curves resulting from sliding $\gamma$, so $\gamma'\cap A_\gamma=\emptyset$.
	\item[3.] Slide $\gamma'$ curves and $A_\gamma$ over $\alpha$ curves (without introducing intersections to $\gamma'\cup A_\gamma$) until they are is disjoint from $\alpha$. Call the resulting collection of arcs $\tilde{A} = \{\tilde{a}_1, \ldots, \tilde{a}_l\}$, where $\tilde{a}_i$ is obtained from $c_i$. Let $\alpha'$ be the curves resulting from sliding $\alpha$, so $\alpha'\cap \tilde{A}=\emptyset$.
	\item[4.] Slide $\alpha'$ and $\tilde{A}$ over $\alpha'$ curves until $\alpha'$ is again equal to the original $\alpha$ curves, while always keeping the curves and arcs disjoint. Call the resulting collection of arcs $\overline{A} = \{\overline{a}_1, \ldots, \overline{a}_l\}$, where $\overline{a}_i$ is obtained from $\tilde{a}_i$. We have $\alpha\cap\overline{A}=\emptyset$.
\end{itemize}
Since $A_\alpha,\overline{A}\subset\Sigma_{\alpha}$ each have complement a disk, we may uniquely define $\phi:\Sigma_{\alpha}\to\Sigma_{\alpha}$ up to isotopy by specifying that $\phi(a_i)=\overline{a}_i$.
\end{algorithm}

\begin{remark}\label{R:1}It is helpful to keep the following facts in mind when performing the above algorithm:
	\begin{itemize}
		\item[1.] Such slides in each step of the algorithm exist since we know any pair of curves can be made to be in standard positions.
		\item[2.] Two types of choices are made when performing the algorithm: the choice of arcs $A_\alpha$ and the choice of arc slides in each step. An important part of the proof of the algorithm is that $\phi$ is independent of these choices, up to isotopy and conjugation in the mapping class group of $\Sigma_{\alpha}$.
		\item[3.] By starting with a standard set of $\alpha$ and $\beta$ curves, we may find initial arcs for this algorithm which are disjoint from both sets of curves. This makes the first step of the algorithm redundant.
	\end{itemize}
\end{remark}

\subsection{The \texorpdfstring{\(\L\)}{L}--invariant} \label{sec:closedL}

In this section, we briefly review the definitions and some of the results in~\cite{linvariant}. We must first understand a certain complex associated to a surface.

\begin{definition}
\label{def:HTComplex}
Given a closed orientable surface $\Sigma$, the \textbf{cut complex} of $\Sigma$, $HT(\Sigma)$, is the simplicial complex built as follows.

Each vertex of $HT(\Sigma)$ corresponds to a cut system for $\Sigma$; that is, a collection of $g$ non-separating simple closed curves on $\Sigma$ whose complement in $\Sigma$ is a punctured sphere.

Each edge in $HT(\Sigma)$ is either type 0 or type 1. If $v,v'$ are vertices corresponding to cut systems $(\alpha_1, \alpha_2,...,\alpha_g)$ and $(\alpha'_1, \alpha'_2,...,\alpha'_g)$, respectively. There is a type 0 edge between $v$ and $v'$ if (up to reordering of either or both cut systems) $\alpha_i=\alpha'_i$ for $i>1$ and $\alpha_1\cap\alpha'_1=\emptyset$. (This relation is sometimes called a generalized handleslide.) Similarly, there is a type 1 edge between $v$ and $v'$ if (up to reordering of either or both cut systems) $\alpha_i=\alpha'_i$ for $i>1$ and $\alpha_1$ intersects $\alpha'_1$ transversely in a single point. 
\end{definition}

Let $\T=(X_1,X_2,X_3)$ be a trisection of a closed 4-manifold $X$. Let $\D=(\Sigma;\alpha,\beta,\gamma)$ be a diagram for $\T$. Each of $\alpha, \beta$, and $\gamma$ are cut systems for $\Sigma$, so correspond to vertices $v_\alpha, v_\beta$, and $v_\gamma$ of $HT(\Sigma)$. We observe that any two cut systems related by slides correspond to vertices in $HT(\Sigma)$ connected by a path of type 0 edges. Let $HT^0(\Sigma)$ be the complex obtained from $HT(\Sigma)$ by deleting all type 1 edges (leaving only type 0 edges). A trisection then naturally gives rise to three connected subgraphs of the cut complex $HT(\Sigma)$, which we denote by $\Gamma_\alpha$, $\Gamma_\beta$, $\Gamma_\gamma$, where $\Gamma_*$ is the component of $HT^0(\Sigma)$ containing $v_*$.

\begin{figure}
	\labellist
		\pinlabel \rotatebox{-90}{\resizebox{8pt}{.8in}{$\}$}} at 90 5
		\pinlabel {$k_1$} at 90 -10
		\pinlabel \rotatebox{-90}{\resizebox{8pt}{.8in}{$\}$}} at 257 5
		\pinlabel {$g-k_1$} at 257 -10
	\endlabellist
	\includegraphics[scale=.6]{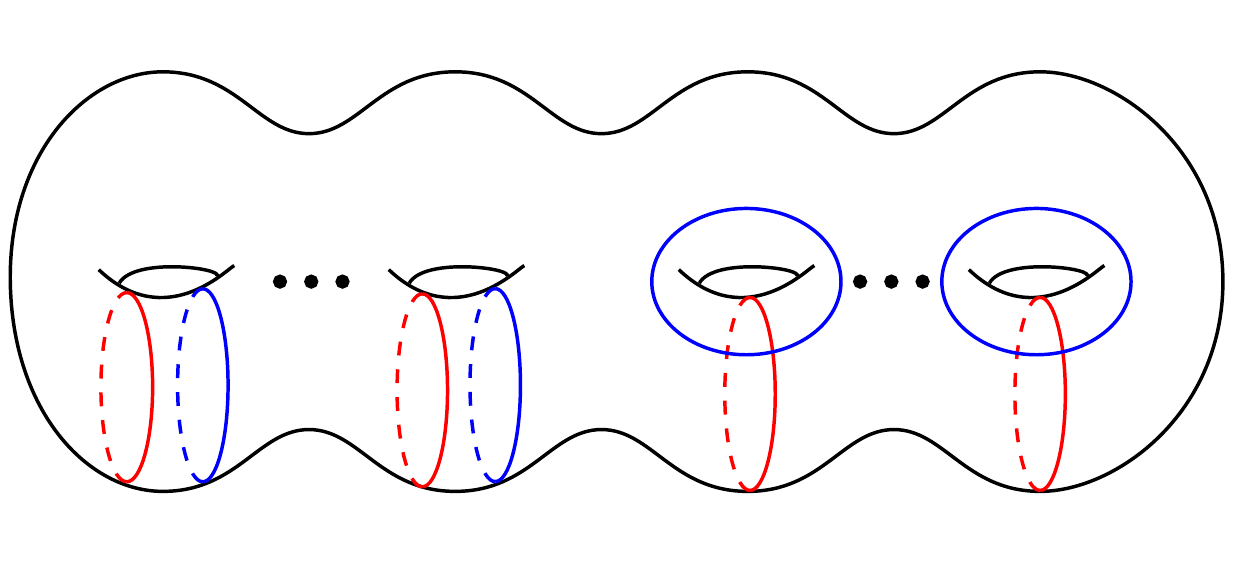}
	\caption{Standard genus $g$ Heegaard diagram for $\#_{k_1} S^1\times S^2$.}
	\label{F:HSStandard}
\end{figure}
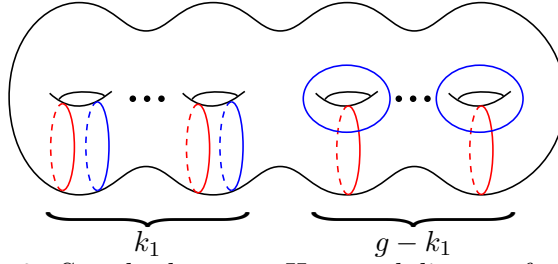

Recall that the 3-dimensional handlebodies in a trisection pairwise form Heegaard splittings for $\#_{k_i}S^1 \times S^2$. By Waldhausen's theorem~\cite{FW}, these Heegaard splittings have diagrams which, after handle slides and a diffeomorphism, can be made to look like the diagram in Figure~\ref{F:HSStandard} and in this position, one can use $g-{k_i}$ type 1 edges to pass between the vertices. Following~\cite{linvariant}, we call a pair of cut systems $\alpha_\beta \in \Gamma_\alpha$ and $\beta_\alpha \in \Gamma_\beta$ defining a genus $g$ Heegaard splitting of $\#_{k_1} S^1 \times S^2$ {\emph{good}} if they are connected by a path of exactly $g-{k_1}$ type 1 edges (and similarly for the pairs $(\beta,\gamma)$, $(\gamma,\alpha)$).

\begin{definition}
\label{def:unnormalizedL}
Let $X$ be a 4-manifold with $(g,k)$-trisection $\T$, and let $H_\alpha \cup H_\beta \cup H_\gamma$ be the spine of $\T$. Let $\delta$ be a loop in $HT(\Sigma)$. We say that $\delta$ is {\emph{valid}} with respect to $\T$ if $\delta$ includes (not necessarily distinct) vertices $\alpha_\gamma,\alpha_\beta,\beta_\alpha,\beta_\gamma,\gamma_\beta,\gamma_\alpha$ in cyclic order so that:
\begin{enumerate}
    \item The segment of $C$ between $\alpha_\gamma,\alpha_\beta$ (inclusive) lies in $\Gamma_\alpha$,
    \item The segment of $C$ between $\beta_\alpha,\beta_\gamma$ (inclusive) lies in $\Gamma_\beta$,
    \item The segment of $C$ between $\gamma_\beta,\gamma_\alpha$ (inclusive) lies in $\Gamma_\gamma$,
    \item The cut systems associated to the pairs $(\alpha_\beta,\beta_\alpha), (\beta_\gamma,\gamma_\beta), (\gamma_\alpha,\alpha_\gamma)$ are all good pairs, and the edges of $\delta$ between these pairs are all type 1.
\end{enumerate}

We define $l_{X,\T}$ to be the length of the shortest loop in $HT(\Sigma)$ which is valid with respect to $\T$. 


    

We then appropriately normalize, taking \[\L_{X,\T} = l_{X,\T}-3g+k_1+k_2+k_3.\]
\end{definition}
To see why the normalization in the definition of $\L_{X,\T}$ is appropriate, we analyze how $\L_{X,\T}$ changes under stabilization of $\T$.

Say $\T$ and $\T'$ are trisections of closed 4-manifolds with triple intersection surfaces $\Sigma$ and $\Sigma'$, respectively. Suppose $\delta$ and $\delta'$ are loops in the cut complexes $HT(\Sigma)$, $HT(\Sigma')$ respectively, where $\delta$ is valid with respect to $\T$ and $\delta'$ is valid with respect to $\T'$. Let $\alpha_\gamma,\ldots,\gamma_\alpha$ and $\alpha'_\gamma,\ldots,\gamma'_\alpha$ be the distinguished vertices of $\delta,\delta'$ as in Definition~\ref{def:unnormalizedL}. Then we may find a loop in $HT(\Sigma\#\Sigma')$ valid with respect to $\T\#\T'$ (see Definition~\ref{stabdef}) in which each vertex splits into the disjoint union of a cut system for $\T$ and a cut system for $\T'$. We start at the vertex $\alpha_\gamma\sqcup\alpha'_\gamma$, by which we mean the vertex whose cut system corresponds to the union of the cut systems for $\alpha_\gamma$ and $\alpha'_\gamma$. We then add edges corresponding to those of $\delta$ between $\alpha_\gamma$ and $\alpha_\beta$, followed by edges corresponding to those of $\delta'$ between $\alpha'_\gamma$ and $\alpha'_\beta$. Then we add edges corresponding to those of $\delta$ between $\alpha_\beta$ and $\beta_\alpha$, and so on, ending with edges corresponding to the segment of $\delta'$ between $\gamma'_\alpha$ and $\alpha'_\gamma$. This loop has length the sum of the lengths of $\delta$ and $\delta'$, so we conclude $l_{X,\T\#\T'}\le l_{X,\T}+ l_{X,\T'}$. When $\T'$ is a genus-1 trisection of $S^4$ (so $\T\#\T'$ is a stabilization of $\T$), then $\delta'$ can be taken to be length two. See Figure~\ref{F:S4genus1}. If $\T$ is a $(g,k)$-trisection, this yields
\begin{figure}
    \centering
    \includegraphics[scale=.7]{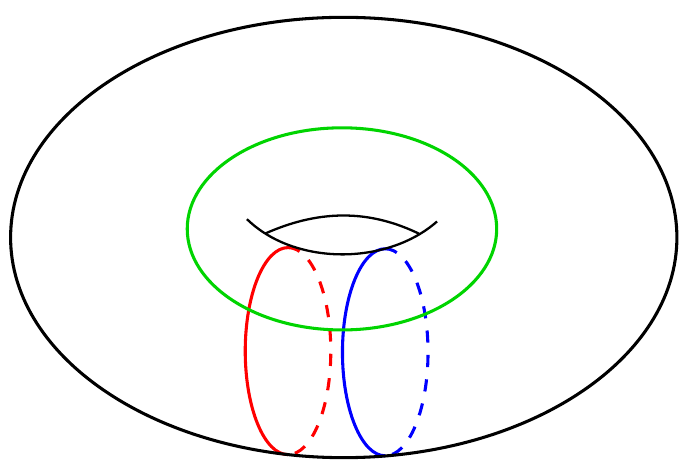}
    \caption{A genus $1$ trisection of $S^4$ whose corresponding path in the cut complex has length $2$.}
    \label{F:S4genus1}
\end{figure}
\begin{align*}
    l_{X,\T\#\T'}&\le l_{X,\T}+2\\
    \L_{X,\T\#\T'}+3(g+1)-k_1-k_2-k_3-1&\le \L_{X,\T}+3g-k_1-k_2-k_3+2\\
    \L_{X,\T\#\T'}&\le \L_{X,\T}.
\end{align*}
Thus, $L_{X,T}$ does not increase under stabilization. This should be kept in mind for the next definition.
 
\begin{definition}[{\cite[Definition 11]{linvariant}}]
Let $X$ be a smooth, closed, orientable 4-manifold. We define $\L(X) = \min_{\T}\{\L_{X,\T}\}$ where $\T$ ranges over all trisections of $X$.
\end{definition}

\section{Relative \texorpdfstring{$\L$}{L}-invariant}\label{sec2:relLdef}
\subsection{Definitions}
In this section, we define a non-negative integer invariant of a relatively trisected $4$-manifold $X$ with boundary (or just a $4$-manifold $X$ with boundary, by minimizing over trisections of $X$). We mirror the definition of the $\L$-invariant of a closed $4$-manifold from Section~\ref{sec:closedL}.

In the relative case, the individual sets of $(g-p)$ $\alpha, \beta$, and $\gamma$ curves of a relative trisection $(\Sigma;\alpha,\beta,\gamma)$ will generally not form a cut system of $\Sigma$, where by “cut system” we mean a set of curves and arcs on $\Sigma$ whose complement in $\Sigma$ is a disk. (In fact, $\Sigma\setminus\alpha$ must be a genus--$p$ surface with $b$ boundary components, so if $p>0$ or $b>1$, then $\alpha$ is certainly not a cut system for $\Sigma$.)  
Instead, the $\alpha$ curves along with some $2p+b-1$ disjoint arcs $A_\alpha$ may form a cut system for $\Sigma$. Similarly, we find sets of $l=2p+b-1$ arcs $A_\beta$ and $A_\gamma$ so that $\beta\cup A_\beta$ and $\gamma\cup A_\gamma$ are cut systems for $\Sigma$.

\begin{definition}
Let $(\Sigma;\alpha,\beta,\gamma)$ be a relative trisection diagram of trisection $\T$. Let $A_\alpha,A_\beta,A_\gamma$ be sets of disjoint, properly embedded arcs in $\Sigma$ with $A_\alpha\cap\alpha=A_\beta\cap\beta=A_\gamma\cap\gamma=\emptyset$ with the property that each of $\alpha\cup A_\alpha,\beta\cup A_\beta,\gamma\cup A_\gamma$ is a cut system for $\Sigma$. We call $(\Sigma;\alpha,\beta,\gamma;A_\alpha,A_\beta,A_\gamma)$ an {\emph{arced relative trisection diagram}} of $\T$.
\end{definition}

We implicitly find an arced relative trisection diagram of relative trisection $\T$ when performing the algorithm of Section~\ref{sec:monodromy} to determine the monodromy of $\O_\T$. 
This is our justification for why cut systems on $\Sigma$ consisting of $(g-p)$ closed curves and $l$ arcs are natural to consider when studying the relative trisection $\T$. 
We now construct a complex associated to a surface with boundary which is analogous to the cut complex of a closed surface as described in Section~\ref{sec:closedL}.

\begin{definition}
\label{def:HTpComplex}
Let $\Sigma$ be a compact orientable genus--$g$ surface with $b\ge 1$ boundary components. The {\emph{$p$-cut complex}} of $\Sigma$, $HT_p(\Sigma)$, is the simplicial complex built as follows.

Each vertex of $HT_p(\Sigma)$ corresponds to a cut system for $\Sigma$ consisting of a collection of $g-p$ non-separating simple closed curves on $\Sigma$ and $2p+b-1$ arcs.

Each edge in $HT_p(\Sigma)$ is either type 0, type $0^\boundary$, or type 1. If $v,v'$ are vertices corresponding to cut systems $(\alpha_1, \alpha_2,...,\alpha_{g-p},a_1,\ldots,a_{2p+b-1})$ and $(\alpha'_1, \alpha'_2,...,\alpha'_{g-p},a'_1,\ldots,a'_{2p+b-1})$, respectively (where $\alpha_i,\alpha'_i$ are closed curves and $a_j,a'_j$ are arcs), then:
\begin{enumerate}
    \item There is a type 0 edge between $v$ and $v'$ if (up to reordering of either or both cut systems) $\alpha_i=\alpha'_i,a_j=a'_j$ for $i>1$ and all $j$ and $\alpha_1\cap\alpha'_1=\emptyset$.
    \item There is a type $0^\boundary$ edge between $v$ and $v'$ if (up to reordering either or both cut systems) $\alpha_i=\alpha'_i,a_j=a'_j$ for all $i$ and all $j>1$ and $a_1,a'_1$ are disjoint in their interiors. We require $\boundary a_1=\boundary a'_1$, and near the two boundary points of $a_1$ a normal framing to $a_1$ must either point toward or away $a'_1$ (i.e. not toward $a'_1$ near one boundary and away at the other).
    \item There is a type 1 edge between $v$ and $v'$ if (up to reordering of either or both cut systems) $\alpha_i=\alpha'_i,a_j=a'_j$ for $i>1$ and all $j$ and $\alpha_1$ intersects $\alpha'_1$ transversely in a single point.
\end{enumerate}
Note that two vertices whose arcs have different endpoints reside in different connected components of $HT_p(\Sigma)$, since vertices connected by edges in $HT_p(\Sigma)$ necessarily correspond to cut systems whose arcs have common boundary.
\end{definition}

Let $\D=(\Sigma;\alpha,\beta,\gamma)$ be a relative trisection diagram of relative trisection $\T$. By the definition of a relative trisection, one can find a sequence of handle slides (i.e. type 0 moves) of each pair of $\{\alpha,\beta,\gamma\}$ so that they become standard, i.e. homeomorphic to the curves in Figure~\ref{F:standard}. (Note that we do {\emph{not}} claim that the three pairs can be made {\emph{simultaneously}} standard; rather, any pair can be made standard while ignoring the third set of curves.)

Choose $2p+b-1$-tuples of arcs $A_\alpha,A_\beta,A_\gamma$ so that $(\Sigma;\alpha,\beta,\gamma;A_\alpha,A_\beta,A_\gamma)$ is an arced relative trisection diagram $\D_A$. 
Let $v_\alpha,v_\beta,v_\gamma$ be the vertices of $HT_p(\Sigma)$ corresponding to cut systems $\alpha\cup A_\alpha, \beta\cup A_\beta,\gamma\cup A_\gamma$, respectively. Let $\Gamma_\alpha$ be the set of all vertices in $HT_p(\Sigma)$ which are connected to $v_\alpha$ by a path consisting of only type 0 and type 0$^\boundary$ edges. Similarly define $\Gamma_\beta$ and $\Gamma_\gamma$. 

We call a pair of vertices $\alpha_\beta \in \Gamma_\alpha$ and $\beta_\alpha \in \Gamma_\beta$ {\emph{good}} if they are connected by a path of exactly $g+p+b-k-1$ type 1 edges. Note that this is the number of dual $\alpha$ and $\beta$ curves in the two cut systems, so that this is the minimum possible number of type 1 edges we could hope to find in a path between $\alpha_\beta$ and $\beta_\alpha$.

\begin{definition}\label{def:relvalid}
We say a path $\delta$ in $HT_p(\Sigma)$ is {\emph{valid with respect to $\T$}} if $\delta$ includes vertices $v_1,v_2,v_3,v_4,v_5,v_6,v_7$ in order (if distinct) with $\delta$ beginning at $v_1$ and ending at $v_7$ so that:
\begin{enumerate}
    \item The segment of $C$ between $v_1,v_2$ (inclusive) lies in $\Gamma_\alpha$,
    \item The segment of $C$ between $v_3,v_4$ (inclusive) lies in $\Gamma_\beta$,
    \item The segment of $C$ between $v_5,v_6$ (inclusive) lies in $\Gamma_\gamma$,
    \item We have $v_7\in \Gamma_\alpha$ and the closed curves of the cut systems corresponding to $v_1$ and $v_7$ agree,
    \item The cut systems associated to the pairs $(v_2,v_3), (v_4,v_5), (v_6,v_7)$ are all good pairs and the edges in $\delta$ between these pairs consist of type 1 edges.
\end{enumerate}

We say that $\delta$ is {\emph{valid with respect to $\D_A$}} or that $\D_A$ {\emph{represents}} $\delta$ if $v_1$, $v_3$, and $v_5$ correspond to $\alpha\cup A_\alpha, \beta\cup A_\beta$, and $\gamma\cup A_\gamma$, respectively. We say that $\delta$ is {\emph{valid with respect to $\D$}} if $\delta$ is valid with respect to any arced relative trisection diagram extending $\D$. We note that not every arced trisection diagram has a valid path, as not every cut system is part of a good pair. However, the following proposition shows that any relative trisection diagram (without arcs) admits a valid path.

\end{definition}

\begin{proposition}\label{prop:validpath}
Let $\mathcal{D}=(\Sigma;\alpha,\beta,\gamma)$ be a relative trisection diagram. For any choices of arcs $A_{\alpha}$ for $\alpha$, there exist arcs $A_\beta,A_\gamma$ so that $\mathcal{D}_A:=(\Sigma;\alpha,\beta,\gamma;A_\alpha,A_\beta,A_\gamma)$ is an arced relative trisection diagram and so that there exists a path $\delta$ in $HT_p(\Sigma)$ that is valid with respect to $\mathcal{D}_A$.
\end{proposition}

\begin{proof}
We describe $\delta$ as a sequence of slides and type 1 moves.

Let $A_\alpha$ be any set of arcs $\Sigma\setminus \alpha$ so that $\alpha\cup A_\alpha$ are a cut system for $\Sigma$. By the definition of a relative trisection diagram, there exist curves $\alpha'$ and $\beta'$ slide-equivalent to $\alpha$ and $\beta$ (respectively) so that $(\alpha',\beta')$ are a good pair. Perform slides on $\alpha,A_\alpha$ to transform $\alpha$ into $\alpha'$ (this may require us to perform some slides of $A_\alpha$ arcs, turning them into arcs $A_{\alpha'}$). Then do $g+p+b-k-1$ type 1 moves to replace $\alpha'$ with $\beta'$. Perform slides on $\beta',A_\beta'$ to transform $\beta'$ into $\beta$ and $A_\beta'$ into arcs $A_{\beta}$.

Repeat the above procedure with $\beta,\gamma$ taking the roles of $\alpha,\beta$. That is, slide $\beta$ and $A_{\beta}$, do $g+p+b-k-1$ type 1 moves, and perform more slides to obtain $\gamma$ and arcs $A_{\gamma}$. Finally, repeat this procedure one more time with $\gamma,\alpha$: slide $\gamma;A_{\gamma}$, do $g+p+b-k-1$ type 1 moves, and perform more slides to obtain $\alpha$ and arcs $A_{\overline{\alpha}}$.

This sequence of moves describes a path $\delta$ in $HT_p(\Sigma)$ that is valid with respect to $(\Sigma;\alpha,\beta,\gamma;A_\alpha,A_\beta,A_\gamma)$.

\end{proof}

In Definition~\ref{def:relvalid}, we label the distinguished vertices with numbers rather than $\alpha,\beta,\gamma$ (as in Definition~\ref{def:unnormalizedL}) to avoid giving the impression that the path $\delta$ is a closed loop. In Definition~\ref{def:relvalid}, we generally cannot hope for $\delta$ to be a closed loop, as the arcs corresponding to $v_7$ should differ from those corresponding to $v_1$ by an application of the monodromy of $\O_\D$.

\begin{remark}
Note that a path $\delta$ which is valid with respect to $\D_A$ roughly corresponds to performing the monodromy algorithm of~\cite{cgp} as described in Section~\ref{sec:monodromy}, as we begin with a cut system (choice of arcs) for the $\alpha$-page $\Sigma_\alpha$ of $\O_\D$ and by sliding the arcs obtain cut systems for $\Sigma_\beta$, $\Sigma_\gamma$, and  $\Sigma_\alpha$ again. This algorithm is the primary motivation for the definition of a valid path. 
\end{remark}

\begin{definition}\label{def:rL}
We define the {\emph{relative $\L$-invariant}} of a relative trisection diagram to be $$\rL(\D)=\min\{|\delta|\mid \delta\text{ valid with respect to $\D$ }\}-3(g+p+b-1)+(k_1+k_2+k_3).\label{Leq}$$ We define the relative $\L$-invariant of a relative trisection $\T$ to be $$\rL(\T)=\min\{\rL(\D)\mid\D\text{ is a relative trisection diagram for $\T$}\}.$$ Similarly, we define the relative $\L$-invariant of a bounded $4$-manifold $X$ to be $$\rL(X)=\min\{\rL(\T)\mid\T\text{ is a relative trisection of $X$}\}.$$
Proposition~\ref{prop:validpath} ensures that these quantities are all well-defined (i.e. that we are not taking the minimum value of an empty set).
\end{definition}
When $\rL(X)=\rL(\D_A)=|\delta|-3(g+p+b-1)+(k_1+k_2+k_3)$ for some path $\delta$ representing arced relative trisection diagram $\D_A$, we say that $(\D_A,\delta)$ {\emph{achieve}} $\rL(X)$, as a convenient shorthand.

In Definition~\eqref{Leq}, $|\delta|$ refers to the length (number of edges) in the path $\delta$. The constant $3(g+p+b-1)-(k_1+k_2+k_3)$ is the minimum number of type 1 edges which must be in $\delta$ for algebraic reasons; note that up to slides the pair $\alpha,\beta$ consist of $k_1-2p-b+1$ pairs of parallel curves and $g+p+b-k_1-1$ pairs of dual curves (and similarly for the pairs $\beta,\gamma$ and $\gamma,\alpha$). As in Section~\ref{sec:closedL}, this normalization ensures that interior stabilization does not increase $\rL(\T)$.

\begin{remark}\label{rem:empty}
Let $\T$ be the $(0,0;0,1)$-relative trisection of $B^4$. Note that $HT_0(D^2)$ is the empty complex. The empty path $\delta$ is valid with respect to $\T$. Moreover, $|\delta|-3(0+0+1-1)+(0+0+0)=0$, so we conclude that $\rL(\T)=0$, so $\rL(B^4)=0$.
\end{remark}

\begin{proposition}\label{prop:relintstab}
Let $\T$ be a $(g,k;p,b)$-relative trisection of $X^4$. Let $\widetilde{T}$ be a $(g+1,\tilde{k};p,b)$-relative trisection obtained from $\T$ by one interior stabilization. Then $\rL(\widetilde{T})\le \rL(\T)$.
\end{proposition}
\begin{proof}
Let $\D_A=(\Sigma;\alpha,\beta,\gamma;A_\alpha,A_\beta,A_\gamma)$ be an arced relative trisection diagram for $\T$. Let  $D'=(\Sigma',\alpha',\beta',\gamma')$ be a trisection diagram of a $(1,k)$-trisection $\T'$ of $S^4$. 

Let $\delta$ be a path in $HT_p(\Sigma)$ which is valid with respect to $\D_A$. Choose $\delta$ and $\D_A$ so that $(\D_A,\delta)$ achieves $\rL(\T)$ (i.e. $\rL(\T)=|\delta|-3(g+p+b-1)+(k_1+k_2+k_3)$). Let $v_1,v_2,\ldots, v_7$ be the distinguished vertices of $\delta$ as in Definition~\ref{def:rL}. Recall that $v_1$, $v_3$, $v_5$ correspond to the cut systems $\alpha\cup A_\alpha, \beta\cup A_\beta, \gamma\cup A_\gamma$.

Let $\delta'$ be a loop in $HT(\Sigma')$ which is valid with respect to $\D'$. Take $\delta'$ specifically to be the length-2 loop implicitly described in Figure~\ref{F:S4genus1}; assume $k_1=1,k_2=k_3=0$ (up to reordering $\alpha',\beta',\gamma'$). Say the vertices of $\delta'$ are $v'_1,v'_2,v'_3$, where so $v'_1,v'_3$ correspond to $\alpha'$. Let $v'_\beta$ be the first of $v'_1,v'_2,v'_3$ to correspond to $\beta'$ and $v'_\gamma$ the first to correspond to $\gamma'$.

Then we may find a loop in $HT_p(\Sigma\#\Sigma')$ valid with respect to $\T\#\T'$ 
in which each vertex splits into the disjoint union of a cut system for $\Sigma$ and a cut system for $\Sigma'$. We start at the vertex $v_1\sqcup v'_1$, by which we mean the vertex whose cut system corresponds to the union of the cut systems for $v_1$ and $v'_1$. We then add edges corresponding to those of $\delta$ between $v_1$ and $v_2$ (if any), followed by edges corresponding to those of $\delta'$ between $v'_1$ and $v_\beta$ (if any). Then we add edges corresponding to those of $\delta$ between $v_2$ and $v_3$, and so on, ending with an edge corresponding to the segment of $\delta'$ between $v_\gamma$ and $v_3$ (if nonempty). This path $\tilde{\delta}$ has length the sum of the lengths of $\delta$ and $\delta'$, namely $|\delta|+2$. Since $\D\#\D'$ is a relative trisection diagram for $\widetilde{T}$, we thus conclude
\begin{align*}
    \rL(\widetilde{\T})&\le|\tilde{\delta}|-3((g+1)+p+b-1)+\Sigma \tilde{k}_i\\&= (|\delta|+2)-3(g+p+b-1)-3+(\Sigma k_i+1)\\
    &=|\delta|-3(g+p+b-1)+\Sigma k_i\\
    &=\rL(\T).
\end{align*}
    
\end{proof}





\begin{definition}
The \emph{boundary complexity} of the relative trisection diagram $\mathcal{D}$ is the non-negative integer
	$$\rL^\partial(\mathcal{D})=\min\{\#\text{of type 0$^{\boundary}$ edges in $\delta\mid\delta$ valid with respect to $\D$}\}.$$
	
	The {\emph{boundary complexity}} of the $4$-manifold $X$ is the non-negative integer \[\rL^\partial(X)=\min\{\rL^\partial(\mathcal{D})\mid\D\text{ is a relative trisection diagram for }X\}.\]

\end{definition}
When $\rL^\boundary(X)=\rL^\boundary(\D_A)=\#\text{type 0$^{\boundary}$ edges in }\delta$ for some path $\delta$ representing arced relative trisection diagram $\D_A$, we say that $(\D_A,\delta)$ {\emph{achieves}} $\rL^\boundary(X)$, as a convenient shorthand. When $(\D_A,\delta)$ achieves $\rL^\boundary(X)$, we need not expect the pair also achieves $\rL(X)$.

\begin{proposition}\label{prop:bdrys3}
If $\boundary X\cong S^3$, then $\rL^\partial(X)=0$.
\end{proposition}
\begin{proof}
By work of the first author~\cite{nickthesis}, there is a $(g,k;0,1)$-relative trisection $\T$ of $X$. A cut system of any diagram $\D=(\Sigma;\alpha,\beta,\gamma)$ of $\T$ is comprised only of closed curves. That is, $HT_p(\Sigma)$ does not include any type 0$^\boundary$ edges. Thus, $\rL^\partial(\T)=0$.
\end{proof}
We will later see a converse to this in Proposition~\ref{L:0not1}. That is, if $rL^\partial(X)=0$, then $\boundary X\cong\#_n(S^1\times S^2)$ for some $n\ge 0$.

\begin{definition}
The \emph{interior complexity} of the $(g,k;p,b)$-relative trisection diagram $\mathcal{D}$ is the non-negative integer
	\begin{align*}
	    \rL^\circ(\mathcal{D})&=\min\{|\delta|-\#(\text{type 0$^\boundary$ edges in $\delta$})\mid\delta\text{ valid with respect to }\D\}\\ &\hphantom{\hspace{1in}}-3(g+p+b-1)+(k_1+k_2+k_3).
	    \end{align*}
	The {\emph{interior complexity}} of the $4$-manifold $X$ is the non-negative integer \[\rL^\circ(X)=\min\{\rL^\circ(\mathcal{D})\mid\D\text{ is a relative trisection diagram for }X\}.\]\
\end{definition}

We note that if a pair $(\D_A,\delta)$ achieves $\rL(X)$, then the valid path $\delta$ consists of edges which count  for either $\rL^\circ(X)$ or for $\rL^\partial(X)$. We therefore immediately obtain the inequality $\rL(X) \leq \rL^\circ(X)+\rL^\partial(X)$. On the other hand, it is plausible that the path realizing $\rL(X)$ minimizes neither $rL^\circ(X)$ nor $\rL^\partial(X)$, motivating the following question.

\begin{question}\label{separatequestion}
Given a smooth, compact $4$-manifold $X$ with connected, nonempty boundary, must $\rL(X)=\rL^\circ(X)+\rL^\partial(X)$?
\end{question}

\begin{remark}
It is simple to see that $$\rL(X)=\rL(-X),\quad\rL^\boundary(X)=\rL^\boundary(-X),\quad\rL^\circ(X)=\rL^\circ(-X).$$ This holds because if $\D=(\Sigma;\alpha,\beta,\gamma)$ is a relative trisection diagram for $X$, then $-\D=(\Sigma,\gamma,\beta,\alpha)$ is a relative trisection diagram for $-X$. If $\delta\in HT_p(\Sigma)$ is a valid path with respect to $\D$, then the reverse of $\delta$ is a valid path in $HT_p(\Sigma)$ for $-\D$.
\end{remark}

\begin{proposition}
Let $\T'$ be obtained from $(g,k;p,b)$-relative trisection $\T$ by interior stabilization. 
Then:
$$\rL(\T')\le\rL(\T),\quad\rL^\boundary(\T')\le\rL^\boundary(\T),\quad\rL^\circ(\T')\le\rL^\circ(\T).$$
\end{proposition}
\begin{proof}
We showed in Proposition~\ref{prop:relintstab} that $\rL(\T')\le\rL(\T)$. We proved this by showing that for any diagram $\D$ of $\T$ and valid path $\delta$ with respect to $\D$, we may obtain a diagram $\D'$ of $\T'$ with valid path $\delta'$, where $|\delta'|=|\delta|+2$ and $\delta'$ has the same number of type 0 and type 0$^\boundary$ edges as does $\delta$ (the two "extra" edges are type 1). This immediately yields $\rL^\boundary(\T')\le\rL^\boundary(\T)$ and $\rL^\circ(\T')\le\rL^\circ(\T)$.
\end{proof}


\begin{remark}

If $\T'$ is obtained from $\T$ by a {\emph{relative}} stabilization or relative double twist, then it is possible that $\rL(\T')>\rL(\T)$. For example, let $\T$ have diagram $\D=(D^2,\emptyset,\emptyset,\emptyset)$, so $\T$ is the $(0,0;0,1)$-relative trisection of $B^4$. We have $\rL(\T)=0$ (see Remark \ref{rem:empty}). 
Let $\T'$ be the $(1,1;0,2)$-relative trisection of $B^4$, obtained from $\T$ by one relative stabilization. The open book $\O_{\T'}$ has binding the Hopf link. In Lemma~\ref{L:0not1}, we will show that if $\rL^\partial(T')=0$ then $\partial B^4\cong S^1\times S^2$, a contradiction. Therefore, $\rL(\T')\ge \rL^\partial(\T')>0=\rL(\T).$

As another example, let $\T''$ be a $(2,2;0,3)$-relative trisection of $B^4$ obtained from $\T$ by a relative double twist. Again by Lemma~\ref{L:0not1}, we must have $\rL^\partial(\T'')>0=\rL(\T)$. (In fact, by Theorem~\ref{thm:boundarysmall}, we must have $\rL^\partial(\T'')>3$, or else $\boundary B^4$ would admit an $S^1\times S^2$ summand.) Applying this theorem, we conclude that $\rL(\T'')=4>\rL(\T')=2>\rL(\T)=0.$)

Compare this to Proposition~\ref{prop:relintstab}, which shows that {\emph{interior}} stabilization cannot increase $\rL(\T)$.
\end{remark}

\begin{example}[Figures~\ref{F:wrinkling} and~\ref{F:lantern}]\label{example}
Consider the relative trisection diagram $\D=(\Sigma;\alpha,\beta,\gamma)$ shown in Figure~\ref{F:wrinkling} (B) obtained from the positive allowable Lefschetz fibration $f:W \rightarrow D^2$ of the 4-manifold $W\cong\overline{\CP}^2\setminus\mathring{B}^4$ with regular fiber the thrice punctured disk and vanishing cycles $a,b,c,d$, shown in Figure~\ref{F:wrinkling}(A). If we denote $\tau_a$ as the positive Dehn twist about the curve $a$, the induced open book is $(S_{0,4}, \phi)$ where $\phi=\tau_a\tau_b\tau_c\tau_d$. The steps in the monodromy algorithm are shown in Figure~\ref{F:lantern}; these steps describe a path $\delta$ in $HT_0(\Sigma)$ which is valid with respect to the $(4,3;0,4)$-relative trisection diagram $\D$.

The path $\delta$ consists of $12$ type 1 edges and $6$ type 0$^{\boundary}$ edges. We thus have $\rL(\D)\le18-3(4+0+4-1)+9=6$, $\rL^\boundary(\D)\le 6$, and $\rL^\circ(\D)=0$. We will see in Corollary~\ref{cor:exsharp} that $\rL(\D)=\rL^\boundary(\D)=6$ (using the fact that $\boundary W$ is a 3-sphere). On the other hand, by considering a $(1,0;0,1)$-relative trisection of $W$, we find $\rL(W)=\rL^\boundary(W)=0$.
\begin{figure}\centering
\captionsetup{width=\textwidth}
\subcaptionbox{The vanishing cycles of a Lefschetz fibration}{
	\labellist
		\pinlabel $a$ at 215 230
		\pinlabel $b$ at 190 100
		\pinlabel $c$ at 240 100
		\pinlabel $d$ at 55 277
	\endlabellist
\includegraphics[width=2in,height=2in]{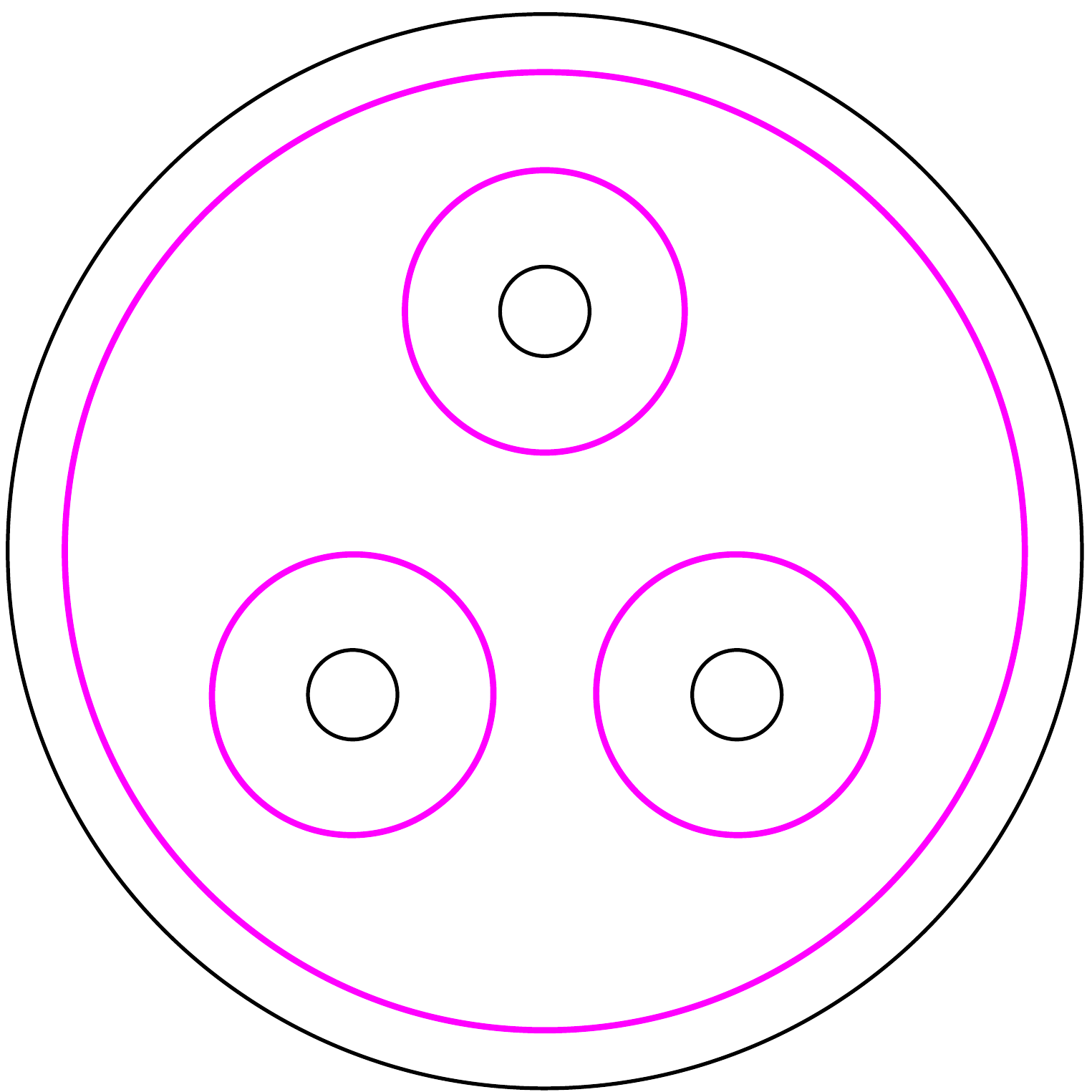}}
\hspace{2cm}
\subcaptionbox{Wrinkling the Lefschetz singularities to obtain a relative trisection diagram}{\includegraphics[width=2in,height=2in]{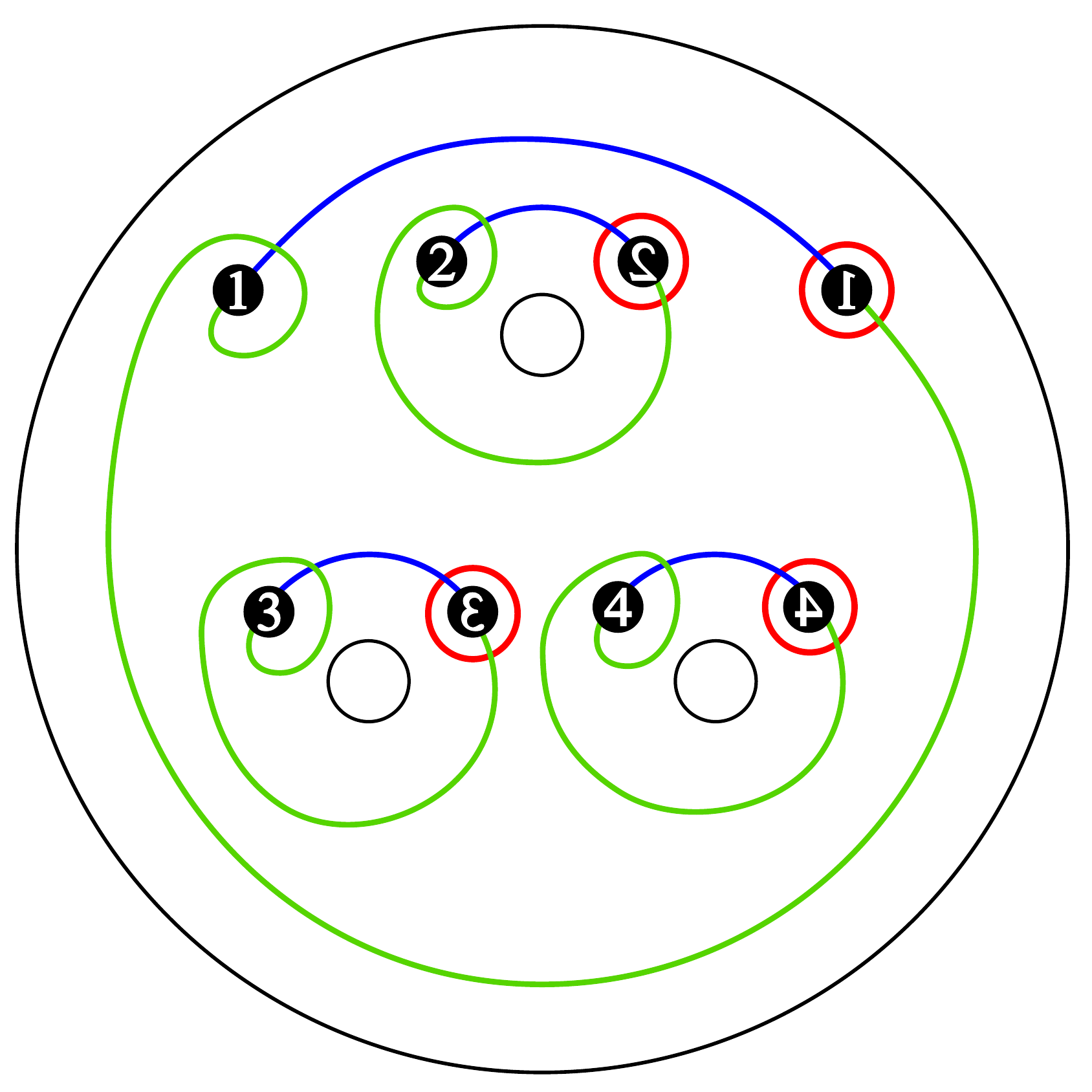}}
\caption{Wrinkling the Lefschetz singularities to obtain a relative trisection diagram.}
\label{F:wrinkling}
\end{figure}
\begin{figure}\centering
\captionsetup{width=\textwidth}
\subcaptionbox{}
{\includegraphics[width=2in,height=2in]{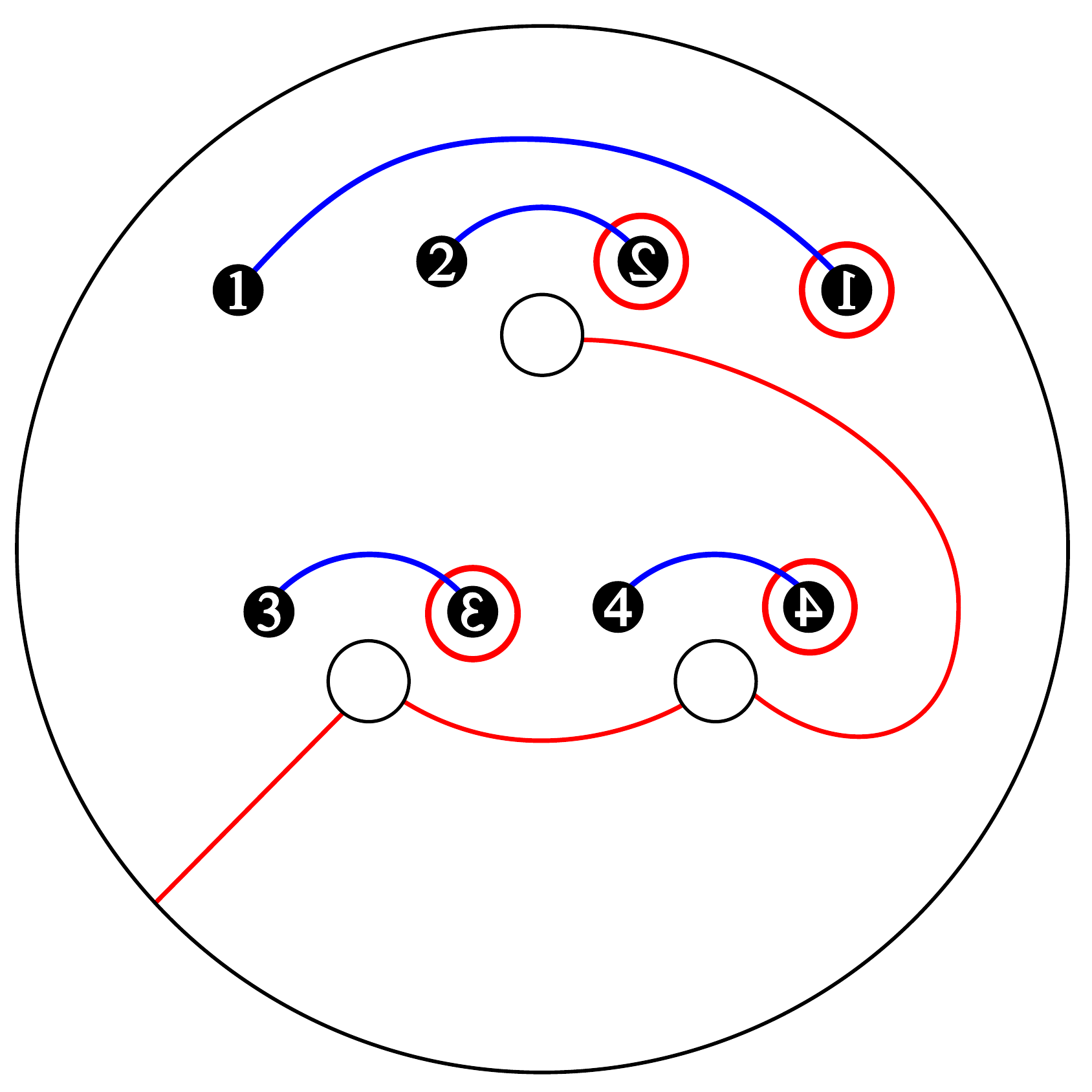}}%
\hspace{2cm}
\subcaptionbox{}
{\includegraphics[width=2in,height=2in]{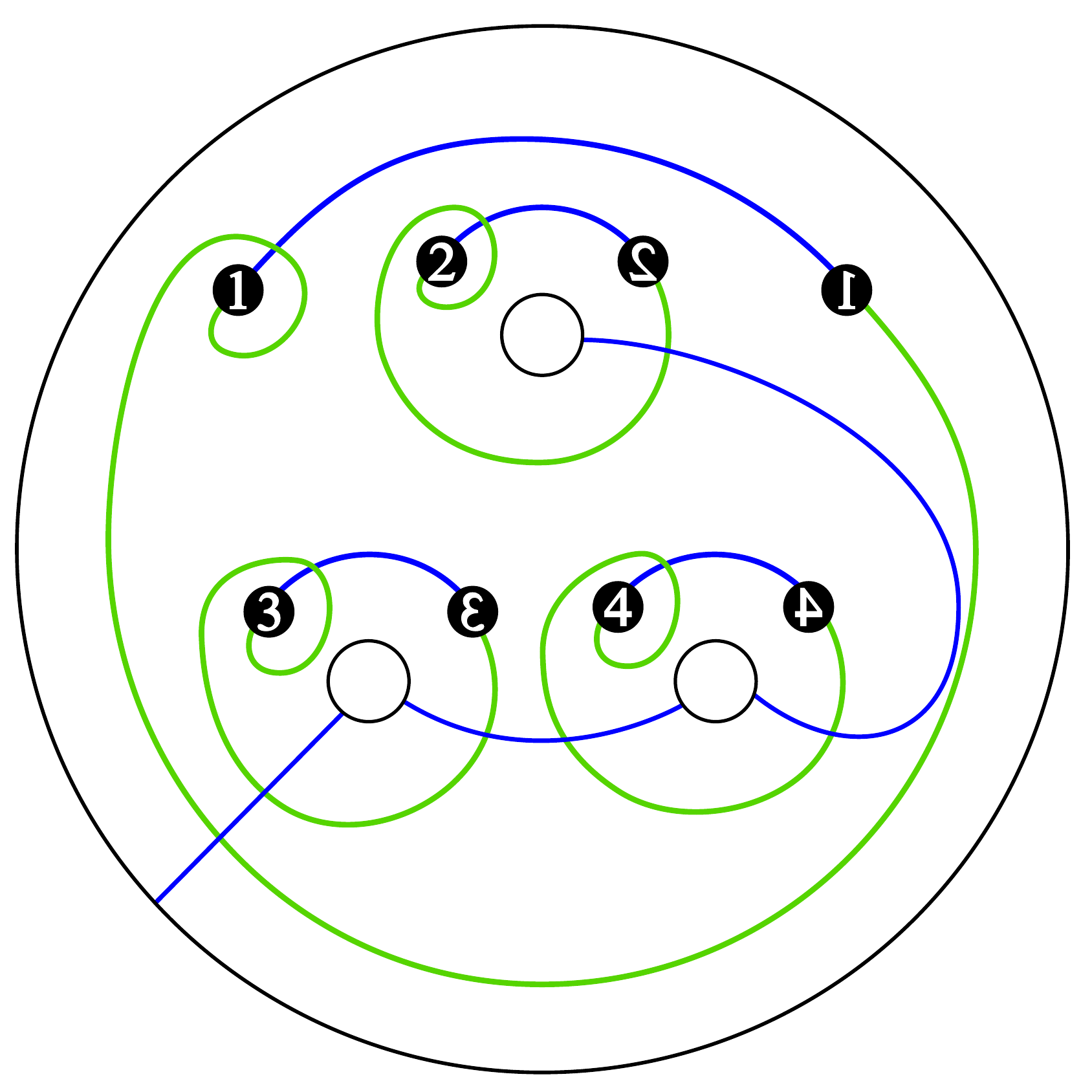}}%

\subcaptionbox{}
{\includegraphics[width=2in,height=2in]{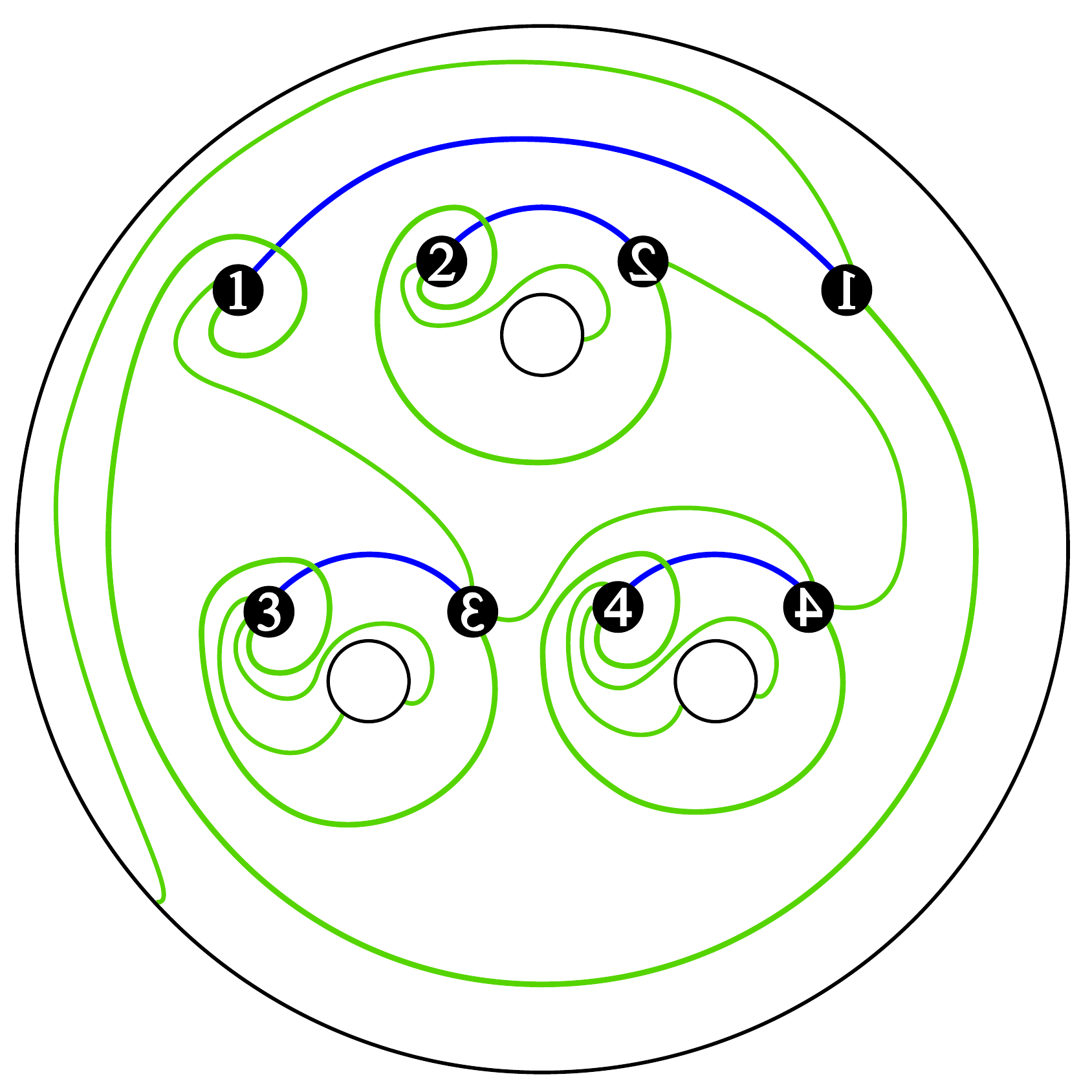}}%
\hspace{2cm}
\subcaptionbox{}
{\includegraphics[width=2in,height=2in]{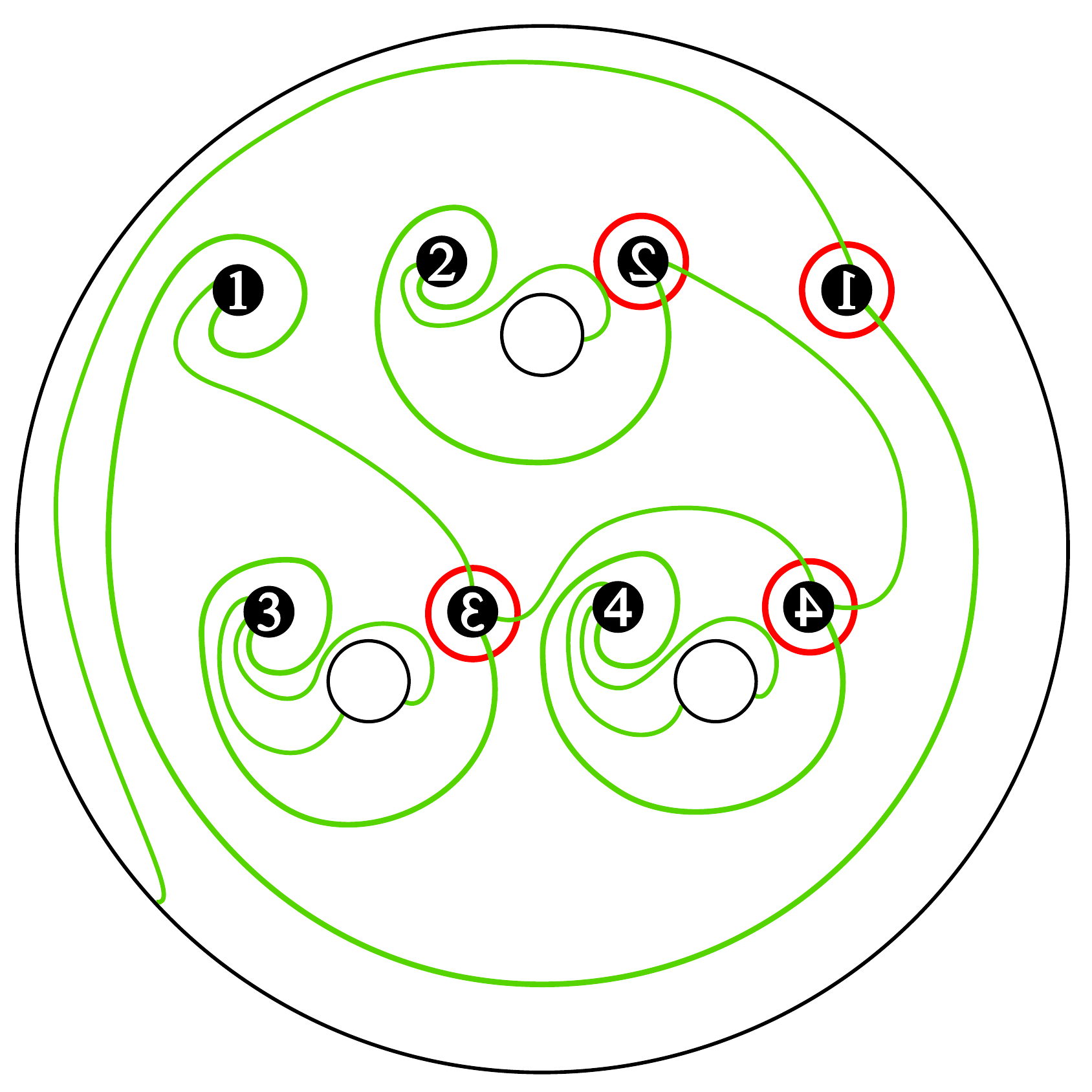}}%

\subcaptionbox{}
{\includegraphics[width=2in,height=2in]{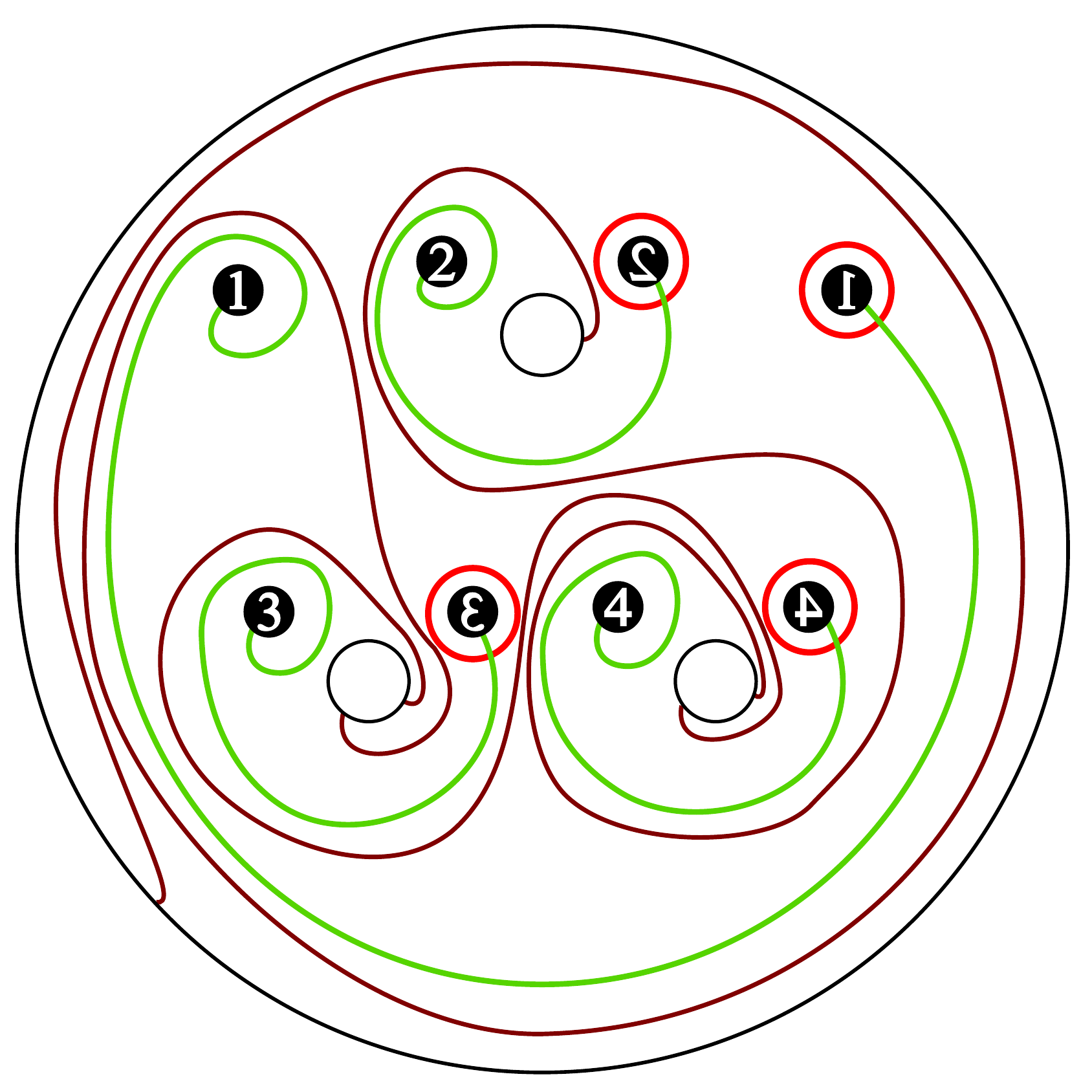}}%
\hspace{2cm}
\subcaptionbox{}
{\includegraphics[width=2in,height=2in]{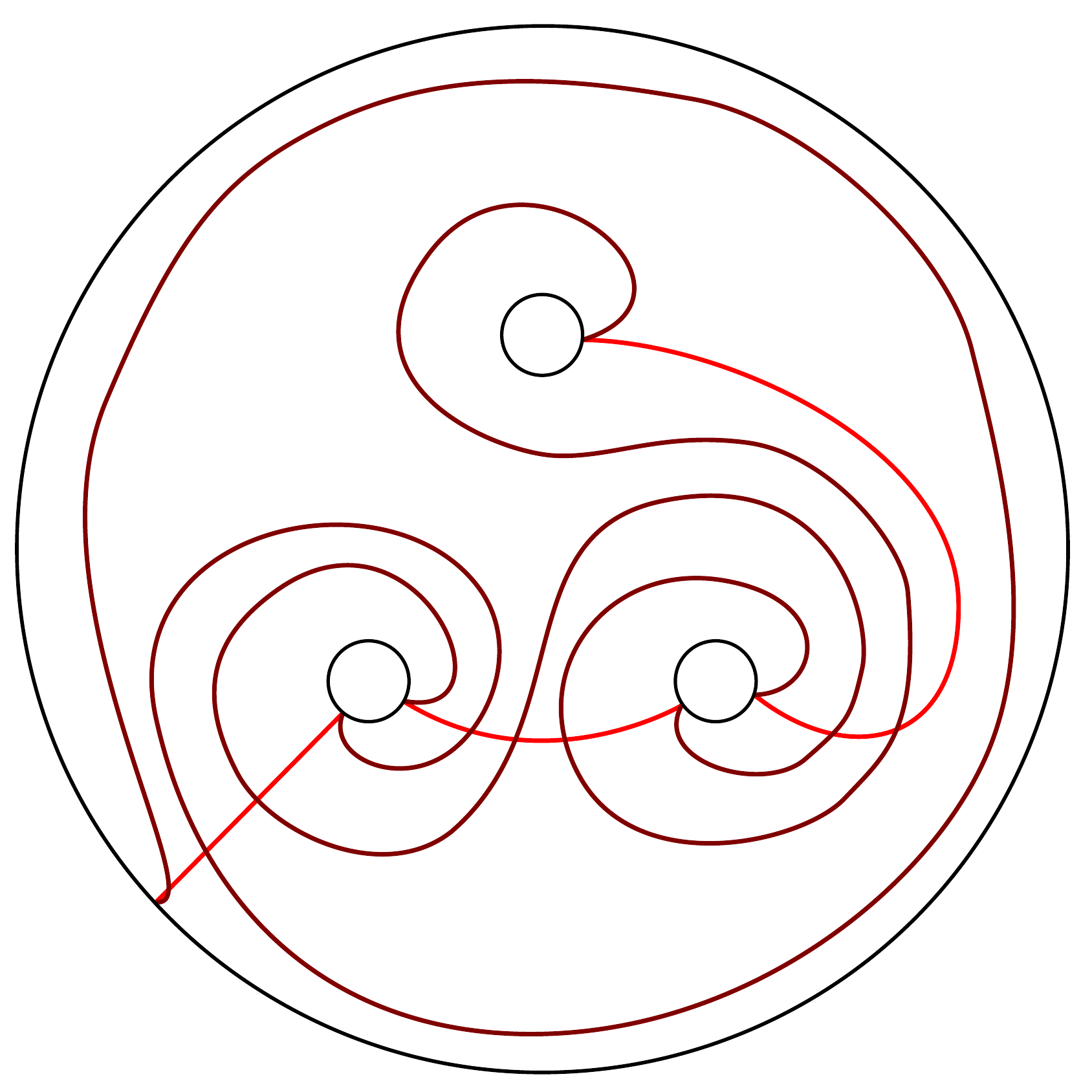}}\caption{A relative trisection diagram $\D$ and a path $\delta\in HT_p(\Sigma)$ valid with respect to $\D$. See Example~\ref{example} for a more detailed caption.}\label{F:lantern}
\end{figure}
Here we provide a more detailed caption for Figure~\ref{F:lantern}.

\begin{enumerate}
    \item[(A)] $\mathcal{A}_\alpha$ is disjoint from both $\alpha$ and $\beta$. The segment of $\delta$ between the vertices corresponding to $\alpha\cup A_\alpha$ and $\beta\cup A_\alpha$ consists of four type 1 edges.
    \item[(B)] We must slide $\mathcal{A}_\beta=A_\alpha$ over $\beta$ curves so that they are disjoint from $\gamma$. These slides correspond to three type 0$^{\boundary}$ edges in $\delta$.
    \item[(C)] $\mathcal{A}_\gamma$ viewed as arcs in $\Sigma$. One generalized arc slide was done to each arc in $\mathcal{A}_\beta$ to obtain $\mathcal{A}_\gamma$. Replacing $\beta$ with $\gamma$ corresponds to four type 1 edges in $\delta$.
    \item[(D)] We must slide $\mathcal{A}_\gamma$ over $\gamma$ so that they are disjoint from $\alpha$ to obtain $\overline{\mathcal{A}}$ These slides correspond to three type 0$^{\boundary}$ edges in $\delta$.
    \item[(E)] One generalized arc slide was done to each arc in $\mathcal{A}_\gamma$ to obtain $\overline{\mathcal{A}}$. Replacing $\gamma$ with $\alpha$ corresponds to four type 1 edges in $\delta$.
    \item[(F)] The monodromy $\phi:\Sigma_\alpha\to\Sigma_\alpha$ induced by the relative trisection. Specifically, we draw $A_\alpha$ and $\phi(A_\alpha)$.
\end{enumerate} 
\end{example}

\subsection{Gluing and comparison with the \texorpdfstring{$\L$}{L}-invariant for closed manifolds}\label{gluingsection}
Let $\D=(\Sigma; \alpha, \beta,\gamma; A_\alpha, A_\beta, A_\gamma)$ and $\D'=(\Sigma', \alpha', \beta',\gamma'; A_{\alpha'}, A_{\beta'}, A_{\gamma'})$ be arced relative trisection diagrams which correspond to relatively trisected $4$-manifolds $X$ and $X'$ with connected boundaries $W$ and $W'$ respectively. (This can be generalized to 4-manifolds with multiple boundary components). Assume that $\O_\D$ and $\O_{\D'}$ are compatible; i.e. 
there is an orientation-reversing diffeomorphism from $W$ to $W'$ taking the pages of $\O_\D$ to the pages of $\O_{D'}$ setwise (rel boundary), and specifically taking the pages $\Sigma_\alpha,\Sigma_\beta$, and $\Sigma_\gamma$ to $\Sigma_{\alpha'},\Sigma_{\beta'}$, and $\Sigma_{\gamma'}$, respectively. 

Such a map defines a closed surface $S = \Sigma \cup_f \Sigma'$ along with  essential, simple, closed curves $\overline{\alpha}=\{\overline{\alpha}_1, \ldots, \overline{\alpha}_{2p+b-1}\},$ where $\overline{\alpha}_i:=a_i\cup_fa'_i$. We will denote $\alpha''=\alpha \cup \overline{\alpha}\cup \alpha'$; similarly for $\beta''$ and $\gamma''$.

\begin{theorem}[{\cite{CO,nickthesis}}]

\label{thm:GluedDiagramCurves}
Let $X_{\T}$ and $X'_{\T'}$ be relatively trisected $4$-manifolds with connected boundary which induce compatible open book decompositions on their diffeomorphic boundaries. Then $(X\underset{\partial}{\cup}X', \T\underset{\partial}{\cup}\T')$ is a closed trisected $4$-manifold. Moreover, if $\D$ and $\D'$ as above are the corresponding compatible relative trisection diagrams, then $(S, \alpha'', \beta'', \gamma'')$ is the trisection diagram corresponding to $\T''=\T\cup\T'$.
\end{theorem}

\begin{lemma}\label{weirdslemma}
Let $\T$ and $\T'$ be relative trisections of $4$-manifolds $X_\T, X_{\T'}$ with $\boundary X_{\T}\cong \boundary X_{\T'}$. Assume $\T$ and $\T'$ are compatible, so they can be glued to obtain a trisection $\T''$ of $X_\T\cup X_{\T'}$. Let $S$ be the closed surface obtained by gluing the $\alpha$-pages of $\T$ and $\T'$, with cut system $s$ made by gluing the $\alpha$ arcs of $\T$ and $\T'$ along their common endpoints.

Then $$\L(\T'')\leq \rL(\T)+\rL(\T')+d(v_s,v_{\phi(s)}),$$ where $\phi$ is the monodromy of the open book $\O_{\T}$, $v_s$ and $v_{\phi(s)}$ are the vertices in the cut complex $HT(S)$ corresponding to cut systems $s$ and $\phi(s)$ (respectively), and $d$ denotes the distance in the cut complex.

\end{lemma}
\begin{proof}
Let $\delta$ and $\delta'$ be valid paths realizing $\rL(\T)$ and $\rL(\T')$ respectively, and let $\D=(\Sigma; \alpha, \beta,\gamma; A_\alpha, A_\beta, A_\gamma)$ and $\D'=(\Sigma', \alpha', \beta',\gamma'; A_{\alpha'}, A_{\beta'}, A_{\gamma'})$ be arced relative trisection diagrams for $\T$ and $\T'$ representing $\delta$ and $\delta'$, respectively. 
We form a trisection diagram $\D''$ for $\T''$ with trisection surface $\Sigma'' = \Sigma \cup \Sigma'$ as in Theorem \ref{thm:GluedDiagramCurves}. We can, in particular identify $A_\alpha$ with  $A_{\alpha'}$,  $A_\beta$ with  $A_{\beta'}$  and $A_\gamma$ with  $A_{\gamma'}$ to form some closed $\alpha'',\beta'',\gamma''$ curves in $\D''$. 
Now we form a path, $\delta''$, in $HT(\Sigma'')$ which is valid with respect to $\T''$. In rough language, the strategy is to take $\delta''$ to agree with $\delta$ on one side of $\Sigma''$ and $\delta'$ on the other half of $\Sigma''$.

Comparing Definitions \ref{def:HTComplex} and \ref{def:HTpComplex}, we see that each type $0$ or type $1$ edge of  $\delta$ or $\delta'$ corresponds directly to a type $0$ or type $1$ edge on the cut systems of $\Sigma''$. 
Moreover, a type $0^{\partial}$ edge of $\delta$ or $\delta'$ corresponds to removing an arc, $a$, and replacing it with an arc, $a'$, which misses all of the other curves and arcs. We can realize this change by a type $0$ edge $HT(\Sigma'')$ 
by removing a curve that contains $a$, and replacing the half of the curve corresponding to $a$ by $a'$ and leaving the other half fixed.

Since each edge in the paths $\delta$ and $\delta'$ corresponded to an edge in  $HT(\Sigma'')$, we may construct a path $\tilde{\delta}''$. by first following all of the edges corresponding to edges in $\delta$ and then following all of the edges corresponding to edges in $\delta'$. Note that $|\tilde{\delta}''|=\rL(\T)+\rL(\T')$. However, $\tilde{\delta}''$ is not a loop: the start and endpoints correspond to cut systems of $\Sigma$ that differ by an application of the monodromy of $\O_{\T}\cup\O_{\T'}$ on the subsurface $S:=\Sigma_{\alpha}\cup\Sigma_{\overline{\alpha}}$ of $\Sigma''$.

That is, $\alpha''=\alpha\cup \alpha'\cup s$ while $\overline{\alpha}''=\alpha\cup\alpha'\cup\phi(s)$.
Since $HT(S)$ is connected \cite{HatThu80}, there is a path from $s$ to $\phi(s)$ in $HT(S)$. Choose a path $\hat{\delta}''$ achieving minimum possible length ($d(v_s,v_{\phi(s)})$). Each slide or flip corresponding to an edge in $\hat{\delta}''$ can be lifted to a slide or flip of $s,\phi(s)$ in $\Sigma''$ disjoint from $\alpha\cup\alpha'$, so we may lift $\hat{\delta}''$ to a path in $HT(\Sigma'')$ from the endpoint of $\tilde{\delta}''$ to the starting point of $\tilde{\delta}''$. 
Therefore, the union $\delta''=\tilde{\delta}''\cup\hat{\delta}''$ is a valid path for $T''$ of length $\rL(\T)+\rL(\T')+d(v_s,v_{\phi(s)})$, proving the lemma.
\end{proof}

We separately consider the special case of puncturing closed manifolds.

\begin{proposition}\label{lem:puncture}
Let $\widehat{X}$ be a smooth, orientable, closed 4-manifold. Let $X:=\widehat{X}\setminus\mathring{B}^4$. Then $\rL(X)\le\L(\widehat{X})$.
\end{proposition}
\begin{proof}
Let $\widehat{\D}=(\widehat{\Sigma},\alpha,\beta,\gamma)$ be a $(g,k)$-trisection diagram of $\widehat{X}$. Deleting a small open disk in $\widehat{\Sigma}$ (away from $\alpha,\beta,\gamma$) yields a $(g,k;p,b)=(g,k;0,1)$-relative trisection diagram $\D=(\Sigma;\alpha,\beta,\gamma)$ of $X$.

Let $\widehat{\delta}$ be a closed loop in $HT(\widehat{\Sigma})$ valid with respect to $\widehat{\D}$. Each vertex in $\widehat{\delta}$ corresponds to a cut system of $g$ curves on $\widehat{\Sigma}$. Perturb each cut system slightly if necessary so that the curves always live in $\Sigma$. Each of these cut systems is then a cut system for $\Sigma$ (including $2p+b-1=0$ arcs). Then each vertex and edge of $\widehat{\delta}$ naturally corresponds to a vertex or edge in $HT_0(\Sigma)$, yielding a path $\delta$ in $HT_0(\Sigma)$ valid with respect to $\D$. Since $|\delta|=|\widehat{\delta}|$, we conclude
\begin{align*}
    \rL(\D)+3(g+p+b-1)-(k_1+k_2+k_3)&\le\L(\widehat{\D})+3g-(k_1+k_2+k_3)\\
    \rL(\D)&\le\L(\widehat{\D}).
\end{align*}

Since this holds for all $\widehat{\D}$, we conclude $\rL(X)\le\L(\widehat{X})$.
\end{proof}

\subsection{Interlude: Murasugi sum}\label{sec:murasugi}

In this subsection, we describe how to perform another common procedure via relative trisection diagrams: Murasugi sum. The language of cut complexes introduced in this section is helpful in carefully stating the procedure.

We remind the reader that an open book $\O$ of a 3-manifold $X\# Y=\mathring{X}\cup_S\mathring{Y}$ (with connected-sum sphere $S$ and $\mathring{X}:=X\setminus\mathring{B}^4, \mathring{Y}:=Y\setminus\mathring{B}^4$) is a Murasugi sum when $S$ intersects every page of $\O$ in a disk. For a fixed page $L$, we usually refer to this disk $P\subset L$ as a $2n$-gon, where $n$ is chosen by requiring alternating edges of $P$ to be in the boundary of $\overline{((L\setminus P)\cap \mathring{X})}$ while the other edges of $P$ are in the boundary of $\overline{((L\setminus P)\cap \mathring{Y})}$. By intersecting $\O$ with $\mathring{X}$ or $\mathring{Y}$ and capping boundaries with balls and disks, we obtain open books $\O_X$ and $\O_Y$ on $X$ and $Y$. We say that $\O$ is obtained by Murasugi-summing $\O_X$ and $\O_Y$ along 2n-gons $P_X$ and $P_Y$ in $\hat{L\cap X}$ and $\hat{L\cap Y}$, respectively. When $n=1$, we usually refer to this operation simply as ``connected-sum." When $n=2$, we generally refer to this operation as ``plumbing." We give an example of plumbing two pages together along a rectangle in Figure~\ref{F:murasugiex} which serves as an illustration of the following theorem.

\begin{theorem}\label{thm:murasugi}
There is an explicit algorithm to plumb two trisections together by Murasugi sum. That is, given relatively trisected 4-manifolds $X$ and $X'$, we may produce a relative trisection of $X\natural X'$ where the induced open book on $\boundary (X\natural X')$ may be chosen to be any desired Murasugi sum of the open books on $\boundary X,\boundary X'$.
\end{theorem}
\begin{proof}
Let $\mathcal{D}=(\Sigma; \alpha, \beta, \gamma)$ and $\mathcal{D}'=(\Sigma', \alpha', \beta', \gamma')$ be relative trisection diagrams corresponding to the $4$-manifolds with boundary $X$ and $X'$ respectively. Let $P\subset\Sigma_\alpha$ and $P'\subset \Sigma'_{\alpha'}$ be $2n$-gons, with alternating edges contained in $\boundary \Sigma$ or $\boundary \Sigma'$ (respectively). Choose arcs $a_1,\ldots, a_{n-1}$ in $P$ which are properly embedded in $\Sigma$ so that $(P\cap\boundary\Sigma)\cup(a_1\cup\cdots\cup a_{n-1})$ is connected. 

We will produce a relative trisection diagram $\overline{\D}=(\overline{\Sigma},\overline{\alpha},\overline{\beta},\overline{\gamma})$ of $X\natural X'$. The open book $\O_{\overline{\D}}$ will be the open book obtained by the Murasugi sum of $\O_\D$ and $\O_{\D'}$ along $P$ and $P'$, respectively.

Let $\delta$ be a path in $HT_p(\Sigma)$ represented by $\D$. Choose the first vertex of $\delta$ so that it corresponds to a cut system including $a_1,\ldots,a_n$ as arcs. (Such a path $\delta$ exists by Proposition \ref{prop:validpath}.) Let $\overline{\Sigma}$ be obtained from $\Sigma,\Sigma'$ by plumbing along $P,P'$. We continue to identify $\Sigma$ and $\Sigma'$ with subsets of $\overline{\Sigma}$. We set $\overline{\alpha}=\alpha\cup\alpha'$.

Let $v_3$ be the first vertex of $\delta$ in $\Gamma_{\beta}$ with closed curves equal to $\beta$. By following the edges in $\delta$, which each either preserve each $a_i$ or change it by a slide to some other arc, the arc $a_i$ corresponds to some arc $b_i$ in the cut system corresponding to $v_3$. The arc $b_i$ is disjoint from the $\beta$ curves. Choose a homeomorphism $f:\overline{\Sigma}\to\overline{\Sigma}$ so that $f(a_i)=b_i$ for $i=1,\ldots, n-1$ and $f|_{\Sigma'\setminus P'}=\id$. 
Set $\overline{\beta}=\beta\cup f(\beta')$. In words, we performed the monodromy algorithm to slide the $\beta$ curves off the plumbing region $P$ 
before adding the $\beta'$ curves.

Now similarly let $v_5$ be the first vertex of $\delta$ in $\Gamma_{\gamma}$ with closed curves equal to $\gamma$. By following the edges in $\delta$, which each either preserve each $b_i$ or change it by a slide to some other arc, the arc $b_i$ corresponds to some arc $c_i$ in the cut system corresponding to $v_5$. The arc $c_i$ is disjoint from the $\gamma$ curves. Choose a homeomorphism $g:\overline{\Sigma}\to\overline{\Sigma}$ so that $g(b_i)=c_i$ for $i=1,\ldots,n-1$ and $g|_{\Sigma'\setminus P'}=\id$. 
Set $\overline{\gamma}=\gamma\cup(g\circ f)(\gamma')$. In words, we performed the monodromy algorithm to slide the $\beta$ and $\gamma$ curves off the plumbing arc before adding the $\gamma'$ curves.

Now consider the resulting relative trisection diagram $\overline{D}:=(\overline{\Sigma};\overline{\alpha},\overline{\beta},\overline{\gamma})$ of a 4-manifold $Z^4$. By construction (using the monodromy algorithm of~\cite{cgp}; see Section~\ref{sec:monodromy}), a page of $\O_{\overline{\Sigma}}$ is obtained by plumbing $\Sigma$ and $\Sigma'$ along $P$ and $P'$, with the monodromy of $\O_{\overline{\Sigma}}$ the composition of the monodromies of $\O_\D$ and $\O_{\D'}$ (viewed as automorphisms of $\Sigma,\Sigma'\subset\overline{\Sigma}$). We conclude that $\partial Z^4\cong\partial X\#\partial X'$, and $\O_{\overline{\D}}$ is the Murasugi sum of $\O_\D$ and $\O_{\D'}$ along $P,P'$.

To obtain $Z^4$ from $\overline{\D}$, recall that we start with $\overline{\Sigma}\times D^2$, attach 2-handles according to $\overline{\alpha},\overline{\beta},\overline{\gamma}$ curves, and then some 3-handles using a standard model \cite{nickthesis}. The copy of $\D$ contained in $\overline{\D}$ thus corresponds to a subset of $Z^4$ that is diffeomorphic to $X$. Sliding $\beta',\gamma'$ curves in $\mathcal{D}$ (corresponding to 2-handle slides) would yield a copy of $D'$ (i.e.quotienting $\overline{\Sigma}$ by $\Sigma$ and deleting the $\alpha,\beta,\gamma$ curves yields $D'$), so we conclude that $Z^4\cong X\cup X'$, where $X$ and $X'$ are glued along a ball whose boundary is the connected sum sphere in $\partial Z^4=\partial X\#\partial X'$. Thus, $Z^4\cong X\natural X'$ as desired.

\end{proof}

\begin{figure}
\centering
{\subcaptionbox{Relative trisection diagrams $\D$ and $\D'$ of $\T$ and $\T'$. We shade $2n$-gons $P$ and $P'$ in $\Sigma_{\alpha}$, $\Sigma_{\alpha'}$ along which we will plumb $\T$ and $\T'$.}
{\labellist
	\pinlabel {$\D$} at 0 250
	\pinlabel {$\D'$} at 250 225
	\pinlabel {$P$} at 35 210
	\pinlabel {$P'$} at 320 90
    \endlabellist
\includegraphics[scale=.5]{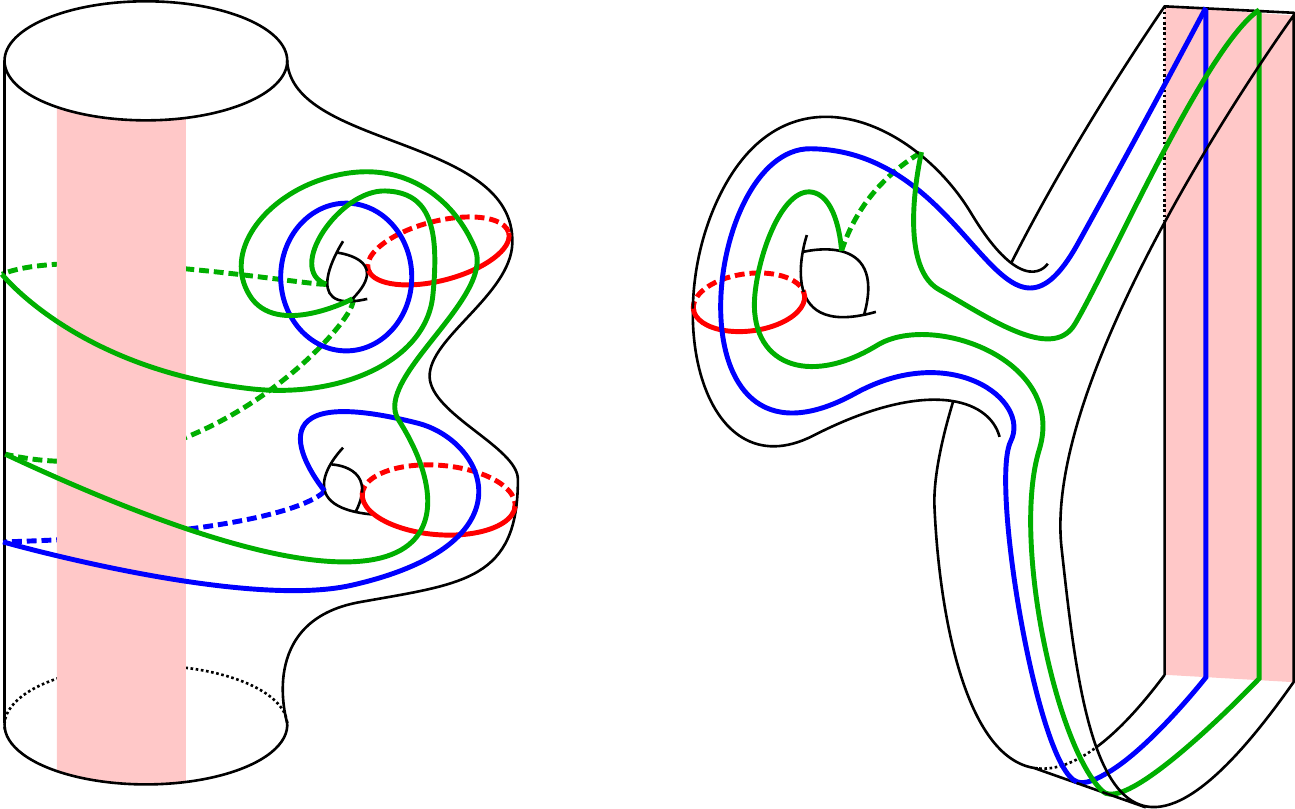}}}
\hspace{.5cm}
\subcaptionbox{The resulting relative trisection diagram $\D''$. This diagram describes the 4-manifold $X_{\T}\natural X_{\T'}$. The induced open book is obtained from $\O_\T$ and $\O_{\T'}$ by Murasugi sum identifying $P$ with $P'$.
}{
\includegraphics[scale=.5]{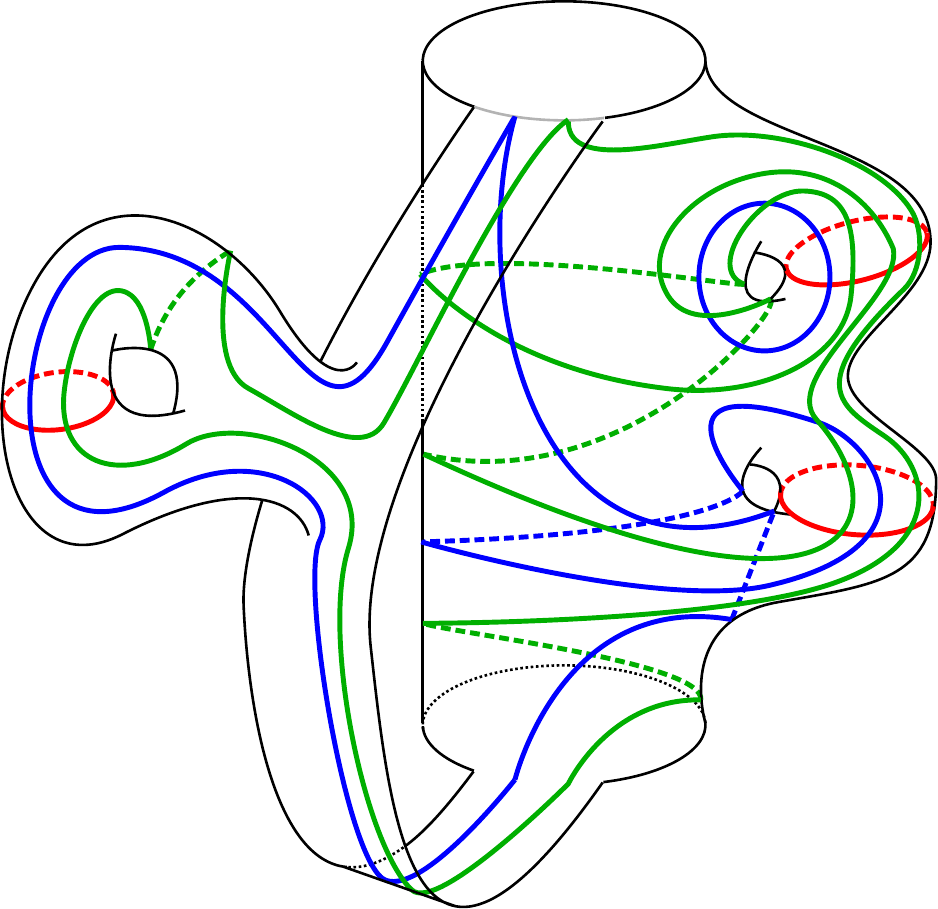}
}

\vspace{.5cm}

{\subcaptionbox{Here we illustrate the procedure described in Theorem~\ref{thm:murasugi}. We perform the monodromy algorithm on $\D$. We include the $\beta'$ and $\gamma'$ curves after making the arcs in $\Sigma$ disjoint from $\beta$ and $\gamma$. In general, we may have to slide $\beta$ and $\gamma$ first as well.}{
    \labellist
	\pinlabel {$a_1$} at 160 245
	\pinlabel {$b_1$} at 450 245
	\pinlabel {$c_1$} at 750 245
	\pinlabel {$\alpha'$} at -20 150
	\pinlabel {$f(\beta')$} at 330 85
	\pinlabel {$(g\circ f)(\gamma')$} at 650 220
    \endlabellist
    \includegraphics[width=120mm]{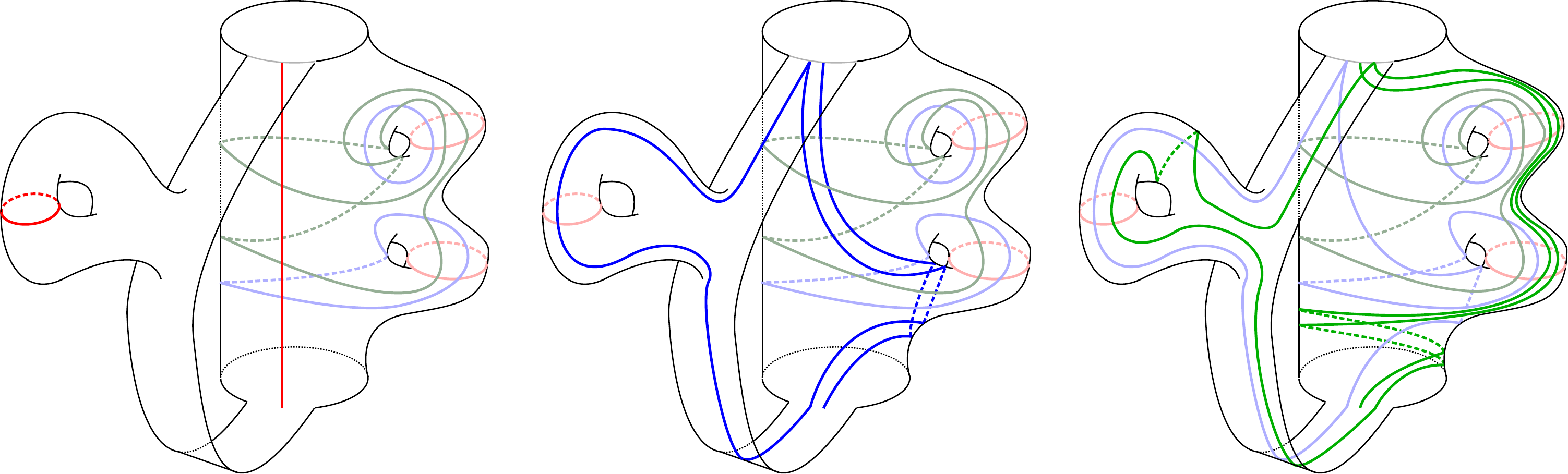}}}
\caption{Theorem~\ref{thm:murasugi} allows us to glue two trisected 4-manifolds and induce any desired Murasugi sum of the open books induced on their boundaries.} 
\label{F:murasugiex}
\end{figure}


\section{Topology of manifolds with small relative \texorpdfstring{\(\L\)}{L}-invariant}\label{sec3:smallL}

Now we prove that when the relative $\L$-invariant of a manifold is small, we may recover information about its topology.

\begin{lemma}
\label{L:0not1}
If $\T$ is a $(g,k;p,b)$-relative trisection diagram of a 4-manifold $X$ with $\rL^\partial(T)\le 1$, then $\partial X\cong \#_{2p+b-1} S^1\times S^2$.
\end{lemma}

\begin{proof}
 Let $\D_A=(\Sigma;\alpha,\beta,\gamma;A_\alpha, A_\beta, A_\gamma)$ be an arced relative trisection diagram for $\T$ and $\delta$ a path in $HT_p(X)$ valid with respect to $\D_A$, so that $(\D_A,\delta)$ achieves $\rL^\boundary(\T)$. Then $\delta$ includes at most one type 0$^{\boundary}$ edge. Let $A_{\overline{\alpha}}$ be the arcs corresponding to the final vertex of $\delta$, so $\alpha \cup A_\alpha$ and $\alpha\cup A_{\overline{\alpha}}$ are cut systems for $\Sigma$. Since $\delta$ includes at most one type $0^\partial$ edge, then $A_\alpha$ and $A_{\overline{\alpha}}$ are equal or differ by a sliding one arc over some disjoint closed curves.

If $A_\alpha$ and $A_{\overline{\alpha}}$ are equal, then from the monodromy algorithm we find that the relative trisection induced by $\D$ has trivial monodromy. Then $\partial X\cong \#_{2p+b-1} S^1\times S^2$ as desired.

On the other hand, suppose that $A_{\overline{\alpha}}$ is obtained from $A_{\alpha}$ by sliding one arc over disjoint closed curves $C_1,\ldots, C_n$ (which themselves are disjoint from $A_\alpha, A_{\overline{\alpha}}$). Since $A_{\alpha},A_{\overline{\alpha}}$ are disjoint from $\alpha$, $[C_1]+\cdots+[C_n]$ is in the span of the $\alpha$ curves in $H_1(\Sigma)$. Therefore, there exist slides over $\alpha$ curves taking $A_{\overline{\alpha}}$ to $A_{\alpha}$, so we again conclude from the monodromy algorithm that $\partial X$ admits an open book with trivial monodromy, so $\partial X\cong \#_{2p+b-1} S^1\times S^2$.

\end{proof}

The argument of Proposition~\ref{L:0not1} yields some insight into slightly larger values of $\rL^\partial(X)$.

\begin{theorem}\label{thm:boundarysmall}
Let $\T$ be a $(g,k;p,b)$-relative trisection of 4-manifold $X$ with $\rL^\boundary(\T)<2(2p+b-1)$ and $2p+b-1>0$. Then $\partial X\cong M\#(S^1\times S^2)$ for some 3-manifold $M$.
\end{theorem}
\begin{proof}
Let $\D_A=(\Sigma;\alpha,\beta,\gamma;A_\alpha,A_\beta,A_\gamma)$ be an arced relative trisection diagram for $\T$ with valid path $\delta$ in $HT_p(\Sigma)$ so that $\delta$ has less than $2(2p+b-1)$ type $0^\partial$ edges. Say that an arc $a$ ``changes to $a'$
during $\delta$" if there is a type 0$^\boundary$ edge in $\delta$ corresponding to replacing $a$ with another arc $a'$. Continue identifying $a$ with $a'$ and if a subsequent type 0$^\boundary$ edge in $\delta$ changes $a'$, say $a$ ``changes again" during $\delta$.

Since $A_\alpha$ contains $2p+b-1$ arcs, there is some arc $a_0$ in $A_\alpha$ which changes at most once in $\delta$.

{\bf{Case 1. $a_0$ never changes.}}
Then the monodromy of the open book induced by $\T$ on $\partial X$ fixes an essential arc (up to isotopy). We conclude $\partial X$ admits an $S^1\times S^2$ summand.



{\bf{Case 2. $a_0$ changes to arc $b_0$ in $A_\beta$.}} 

Then $b_0$ is obtained from $a_0$ by slides over $\alpha$ curves. We again conclude that up to isotopy, the monodromy of the induced open book on $\partial X$ fixes an essential arc.


{\bf{Case 3. $a_0$ changes to arc $c_0$ in $A_\gamma$ (which is not also in $A_\beta$).}}
Then $c_0$ is obtained from $a_0$ by slides over $\beta$ curves.

Slide $\alpha,a_0,c_0$ and separately slide $\beta$ to turn $\alpha,\beta$ into a standard pair $\alpha',\beta'$. This may involve sliding $c_0$ over $\alpha$ curves, but we abuse notation and still refer to the obtained arc as $c_0$. Now $c_0$ can be obtained from sliding $a_0$ over some collection of $\beta'$ curves. Moreover, $a_0$ and $c_0$ are disjoint from $\alpha'$. Since $\alpha',\beta'$ are standard, we conclude that the union of curves over which we slide $a_0$ are in the span of $\alpha'$ in $H_1(\Sigma)$, so $a_0$ and $c_0$ are isotopic in the $\alpha$-page of $\Sigma$. Thus, we conclude that up to isotopy, the monodromy of the induced open book on $\partial X$ fixes an essential arc.

{\bf{Case 3. $a_0$ changes to arc $c_0$ in $A_{\overline{\gamma}}$ (which is not also in $A_\gamma$).}}
By replacing $X$ with $-X$, exchanging $\beta$ and $\gamma$, and reversing the direction of $\gamma$, we find ourselves in Case 2.

\end{proof}

\begin{corollary}\label{cor:exsharp}
The relative trisection $\T$ in Example~\ref{example} has $\rL(\T)=\rL^\boundary(\T)=6$.
\end{corollary}

Theorem~\ref{thm:boundarysmall} allows us to construct 4-manifolds with arbitrarily large relative-$\L$ invariant.

\begin{corollary}\label{cor:arbbig}
There exist 4-manifolds $X_1,X_2,\ldots$ so that $\rL(X_n)\ge n$ for each $n\in\N$.
\end{corollary}
\begin{proof}
Let $X_0$ be the $4$-manifold obtained from $B^4$ by attaching a $2$-framed $2$-handle along an unknot in $S^3$, so $\boundary X_1\cong L(2,1)$. Let $X_n\cong\natural_n X_1$.

Suppose $X_n$ admits a $(g,k;p,b)$-relative trisection. Then $\boundary X_n$ admits an open book with genus-$p$ pages and $b$ binding components. This implies that $H_1(\boundary X_n;\Z)$ admits a presentation with $2p+b-1$ generators. Since $H_1(\boundary X_n;\Z)\cong\oplus_n\Z/2$, we then have $2p+b-1\ge n$. Since $\partial X_n\cong\#_n L(2,1)$ does not admit an $S^1\times S^2$ summand, Theorem~\ref{thm:boundarysmall} implies $\rL(X_n)\ge\rL^\boundary(X_n)\ge 2n$.
\end{proof}

Now we deal with the global topology of the 4-manifold.



\begin{theorem}\label{thm:b4}
Let $X^4$ be a rational homology ball with $\rL(X^4)=0$. Then $X^4$ is diffeomorphic to $B^4$.
\end{theorem}

\begin{proof}
Let $(\D_A,\delta)$ be a pair achieving $\rL(X^4)$. That is, find a $(g,k;p,b)$-relative trisection diagram $\D_A=(\Sigma;\alpha,\beta,\gamma;A_{\alpha},A_{\beta},A_{\gamma})$ for $X$ and path $\delta$ in $HT_p(\Sigma)$ valid with respect to $\D_A$ so that $|\delta|=3(g+p+b-1)-(k_1+k_2+k_3)$. Then $\delta$ consists only of type 1 edges. 
By Lemma~\ref{L:0not1}, we have $\boundary X^4\cong\#_{2p+b-1}(S^1\times S^2)$. Since $\boundary X^4$ is a rational homology 3-sphere, we must have $2p+b-1=0$, so $p=0$ and $b=1$. Thus, $A_{\alpha},A_{\beta},A_{\gamma}$ are all empty. We now write $\D=(\Sigma;\alpha,\beta,\gamma)$.

We wish to proceed by induction; to make the proof easier we weaken our knowledge of $\delta$. Let $v_\alpha, v_\beta, v_\gamma$ be vertices in $HT_0(S)$ corresponding to the $\alpha, \beta, \gamma$ curves (respectively) in $\D$, so these vertices appear in order in $\delta$ (with $v_{\alpha}$ as the first vertex of $\delta$). Restrict $\delta$ to the subinterval of $\delta$ from $v_{\alpha}$ to $v_{\gamma}$. From now on, the only properties we will assume of $\delta$ are that $\delta$ is a path in $HT_p(\Sigma)$ from $v_{\alpha}$ to $v_{\beta}$ to $v_{\gamma}$ consisting of exactly $2(g+p+b-1)-(k_1+k_2)$ type 1 edges.



 The segment of $\delta$ between $v_{\alpha}$ and $v_{\beta}$ consists of $g-k_1$ type 1 edges. If $v_1,\ldots, v_{g-k_1+1}$ are the vertices of $\delta$ (in order) between $v_{\alpha},v_{\beta}$, with $v_1=v_{\alpha},v_{g-k_1+1}=v_{\beta}$, then the cut system corresponding to $v_{i+1}$ differs from that of $v_i$ by replacing one curve in $\alpha$ with a distinct curve in $\beta$. Write the disjoint curves in $\alpha$ as $\alpha_1,\ldots,\alpha_g$ and those in $\beta$ as $\beta_1,\ldots,\beta_g$, where the edge from $v_i$ to $v_{i+1}$ corresponds to replacing $\alpha_i$ with $\beta_i$ ($i\le g-k_1$). Note $\alpha_j=\beta_j$ for $g-k_1<j\le g$. Write the curves in $\gamma$ as $\gamma_1,\ldots,\gamma_g$.


Suppose $k_1>0$, so $\alpha_g$ and $\beta_g$ are parallel. 
There cannot be a $\gamma_j$ curve parallel to $\alpha_g$ and $\beta_g$, or else this would yield a connected-sum factor of $S^1\times S^3$ in $X^4$, violating $H_1(X^4;\Q)=0.$ Therefore, some edge in $\delta$ between $v_{\beta}$ and $v_{\gamma}$ corresponds to replacing $\beta_g$ with some $\gamma_j$. Choose the labelings of the $\gamma$ curves so this curve is $\gamma_g$. See Figure~\ref{fig:thmb4}. 
\begin{figure}
\labellist
	\pinlabel {Slides} at 257 95
	\pinlabel {$\D$} at 10 120

	\pinlabel {Destabilize} at 545 95
	
	\pinlabel {$\D'$} at 600 120
	
	\pinlabel {$\alpha_g$} at 60 40
	\pinlabel {$\beta_g$} at 148 40
	\pinlabel {$\gamma_g$} at 50 80
	\pinlabel {$\gamma_r$} at -10 20
	\pinlabel {$\beta_j$} at 170 100
	\pinlabel {$\alpha_i$} at 140 137
	\pinlabel {$\alpha'_i$} at 670 50
	\pinlabel {$\beta'_j$} at 760 50
	\pinlabel {$\gamma'_r$} at 620 40
\endlabellist
\includegraphics[width=120mm]{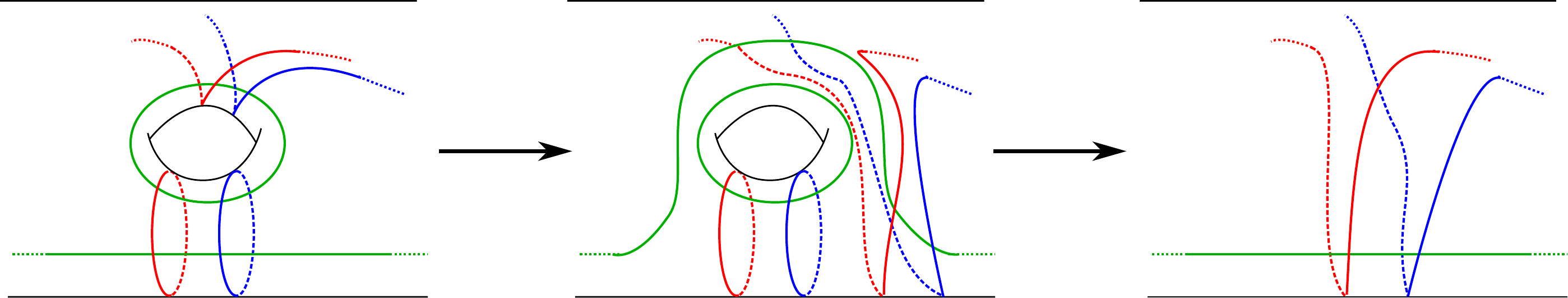}

\caption{Destabilizing a relative trisection diagram as in the proof of Theorem~\ref{thm:b4}.}\label{fig:thmb4}
\end{figure}
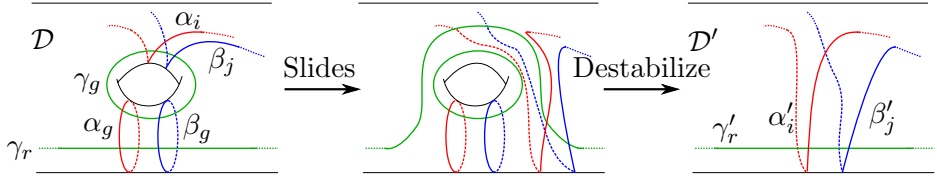


Now after slides to remove intersections of other $\alpha, \beta$ curves with $\gamma_g$ and intersections of other $\gamma$ curves with $\alpha_g$ and $\beta_1$, $\D$ can be destabilized along $\nu(\alpha_g\cup\gamma_g)$ to obtain a $(g-1,k',0,1)$-relative trisection diagram $\D'=(\Sigma',\alpha',\beta',\gamma')$ (schematically pictured in Figure~\ref{fig:thmb4}).

\begin{claim}
Let $E$ be an edge in $\delta$ between $v_{\alpha},v_{\gamma}$ with endpoints $w_1,w_2$. Assume $E$ does not correspond to replacing $\beta_1$ with $\gamma_1$. Let $w'_1$ and $w'_2$ be the corresponding vertices in $HT_0(\Sigma')$ (i.e. corresponding to the cut systems by replacing $\alpha,\beta,\gamma$ curves by $\alpha',\beta',\gamma'$ and deleting $\alpha_1,\beta_1,\gamma_1$). Then there is a type 1 edge between $w'_1$ and $w'_2$.
\end{claim}
\begin{proof}
For $i>1$, we write $\alpha_i'$ to denote the $\alpha$ curve in $\D'$ corresponding to $\alpha_i$ in $\D$ (and do similar for $\beta$ and $\gamma$ curves).

If $\alpha_i$ intersects $\gamma_1$ and $\gamma_j$ intersects $\alpha_1$, then $|\alpha'_i\cap\gamma'_j|>|\alpha_i\cap\gamma_j|$. Similarly, if $\beta_l$ intersects $\gamma_1$ then $|\beta'_l\cap\gamma'_j|>|\beta_l\cap\gamma_j|$. However, we always have $|\alpha'_p\cap\beta'_q|=|\alpha_p\cap\beta_q|$ for any $p,q>1.$


{\bf{Case 1.}}
Suppose $E$ corresponds to replacing $\alpha_i$ with $\beta_i$, $1<i\le g-k_1$. The vertices $w_1$ and $w_2$ each correspond to some combination of $\alpha$ and $\beta$ curves in $\D$. Then $w_1$ and $w_2$ descend to vertices $w'_1, w'_2$ of $HT_0(\Sigma')$ which correspond to  combinations of $\alpha'$ and $\beta'$ curves in $\D'$. The cut system for $w'_1$ is obtained from that of $w'_2$ by deleting $\alpha'_i$ and replacing it with $\beta'_i$. Since $|\alpha'_i\cap\beta'_i|=|\alpha_i\cap\beta_i|=1$, $w'_1$ and $w'_2$ are connected by a type 1 edge. 

{\bf{Case 2.}}
Suppose $E$ corresponds to replacing $\beta_i$ with $\gamma_j$ for some $i,j>1$.

To make notation easier in this section, we write $E_{\alpha_r}$ or $E_{\beta_r}$ to indicate the edge corresponding to replacing $\alpha_r$ or $\beta_r$ with a different curve. Then $E=E_{\beta_i}$.

\begin{claim}
$|\beta'_i\cap\gamma'_j|=|\beta_i\cap\gamma_j|=1$.
\end{claim}
\begin{proof}
Assume otherwise. Then $\beta_i\cap\gamma_1\neq\emptyset$ and $\gamma_j\cap\beta_1\neq\emptyset$.

Since $\beta_i$ intersects $\gamma_1$, in $\delta$ $E$ must occur before $E_{\beta_1}$ (since $\beta_1$ is replaced by $\gamma_1$, which would intersect $\beta_i$ if $\beta_i$ had not yet been replaced). 
Similarly, since $\beta_1$ intersects $\gamma_j$, in $\delta$ $E_{\beta_1}$ must occur before $E_{\beta_i}=E$ (since $\beta_i$ is replaced by $\gamma_j$, which would intersect $\beta_1$ if $\beta_1$ had not yet been replaced). 
This yields a contradiction.
\end{proof}
Since the cut systems corresponding to $w'_1$ and $w'_2$ differ by replacing $\beta'_i$ with $\gamma'_j$, we conclude there is a type 1 edge between $w'_1$ and $w'_2$.%


\end{proof}
Thus, there is still a path in $HT_0(\Sigma')$ from $v'_\alpha$ to $v'_\beta$ to $v'_\gamma$ consisting of $|\delta|-1=2(g+p+b-1)-(k_1+k_2)-1=2((g-1)+p+b-1)-((k_1-1)+k_2)-1$ type 1 edges. Recall that $(\Sigma',\alpha',\beta'\gamma')$ is a $((g-1),(k_1-1,k_2,k_3);p,b)$-relative trisection (and $p=0,b=1$). Thus we may proceed inductively until finding a $(\overline{g},(0,\overline{k}_2,\overline{k}_3);0,1)$-relative trisection diagram $\overline{D}=(\overline{\Sigma},\overline{\alpha},\overline{\beta},\overline{\gamma})$ for $X$ so there is a path $\overline{\delta}$ in $HT_0(\overline{\Sigma})$ from $v_{\overline{\alpha}}$ to $v_{\overline{\gamma}}$ to $v_{\overline{\gamma}}$ consisting of $2\overline{g}-k_2$ type 1 edges. The curves $\overline{\alpha}$ and $\overline{\beta}$ are algebraically dual.

By repeating the argument taking the reverse of $\delta$ to take the role of $\delta$ (exchanging the roles of $\gamma$ and $\alpha$), we may also take $\overline{k_2}=0$, so the $\overline{\beta}$ and $\overline{\gamma}$ curves are dual.

This relative trisection description yields a handle decomposition of $X^4$ into a $0$--handle, zero $1$--handles, $\overline{g}$ $2$--handles, and $\overline{k}_3$ $3$--handles.(See~\cite{price} for a description of going from relative trisections to handle structures and vice versa). Since $H_3(X^4;\Q)=0,$ $\overline{k}_3=g$. Therefore, the $\overline{\gamma}$ and $\overline{\alpha}$ curves on $\overline{\Sigma}$ define the same genus--$\overline{g}$ handlebody.


Slide only the $\overline{\alpha}$ curves until they agree with the $\overline{\gamma}$ curves (here using the fact that $\overline{\alpha}$ and $\overline{\gamma}$ define the same handlebody, so that we need not slide $\overline{\gamma}$ as well. Although pairs of curves in a relative trisection diagram are standard, we generally expect to have to slide both sets of curves to standardize). Then we may slide the $\beta$ curves and simultaneously slide the $\alpha$ and $\gamma$ curves (so that the $\alpha$ and $\gamma$ curves always coincide) until $\alpha,\beta$ intersect standardly (and $\beta,\gamma$ also intersect standardly). After all of these slides, $\overline{D}$ is a connected sum of the $(0,0;0,1)$-relative trisection diagram for $B^4$( i.e. $(D^2,\emptyset,\emptyset,\emptyset)$) and genus--$1$ trisection diagrams for $S^4$, so we conclude that $X^4\cong B^4.$

\end{proof}

\begin{corollary}
Let $\widehat{X}$ be a rational homology 4-sphere. Then $\L(\widehat{X})=0$ if and only if $\widehat{X}\cong S^4$.
\end{corollary}

\begin{proof}
Let $X:=\widehat{X}\setminus\mathring{B}^4$. By Proposition~\ref{lem:puncture}, $\rL(X)=0$. By Theorem~\ref{thm:b4}, $X\cong B^4$. Therefore, $\widehat{X}\cong S^4$. 
\end{proof}



\section{Bounds and the arc complex}\label{sec4:arccomplex}

As we are only concerned with surfaces with boundary (rather than surfaces with punctures), we will restrict our discussion to this setting. Given a surface automorphism $\phi:P\to P$, let $\M_{\phi}$ denote the 3-manifold with open book $\O_{\phi}$ induced by $\phi$.

\begin{definition}
The \emph{essential arc complex} $\mathcal{A}_e(P)$ of a surface with boundary $P$ is a simplicial complex such that
	\begin{itemize}
		\item[i)] each vertex in the $0$--skeleton $\mathcal{A}_e^0(P)$ corresponds to an essential, properly embedded arc,
		\item[ii)] the collection of vertices $\{v_0, \ldots, v_n\}$ defines an $n$--cell if the arcs $a_i$ and $a_j$ corresponding to $v_i$ and $v_j$ are disjoint for every $i$ and $j$.
	\end{itemize}
\end{definition}

Note that edges in the essential arc complex behave differently than edges in the cut complex. Given two disjoint arcs $a_1,a_2$ in $\Sigma$, there is an edge in $\mathcal{A}_e^0(P)$ connecting vertices corresponding to those edges. However, if $C_1,C_2$ are cut systems for $\Sigma$ (each consisting of $g$ closed curves and $2p+b-1$ arcs), then if there is an edge in the cut complex between vertices corresponding to $C_1$ and $C_2$, it must be the case that $\boundary C_1=\boundary C_2$. In some sense, the edges in the essential arc complex are more flexible, as we connect vertices corresponding to arcs with very different boundaries (perhaps even meeting distinct components of $\boundary \Sigma$).

Let $d_e: \mathcal{A}_e^0(P) \times \mathcal{A}_e^0(P) \rightarrow \mathbb{N}\cup\{0\}$ count the minimal number of edges between any two essential, properly embedded arcs in $P$. The \emph{displacement distance} of an orientation preserving diffeomorphism $\phi:P\rightarrow P$ which fixes the (non-empty) boundary pointwise is
	$$d_e(\phi)=\min\{d_e(a, \phi(a))|a\in\mathcal{A}_e^0\}.$$
\begin{theorem}[Etnyre, Li]\label{T:displ}
If $d_e(\phi)=0,$ then $\M_{\phi}$ decomposes as $Y^3\#S^1\times S^2$. If $d_e(\phi)=1$, then $\O_{\phi}$ admits a positive or negative Hopf destabilization.
\end{theorem}

Like the closed $\L$-invariant, the relative $\L$-invariant is typically difficult to compute in practice. Nevertheless, we may still bound the invariant in terms of other, more calculable, invariants. Our goal will be to obtain a bound on our invariant in terms of the complexity of the monodromy of the boundary. Here, following Etnyre and Li~\cite{EtnyreLi}, our notion of complexity will be the translation distance of the monodromy in the arc complex. We begin with a bound on distances in this complex in terms of intersection numbers.

\begin{proposition}\label{prop:homessential}
Let $a$ and $b$ be homologically essential arcs on $\Sigma$.  Then $d_e(a,b) \leq |a \cap b| +1$.
\end{proposition}

\begin{proof}
\label{prop:intersectionBound}
We proceed by induction. If $a$ and $b$ are disjoint, then the statement clearly holds. Suppose that $|a \cap b| = n$, and isotope $a$ and $b$ until they intersect minimally. Let $b_1$ be the subarc of $b$ whose first endpoint lies on $\boundary \Sigma$ and whose second endpoint is the first intersection with $a$. Let $b_1'$  and $b_1''$ be the two arcs formed by following $b_1$ to the intersection point and then following $a$ to its endpoints in either direction (See Figure~\ref{fig:buildingPath}). Note that $\boundary \nu(b_1 \cup a) = a \cup b_1'\cup b_1''$ so that $b_1'\cup b_1''$ is homologous to $a$. Therefore at least one of $b_1'$ or $b_1''$ is homologically essential. Without loss of generality, suppose $b_1'$ is homologically essential.

Now  $b_1'$ is disjoint from $a$. Also, $|b_1'\cap b|=|a \cap b|-1$. Then by the inductive hypothesis, $d_e(b_1', b) \leq (n-1)+1 = n$. Then $d_e(a,b) \leq d_e(a,b_1') + d_e(b_1', b) \leq  n+1$.
\end{proof}

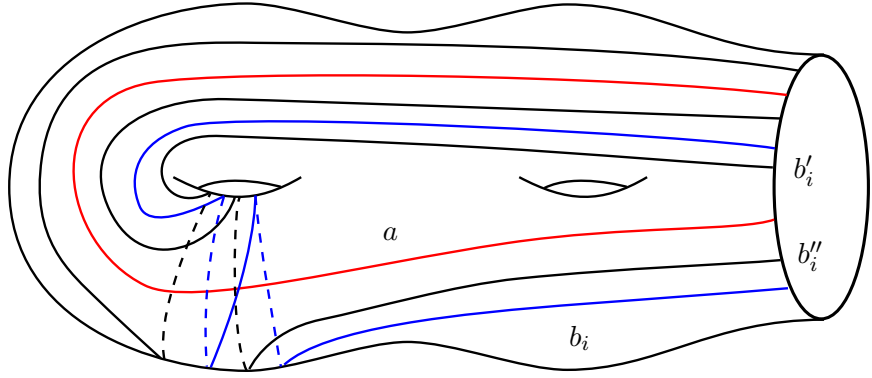
\begin{figure}
    \centering
    \labellist
        \pinlabel {$a$} at 125 50
        \pinlabel {$b_i$} at 185 17
        \pinlabel {$b_i'$} at 256 70
        \pinlabel {$b_i''$} at 258 43
    \endlabellist
    \includegraphics[scale=1.2]{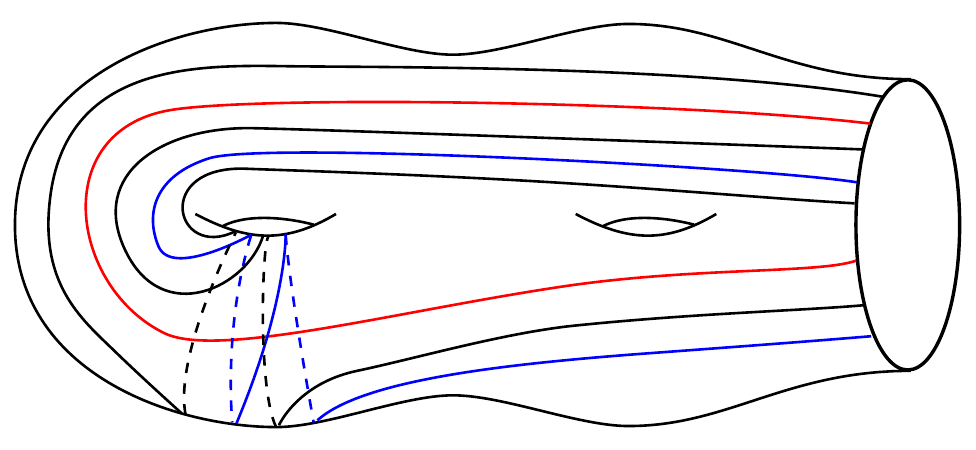}
    \caption{The arcs $a$ and $b$ are disjoint from the arcs $b_1'$ and $b_1''$.}
    \label{fig:buildingPath}
\end{figure}

In an arced relative trisection diagram, the arcs corresponding to the page $\Sigma_{\beta}$ will typically intersect the $\alpha$ curves. This makes it unclear how they fit onto the page $\Sigma_{\alpha}$, where the monodromy of the page is typically defined. To understand all of the arcs on a single page, we look towards subsurface projection.

\begin{definition}
Let $a$ be a properly embedded arc on a surface $\Sigma$, and let $\Sigma'$ be an essential subsurface of $\Sigma$. Isotope $a$ so that $|a \cap \boundary \Sigma'|$ is minimized. Then if  $a \cap \Sigma' \neq \emptyset$, we say that $a$ \textbf{cuts} $\Sigma'$. If $a$ cuts $\Sigma'$, then we define the \textbf{subsurface projection} of $a$ onto $\Sigma'$ to be:
\begin{enumerate}
    \item $a$, if $a \subset \Sigma'$.
    \item Any component of the boundary of a regular neighbourhood of $a \cap \partial \Sigma'$, if  $a \cap \partial \Sigma' \neq \emptyset$.
\end{enumerate}
\end{definition}

Disjoint arcs on a surface will not necessarily project to disjoint arcs on a subsurface. The issue arises when both arcs are projected onto the same boundary component. In this case one can quickly verify the following lemma.

\begin{lemma}
\label{lem:projectionIntersection}
If $a$ and $b$ are disjoint arcs on $\Sigma$ and $a'$ and $b'$ are their projections onto $\Sigma'$, then $|a' \cap b'| \leq 2$.
\end{lemma}

By combining Lemma~\ref{lem:projectionIntersection} and Proposition~\ref{prop:intersectionBound}, we obtain the following.

\begin{corollary}
\label{cor:projectedPath}
Let $(a_0,...,a_n)$ be a path in $A_e(\Sigma)$ where every arc cuts an essential subsurface $\Sigma'$. Let $a_0'$ and $a_n'$ be the subsurface projections of $a_0$ and $a_n$, respectively. Then $d_e{A_e(\Sigma'))}(a_0',a_n') \leq 3n$.
\end{corollary}

We are now ready to prove the main proposition of this section.

\begin{proposition}
Let $\T$ be a $(g,k;p,b)-$relative trisection $X$ such that $g \geq 2$ or $g=1$ and $b \geq 2$. Let $\phi:\Sigma_\alpha\to\Sigma_\alpha$ be the monodromy of the open book decomposition of $\partial X$ induced by $\T$. Then $\rL^{\boundary}(\T) \geq \frac{1}{3}(2p+b-1)d_e(\phi)$.
\end{proposition}

\begin{proof}
Let $\D_A=(\Sigma;\alpha,\beta,\gamma;A_\alpha,A_\beta,A_\gamma)$ be a relative trisection diagram for $\T$ and $\delta$ a path in $HT_p(\Sigma)$ valid with respect to $\D_A$. We will show that each arc in $A_{\alpha}$ must be replaced at least $\frac{1}{3}d_e(\phi)$ times during the course of $\delta$ (here we are a little relaxed in terminology. Implicitly, we identify the arc systems corresponding to either side of a type 0$^{\boundary}$ edge in $\delta$, so that we may think of one arc being replaced several times when following $\delta$ from start to end). Using the fact that there are $(2p+b-1)$ arcs, we will then obtain the result as stated.

Let $a_0$ be an arc in $A_\alpha$. Find the first type 0$^{\boundary}$ edge in $\delta$ corresponding to replacing $a_0$ with an arc $a_1$, if such an edge exists. Say the next type 0$^{\boundary}$ edge corresponding to replacing $a_1$ changes $a_1$ to $a_2$. Repeat until finding an edge $a_n$ which is never changed in cut systems corresponding to type 0$^{\boundary}$ edges once it appears. Since $\delta$ is valid with respect to $\D_A$, $a_n=\phi(a_0)$.

By definition of a type 0$^{\boundary}$ edge, $a_i$ and $a_{i+1}$ must be disjoint in their interiors and can be pushed off each other by isotoping the boundary of $a_{i+1}$ slightly. Therefore, there is a path in $A_e(\Sigma)$ whose vertices correspond (in order) to $a_0,a_1,\ldots, a_n$ (after pushing off slightly). 


Let $\Sigma'$ be the subsurface of $\Sigma$ obtained by deleting a small annular neighborhood of each $\alpha$ curve. 
Since each $a_i$ is a properly embedded arc in $\Sigma$ and $\boundary\Sigma\subset\Sigma'$, we must have $a_i\cap\Sigma'\neq\emptyset$. Let $a'_i$ be a subsurface projection of $a_i$ to $\Sigma'_i$. We have $a'_0=a_0$ and $a'_n=a_n$, since $a_0$ and $a_n$ are both contained in $\Sigma_\alpha$. 

By the hypothesis, $\Sigma'$ is not a pair of pants. 
Then by Corollary~\ref{cor:projectedPath}, $d_{A_e(\Sigma')}(a_0',a_n') \leq 3n$. Now $\Sigma_\alpha$ is obtained from $\Sigma'$ by capping some boundary components with disks, so  $d_e(a_0,a_n)\le 3n$. Since $a_n=\phi(a_0)$, 
we conclude $d_e(\phi) \leq 3n$. That is, $n \geq\frac{1}{3}d_e(\phi)$.

\end{proof}

Note that pseudo-Anosov maps have positive stable translation distance in the arc complex (see, for example,~\cite{FuSh} and~\cite{BaSt})). In particular, this implies that for any $n\in\N$ and any surface $P$ with negative Euler characteristic, there exists a pseudo-Anosov map $\phi:P\to P$ with translation distance greater than $k$. We may use this fact, together with the previous corollary, to construct relative trisections with large relative $\L$-invariant. More precisely, we start with a map pseudo-Anosov map $\phi$ on a genus--$p$ surface whose translation distance is greater than some given $n$. We factor $\phi$ into a collection of $m$ Dehn twists. This decomposition allows us to  construct a Lefschetz fibration with $m$ Lefschetz singularities, such that the monodromy of the fibration is the original map $\phi$. A construction in~\cite{nickthesis} shows how to turn this fibration into a $(p+m,0;p,b)$-relative trisection whose induced open book decomposition on the boundary has monodromy $\phi$ (see also Subsection~\ref{sec:stab}). We summarize this discussion in the following corollary.

\begin{corollary}\label{cor:arbbigtri}
For fixed $p,b$ with $2p+b-1>1$ and for all $n\in\N$, there exists a $(g_n,k_n;p,b)$-relative trisection $\T_n$ so that $\rL^\boundary(\T_n)>n$.
\end{corollary}

Compare this statement to Corollary~\ref{cor:arbbig}, in which we produced manifolds $X_1,X_2,\ldots$ with arbitrarily large relative $\L$-invariant. In those examples, the genus or boundary number of a page of an open book on $X_n$ was forced to grow large with $n$ due to the dimension of $H_1(\boundary X_n;\Z)$. In contrast, the induced open books in Corollary~\ref{cor:arbbigtri} live in 3-manifolds $Y_1,Y_2,\ldots,$ where the dimension of H$_1(\boundary Y_n;\Z)$ is bounded above uniformly.

In the examples of Corollary~\ref{cor:arbbigtri}, one must continuously iterate the original pseudo-Anosov map in order to increase the complexity of the relative trisection. This has the side effect of increasing the genus of the decomposition. In light of this, we pose the following natural question.
\begin{question}
Fix a positive integer $g$. Do there exist genus--$g$ relative trisections with arbitrarily large relative $\L$-invariant? 
\end{question}

\bibliographystyle{alpha}

\bibliography{references}

\end{document}